\documentclass[10pt,a4paper]{article}

\usepackage[utf8]{inputenc}
\usepackage{fullpage}
\usepackage{comment}
\usepackage[dvipsnames]{xcolor}
\usepackage{amsmath}
\usepackage{amssymb}
\usepackage{amsfonts}
\usepackage{amsthm}
\usepackage{graphicx}
\usepackage{amscd}
\usepackage{bbm}

\usepackage{tikz}
\usepackage{tikz-cd}
\usepackage{tikz-3dplot} \usetikzlibrary{calc}

\usepackage{enumitem}
\setitemize[1]{left=0.0cm}

\usepackage{bm} \usepackage[normalem]{ulem} 

\iftrue

\fi

\usepackage[colorlinks=true]{hyperref}         \hypersetup{
    colorlinks=true,
    linkcolor=Maroon,    citecolor=red,    urlcolor=blue,     urlbordercolor={1 1 1},
    pdfborder={0.5 0.5 0}  }

\newcommand{\recindex}{\jmath}

\newtheorem{theorem}{Theorem}[section]
\newtheorem{lemma}[theorem]{Lemma}

\newtheorem{remark}[theorem]{Remark}
\newtheorem{example}[theorem]{Example}
\newtheorem{proposition}[theorem]{Proposition}

\makeatletter
\newcommand\suchthat{\@ifstar
  {\mathrel{}\middle|\mathrel{}}
  {\mid}}
\makeatother

\newcommand{\Ceins}[1]{C_{1,#1}}
\newcommand{\Czwei}[1]{C_{2,#1}}

\newcommand{\Cfive}[2]{C_{5,#1,#2}}
\newcommand{\Cfiveprime}[2]{C'_{5,#1,#2}}
\newcommand{\Cdetfive}[2]{C^{\det}_{5,#1,#2}}
\newcommand{\Cfivetrafo}[4]{C_{5,#1,#2,#3,#4}}

\newcommand{\Csechs}[2]{C_{6,#1,#2}}
\newcommand{\Csechsprime}[2]{C'_{6,#1,#2}}
\newcommand{\Cdetsechs}[2]{C^{\det}_{6,#1,#2}}
\newcommand{\Csechstrafo}[4]{C_{6,#1,#2,#3,#4}}

\newcommand{\Csieb}[2]{C_{7,#1,#2}}
\newcommand{\Csiebprime}[2]{C'_{7,#1,#2}}
\newcommand{\Cdetsieb}[2]{C^{\det}_{7,#1,#2}}
\newcommand{\Csiebtrafo}[4]{C_{7,#1,#2,#3,#4}}

\newcommand{\Cacht}[2]{C_{8,#1,#2}}
\newcommand{\Cachtprime}[2]{C'_{8,#1,#2}}
\newcommand{\Cdetacht}[2]{C^{\det}_{8,#1,#2}}
\newcommand{\Cachttrafo}[4]{C_{8,#1,#2,#3,#4}}

\newcommand{\fooL}{L}

\newcommand{\PF}{{\rm PF}}

\newcommand{\minus}{-}

\newcommand{\aspectratio}{\kappa_{\rm A}}
\newcommand{\algebraicshapemeasure}{\kappa_{\rm M}}

\newcommand{\volumeratio}{C_\rho}
\newcommand{\diameterratio}{C_\theta}
\newcommand{\reflectionestimate}{C_\xi}

\newcommand{\Xieins}{\Xi_{1}}
\newcommand{\Xizwei}{\Xi_{2}}

\newcommand{\vecnorm}[1]{\| #1 \|}
\newcommand{\matnorm}[1]{\| #1 \|_{2}}

\newcommand{\Poinc}{\calP}
\newcommand{\Bogov}{\calB}

\newcommand{\mollifier}{\frakm}

\newcommand{\eps}{\epsilon}

\newcommand{\Id}{{\rm Id}}

\newcommand{\diff}{\mathop{}\!\mathrm{d}}

\newcommand{\mia}{k}
\newcommand{\jeza}{\ell}

\newcommand{\UPatch}{U}

\renewcommand{\restriction}{\mathord{|}}

\DeclareMathOperator{\adj}{adj}
\newcommand{\inv}{{-1}}

\newcommand{\Jacobian}{{\mathbf{J}}}

\newcommand{\signum}{\operatorname{sgn}}

\newcommand{\grad}{\operatorname{grad}}

\newcommand{\curl}{\operatorname{curl}}

\newcommand{\rot}{\operatorname{rot}}
\newcommand{\divergence}{\operatorname{div}}

\newcommand{\cartan}{d}
\newcommand{\cartanx}{dx}

\DeclareMathOperator*{\argmin}{{\rm argmin}}

\newcommand{\carapace}{{\partial \rm st}}

\newcommand{\supp}{\operatorname{supp}}

\newcommand{\vol}{\operatorname{vol}}

\newcommand{\diam}{{\delta}}
\newcommand{\height}{h}

\newcommand{\patch}{\operatorname{st}}

\newcommand{\trace}{\operatorname{tr}}

\newcommand{\subsimplex}{\calS^{\downarrow}}

\newcommand{\Vertices}{\calV}

\newcommand{\Faces}{\calF}
\newcommand{\Ball}{\calB}

\newcommand{\underlying}[1]{\left| #1 \right|}

\newcommand{\Lebesgue}{L}

\newcommand{\Alt}{\Lambda}

\newcommand{\bbR}{{\mathbb R}}

\newcommand{\bbZ}{{\mathbb Z}}

\newcommand{\bfL}{\bm{L}}

\newcommand{\bfW}{\bm{W}}

\newcommand{\bff}{\bm{f}}
\newcommand{\bfg}{\bm{g}}

\newcommand{\bfu}{\bm{u}}
\newcommand{\bfv}{\bm{v}}
\newcommand{\bfw}{\bm{w}}

\newcommand{\calB}{{\mathcal B}}
\newcommand{\calC}{{\mathcal C}}

\newcommand{\calF}{{\mathcal F}}
\newcommand{\calG}{{\mathcal G}}

\newcommand{\calK}{{\mathcal K}}

\newcommand{\calM}{{\mathcal M}}

\newcommand{\calP}{{\mathcal P}}

\newcommand{\calR}{{\mathcal R}}
\newcommand{\calS}{{\mathcal S}}
\newcommand{\calT}{{\mathcal T}}

\newcommand{\calV}{{\mathcal V}}

\newcommand{\frakm}{{\mathfrak m}}

\newcommand{\disk}{ball}

\title{Computable Poincar\'e--Friedrichs constants for the $L^{p}$ de~Rham complex over convex domains and domains with shellable triangulations}
\author{
    Th\'eophile Chaumont-Frelet\thanks{Inria Univ. Lille and Laboratoire Paul Painlev\'e, 59655 Villeneuve-d'Ascq, France, \texttt{theophile.chaumont@inria.fr}} \and
    Martin Werner Licht\thanks{Corresponding author. \'Ecole Polytechnique F\'ed\'erale de Lausanne, CH-1015 Lausanne, Switzerland, \texttt{martin.licht@epfl.ch}} \and
    Martin Vohral\'ik\thanks{Inria, 48 rue Barrault, 75647 Paris, France \& CERMICS, Ecole des Ponts, 77455 Marne-la-Vall\'ee, France, \texttt{martin.vohralik@inria.fr}}
}
\date{}

\begin{document}

\maketitle
\begin{abstract}
    We construct potentials for the exterior derivative, in particular, for the gradient, the curl, and the divergence operators, over domains with shellable triangulations.
    Notably, the class of shellable triangulations includes local patches (stars) in two or three dimensions.
    The operator norms of our potentials satisfy explicitly computable bounds that depend only on the geometry.
    We thus compute upper bounds for constants in Poincar\'e--Friedrichs inequalities and lower bounds for the eigenvalues of vector Laplacians.
    As an additional result with independent standing, we establish Poincar\'e--Friedrichs inequalities with computable constants for the $L^{p}$ de~Rham complex over bounded convex domains, derived as explicit operator norms of regularized Poincar\'e and Bogovski\u{\i} potential operators.
    We express all our main results in the calculus of differential forms and treat the gradient, curl, and divergence operators as instances of the exterior derivative.
    Computational examples illustrate the theoretical findings.
\end{abstract}

\tableofcontents

\section{Introduction}\label{section:intro}

Potentials for the differential operators of vector calculus and exterior calculus are of fundamental importance.
The operator norms of these potentials are upper bounds for the \emph{Poincar\'e--Friedrichs constants},
which quantify the fundamental stability properties of {numerous} partial differential equations and enter the stability and convergence theory of numerical methods.
Upper bounds for the Poincar\'e--Friedrichs constants also provide lower bounds for the eigenvalues of the associated Laplacians.
However, while potentials and Poincar\'e--Friedrichs constants for the gradient have been subject to extensive study,
quantifiable results regarding the curl and divergence operators, or more generally, the exterior derivative, are largely unavailable.

This manuscript contributes to the theory of computable estimates for Poincar\'e--Friedrichs constants for the differential operators of vector calculus.
How to construct potentials for the gradient over triangulated domains is well-documented in the literature.
Here, we extend this construction to the curl and divergence operators, treating them as instances of the exterior derivative.
However, to proceed in the general exterior derivative case, we restrict our efforts to so-called \emph{shellable triangulations}.
The class of shellable triangulations includes practically relevant triangulations:
for example, local triangulations around simplices within a larger triangulation (the so-called local patches or stars) are shellable in dimensions two and three; see also Figure~\ref{figure:patches}.
Only contractible domains can ever admit a shellable triangulation, but having computable upper bounds for such domains is an important stepping stone toward more general situations.
Additionally, we include a study of regularized Poincar\'e and Bogovski\u{\i} operators that leads to new Poincar\'e--Friedrichs inequalities with computable constants for the whole $L^{p}$ de~Rham complex over bounded convex domains.

\subsection{Conceptual overview}

We give a conceptual overview of the topic before we outline the known results in the literature and our contributions in more detail. 
Our conceptual point of reference is the Poincar\'e--Friedrichs inequality for the gradient of scalar functions, which has been subject to extensive research.

\subsubsection{Potentials and Poincar\'e--Friedrichs inequalities for the gradient}\label{section:intro:potential}

For the purpose of illustration, we let $\Omega \subseteq \bbR^3$ be a bounded connected open set.
We let $L^{p}(\Omega)$ denote the Lebesgue space over $\Omega$ with integrability exponent $1 \leq p \leq \infty$, 
and we write $W^{1,p}(\Omega)$ for the first-order Sobolev space over $\Omega$ with integrability exponent $p$.

We are interested in finding a constant $C_{\Omega,\grad,p} > 0$ such that the following holds:
for every gradient vector-valued field $\bff \in \nabla W^{1,p}(\Omega)$ 
there exists a scalar potential $u \in W^{1,p}(\Omega)$ (meaning a preimage under the gradient)
such that $\nabla u = \bff$ and 
\begin{align}\label{math:intro:inequalitypf:gradient} 
    \| u \|_{L^{p}(\Omega)}
    \leq 
    C_{\Omega,\grad,p} 
    \| \bff \|_{L^{p}(\Omega)}
    .
\end{align}
This inequality is called \emph{Poincar\'e--Friedrichs inequality} and the constant $C_{\Omega,\grad,p}$ is called the \emph{Poincar\'e--Friedrichs constant} with exponent $p$. 
The question is therefore whether we can always find a gradient potential of sufficiently small Lebesgue norm.
One possible choice is the norm-minimizing gradient potential
\begin{align}\label{math:intro:gradientpotential}
    \Phi_{\grad}( \bff ) := \argmin\limits_{ \substack{ u \in W^{1,p}(\Omega) \\ \nabla u = \bff } } \| u \|_{L^{p}(\Omega)}, \qquad \bff \in \nabla W^{1,p}(\Omega).
\end{align}
In fact, its operator norm is the best Poincar\'e--Friedrichs constant: 
\begin{align}\label{math:intro:gradient:potentialnorm}
    C_{\Omega,\grad,p} := \max\limits_{ u \in W^{1,p}(\Omega) \setminus \bbR } 
    \frac{ \| \Phi_{\grad}(\nabla u) \|_{L^{p}(\Omega)} }{ \| \nabla u \|_{\fooL^{p}(\Omega)} }.
\end{align}
In the present setting, where $\Omega$ is connected, the different possible gradient potentials only differ by constants.
Thus, $\Phi_{\grad}( \bff )$ solves a one-dimensional convex minimization problem, and the constant $C_{\Omega,\grad,p}$ is the best possible constant in the inequality
\begin{align}\label{math:intro:poincarefriedrichs:grad}
    \min\limits_{c\in\bbR}
    \| u - c \|_{L^{p}(\Omega)}
    \leq 
    C_{\Omega,\grad,p} \| \nabla u \|_{L^{p}(\Omega)},
    \quad 
    \forall
    u \in W^{1,p}(\Omega) 
    .\tag{PF} 
\end{align}
Finding the norm-minimizing potential and the optimal Poincar\'e--Friedrichs constant is generally challenging
because $\Phi_{\grad}$ is a nonlinear operator unless $p = 2$.
For that reason, research has focused on \emph{linear} potential operators. 
Indeed, any \emph{linear} potential operator $\Phi_{} : \nabla W^{1,p}(\Omega) \rightarrow W^{1,p}(\Omega)$ satisfying $\nabla \Phi_{}(\bff) = \bff$ for any $\bff \in W^{1,p}(\Omega)$ satisfies
\begin{align}
    \max\limits_{ u \in W^{1,p}(\Omega) \setminus \bbR } 
    \frac{ \| \Phi_{\grad}(\nabla u) \|_{L^{p}(\Omega)} }{ \| \nabla u \|_{\fooL^{p}(\Omega)} }
    \leq 
    \max\limits_{ u \in W^{1,p}(\Omega) \setminus \bbR } 
    \frac{ \| \Phi_{}(\nabla u) \|_{L^{p}(\Omega)} }{ \| \nabla u \|_{\fooL^{p}(\Omega)} }
    .
\end{align}
Thus, its operator norm serves as an upper bound for the Poincar\'e--Friedrichs constant $C_{\Omega,\grad,p}$.
This is a general principle: upper bounds for Poincar\'e--Friedrichs constants are easily obtained from linear potentials. 
We highlight this perspective because it is our theoretical underpinning when generalizing the discussion to the curl and divergence operators below. 

One straightforward example for such a bounded linear potential operator is the average-free potential.
This potential operator $\Phi_{\varnothing}$ is defined by 
\begin{align}\label{math:intro:poincarepotential}
    \Phi_{\varnothing}( \nabla u ) 
    := 
    u - u_{\Omega}, 
    \qquad 
    u \in W^{1,p}(\Omega)
    ,
\end{align} 
where $u_\Omega$ denotes the average of $u$ over $\Omega$. 
Its operator norm is the Poincar\'e constant $C_{\Omega,\varnothing,p} > 0$,
which satisfies the inequality \begin{align}\label{math:intro:poincare:grad}
    \| u - u_{\Omega} \|_{L^{p}(\Omega)}
    \leq 
    C_{\Omega,\varnothing,p} \| \nabla u \|_{L^{p}(\Omega)},
    \quad 
    \forall 
    u \in W^{1,p}(\Omega)
    . \tag{P} 
\end{align}
We emphasize that the average-free gradient potential \eqref{math:intro:poincarepotential} is generally not the norm-minimizing gradient potential unless $p = 2$.
Hence, the optimal Poincar\'e constant in \eqref{math:intro:poincare:grad} generally differs from the optimal Poincar\'e--Friedrichs in \eqref{math:intro:poincarefriedrichs:grad}.
Considerable research efforts have gone into determining the best possible values for these constants. 
Upper estimates for these constants also correspond to lower bounds for the spectra of Neumann--Laplacians over those domains.
The relationship between the Poincar\'e constant and the Poincar\'e--Friedrichs constant will be further elaborated upon in later sections of this manuscript.

\subsubsection{Potentials and Poincar\'e--Friedrichs inequalities for the curl and divergence}

The study of potentials and Poincar\'e--Friedrichs inequalities for the curl or divergence operators in vector calculus is substantially different, with much fewer results available in the literature.
Consider the spaces of vector fields
\begin{align}
    \bfW^{p}(\curl,\Omega) &:= \left\{ \bfu \in L^{p}(\Omega)^{3} : \curl \bfu \in L^{p}(\Omega)^{3} \right\}
    \label{math:intro:Wcurl}
    ,
    \\
    \bfW^{p}(\divergence,\Omega) &:= \left\{ \bfu \in L^{p}(\Omega)^{3} : \divergence \bfu \in L^{p}(\Omega) \right\}
    \label{math:intro:Wdiv}
    .
\end{align}
Their members are those vector-valued fields in Lebesgue spaces whose distributional curls and divergences, respectively, are in Lebesgue spaces. 
In contrast to the gradient, this only requires that certain sums of distributional partial derivatives are integrable, 
which means that these spaces are not classical Sobolev spaces of vector-valued fields. 
Our objective is to find bounded potential operators (i.e., right inverses) for the differential operators
\begin{align*}
    \curl : \bfW^{p}(\curl,\Omega) \rightarrow L^{p}(\Omega)^{3},
    \qquad 
    \divergence : \bfW^{p}(\divergence,\Omega) \rightarrow L^{p}(\Omega).
\end{align*}
We are interested in the natural analogs for the Poincar\'e--Friedrichs inequality of the gradient~\eqref{math:intro:poincarefriedrichs:grad},
which for the curl and divergence are the inequalities 
\begin{align}
    \min\limits_{ \substack{ \bfv \in \bfW^{p}(\curl,\Omega) \\ \curl \bfv = \bff } } 
    \| \bfv \|_{L^{p}(\Omega)}
    &\leq 
    C_{\Omega,\curl,p}
    \| \bff \|_{L^{p}(\Omega)},
    \qquad 
    \forall \bff \in \curl \bfW^{p}(\curl,\Omega)
    \label{math:intro:Wcurl:pf}
    ,
    \\\min\limits_{ \substack{ \bfv \in \bfW^{p}(\divergence,\Omega) \\ \divergence \bfv = f } } 
    \| \bfv \|_{L^{p}(\Omega)}
    &\leq 
    C_{\Omega,\divergence,p}
    \| f \|_{L^{p}(\Omega)}
    \qquad 
    \forall \bff \in \divergence \bfW^{p}(\divergence,\Omega)
    \label{math:intro:Wdiv:pf}
    .
\end{align}
The fundamental difference to the gradient is that the curl and divergence have infinite-dimensional kernels.
The kernel of the gradient is the one-dimensional space of constant functions, and it is thus trivially complemented for all $p$, with a canonical choice of projection.
By contrast, the kernels of the curl and divergence operators are generally infinite-dimensional.
Correspondingly, the curl and the divergence do not admit any easily available analog to the Poincar\'e inequality~\eqref{math:intro:poincare:grad}.
Moreover, it is not immediately evident that the kernels of the curl and divergence operators are complemented in the Banach space case when $p \neq 2$,
and a canonical projection only exists here in the Hilbert setting $p=2$.
In the light of these facts, there generally is no natural analog to the Poincar\'e inequality~\eqref{math:intro:poincare:grad} for the curl and divergence.

In the Banach space case, not even the existence of norm-minimizing potentials is trivial.
We are therefore interested in any potentials 
\begin{gather}
    \Phi_{\curl} : \curl \bfW^{p}(\curl,\Omega) \rightarrow \bfW^{p}(\curl,\Omega)
    ,
    \\
    \Phi_{\divergence} : \divergence \bfW^{p}(\divergence,\Omega) \rightarrow \bfW^{p}(\divergence,\Omega),
\end{gather}
and their operator norms. We remark that upper bounds for the Poincar\'e--Friedrichs inequality of the curl operator correspond to lower bounds for the so-called Maxwell eigenvalues. 
In sharp contrast to the extensive research on gradient potentials, not many results seems available on computable constants in such Poincar\'e--Friedrichs inequalities for the curl and divergence operators.

\subsubsection{Potentials and Poincar\'e--Friedrichs inequalities for the exterior derivative}

Though the above discussion is presented in vector calculus, our main results and arguments are given in the formalism of exterior calculus.
Exterior calculus~\cite{greub1967multilinear,lee2012smooth} is used ubiquitously in the mathematical literature of physics and engineering and has found widespread adoption in the theoretical and numerical analysis for vector-valued partial differential equations~\cite{hiptmair2002finite,gross2004electromagnetic,arnold2006finite,arnold2009geometric,arnold2010finite,demlow2014posteriori,licht2021local,arnold2021complexes}. 
This formalism is independent of the spatial dimension and highlights the underlying geometric structures common to the gradient, curl, and divergence operators in three space dimensions.
In this setting, the Poincar\'e--Friedrichs inequalities we are interested in read as follows:
For every $u \in W^{p}\Alt^{k}(\Omega)$, there exists $w \in W^{p}\Alt^{k}(\Omega)$ with the same exterior derivative $\cartan w = \cartan u$ and such that\footnote{All the notation is fixed in detail in Section~\ref{section:calculus} later.}
\begin{align}\label{math:intro:exterior:pf} 
    \| w \|_{L^{p}(\Omega)} \leq  C_{\Omega,k,p} \| \cartan w \|_{L^{p}(\Omega)}.
\end{align} 
For the purpose of our discussion, this formalism allows us to leverage results from a larger body of literature in differential geometry and functional analysis.

\subsection{Literature review}

We review the literature on Poincar\'e--Friedrichs inequalities. We also identify the obstructions inherent to presently known results that we wish to overcome with the present contribution. 

\subsubsection{General results}

The qualitative existence of Poincar\'e--Friedrichs inequalities for the differential operators of vector calculus and exterior calculus is known for a large class of domains.
We emphasize that non-constructive arguments, such as the Rellich embedding theorem, generally do not yield explicitly computable constants.

There is a vast body of literature on the Poincar\'e inequality for the gradient and its numerous variants, which may include weighted integrals or boundary terms.
The divergence operator has received less attention.
In the Hilbert space case $p=2$, the Friedrichs inequality~\cite{burenkov1998sobolev}
over the Sobolev space $W^{1,2}(\Omega)$ with homogeneous Dirichlet boundary conditions
along $\partial \Omega$ yields the Poincar\'e--Friedrichs constant for the divergence
operator in~\eqref{math:intro:Wdiv:pf} by duality.
However, that easy duality argument, which yields an explicit upper bound proportional to the domain diameter, seems inherently restricted to $p=2$.

The core challenges are found in the discussion of the curl operator in three dimensions,
where Poincar\'e--Friedrichs inequalities have appeared under different names.
Let us assume momentarily that $\Omega$ is a weakly Lipschitz domain with trivial topology and consider only the Hilbert case $p=2$.
Then the constant in~\eqref{math:intro:Wcurl:pf} agrees with the constant in the so-called Poincar\'e--Friedrichs--Weber inequality
\begin{align}
    \label{math:intro:weber}
    \| \bfu\|_{L^{2}(\Omega)}   \leq C_{\Omega,\curl,p} \| \curl \bfu \|_{L^{2}(\Omega)},
\end{align}
valid for all $\bfu \in \bfW^{2}(\curl,\Omega) \cap \bfW^{2}(\divergence,\Omega)$ that satisfy $\divergence \bfu = 0$ and have vanishing normal or vanishing tangential trace along $\partial \Omega$. 
Equivalently,~\eqref{math:intro:weber} is valid for all $\bfu \in \bfW^{2}(\curl,\Omega)$ that are $L^{2}$-orthogonal to the gradients of scalar-valued fields in $W^{1,2}(\Omega)$
or that have vanishing tangential trace and are $L^{2}$-orthogonal to the gradients of those scalar-valued fields in $W^{1,2}(\Omega)$ that satisfy Dirichlet boundary conditions.
We refer to Equation~(5) in~\cite{Fried_diff_forms_55}, Equation~(2) in~\cite{Gaff_Hilbert_harm_55}, \cite{Web_compact_Maxw_80} as well as~\cite[Lemmas~3.4 and~3.6]{Gir_Rav_NS_86}, \cite[Proposition~7.4]{Fer_Gil_Maxw_BC_97}, \cite{paulyvaldman} and the references therein. 
The general case of $L^p$ differential forms over Lipschitz manifolds subject to partial boundary conditions is discussed in~\cite{goldshtein2011hodge}.
However, while many of the above results rely on non-constructive estimates, 
we are interested in practically computable upper bounds.

\subsubsection{Analytical constants over convex domains}

Considerable research effort has gone into computing explicit upper estimates for the constants in the Poincar\'e--Friedrichs inequalities over convex domains.
Notably, if the constants are required to depend on the convex domain only via its diameter,
then the optimal gradient Poincar\'e--Friedrichs and Poincar\'e constants for the entire range of Lebesgue exponent $1 \leq p < \infty$ are known explicitly~\cite{Pay_Wei_Poin_conv_60,bebendorf2003note,acosta2004optimal,esposito2013poincare,ferone2012remark}, see also \cite{chua2006estimates}.

The literature on Poincar\'e--Friedrichs constants and potentials of curls and divergences is less extensive than for potentials of gradients, even over convex domains.
Guerini and Savo~\cite{guerini2004eigenvalue} address the spectrum of the Hodge--Laplace operator on bounded convex domains with smooth boundary in the Hilbert space case $p=2$. Among their results is the observation that the Poincar\'e--Friedrichs constant for the gradient already estimates the corresponding constants for the curl and divergence operators.
They also provide explicit (but not necessarily optimal) upper bounds for Poincar\'e--Friedrichs constants that depend only on the dimension and diameter of the convex domain.
A duality argument then yields upper estimates of the Poincar\'e--Friedrichs constants for the gradient, curl, and divergence operators subject to Dirichlet, tangential, and normal boundary conditions, respectively, along the entire boundary.
However, no results as those of~\cite{guerini2004eigenvalue} are known over bounded convex Lipschitz domains and for general Lebesgue exponents $1 \leq p \leq \infty$.

\subsubsection{Domains star-shaped with respect to a ball}

When the domain is not necessarily convex but star-shaped with respect to a ball,
then several estimates for Poincar\'e--Friedrichs constants are known.\footnote{In fact, estimates for Poincar\'e--Friedrichs constants even hold over any star-shaped open bounded set; see~\cite[Theorem~3.1]{hurri1988poincare}.} Polynomial interpolation estimates already imply the gradient Poincar\'e--Friedrichs inequality~\cite{brenner2008mathematical,ern2021finite}.
We pay particular attention to the regularized Poincar\'e and Bogovski\u{\i} potential operators for the exterior derivative,
such as those of Costabel and McIntosh~\cite{costabel2010bogovskiui}.
If we conceive these potentials as mappings between Lebesgue spaces of differential forms, then their operator norms are upper estimates for the Poincar\'e--Friedrichs constants of the domain.
We are aware of estimates for the higher-order seminorms of these potentials~\cite{guzman2021estimation},
but estimates in Lebesgue norms have not been made explicit in the literature yet, to the best of our knowledge.
We particularly emphasize that all these estimates for domains star-shaped with respect to a ball have in common that they rely on upper bounds for the \emph{eccentricity} of the domain.

Let us briefly discuss the practical limitations of the aforementioned estimates for Poincar\'e--Friedrichs constants for bounded convex domains or domains that are star-shaped with respect to a ball.
Recall that the main objective of this manuscript is bounding Poincar\'e--Friedrichs constants over domains with shellable triangulations, with a key application being local stars within triangulated domains.
Not all local stars describe convex subdomains.
Even if the local stars are star-shaped with respect to a ball,
which would enable the averaged Poincar\'e and Bogovski\u{\i} operators~\cite{costabel2010bogovskiui},
the estimates that rely on this geometric condition deteriorate when the aforementioned ball has a radius much smaller than the domain diameter.
While this is not as much a problem over local patches (stars) around interior subsimplices, where the size of the interior ball only depends on the shape regularity of the triangulation, the interior ball can be arbitrarily small when the local patch is around a boundary simplex, even if the triangulation displays decent shape regularity.
This occurs most prominently when the domain has sharp reentrant corners, illustrative limit cases including the slit domain~\cite{veeser2012poincare} and the crossed bricks domain~\cite{licht2019smoothed},
which contain local finite element patches that are not star-shaped with respect to any ball.
In view of this, we refrain from treating local patches (stars) in triangulations as domains star-shaped with respect to a ball.

\subsubsection{Triangulated domains}

Geometric settings that admit finite triangulations enable different pathways to estimate the constant in Poincar\'e--Friedrichs inequalities.
We review the main outcomes.

Computable estimates for Laplacian eigenvalues over triangulated domains have received much attention.
The constant in~\eqref{math:intro:inequalitypf:gradient} for $p=2$ corresponds to a lower bound for the Laplace eigenvalues and quantifies the stability properties of the Laplacian on the domain $\Omega$.
Similarly, the constant in~\eqref{math:intro:Wcurl:pf} for $p=2$ corresponds to a lower bound for the Maxwell eigenvalues and quantifies the stability properties of the Maxwell system on $\Omega$.
Thus, computable upper bounds on the Poincar\'e--Friedrichs constants also give computable lower bounds for the eigenvalues of the associated Laplacians and vice versa.
Prominent methods numerically compute guaranteed upper bounds on the Poincar\'e--Friedrichs constants upon solving a finite element system over a sufficiently fine triangulation and using some clever post-processing estimates.
This approach has led to estimates for scalar Laplacian eigenvalues~\cite{Cars_Ged_LB_eigs_14,Liu_fram_eigs_15,zbMATH06969657} and vector Laplacian eigenvalues~\cite{gallistl2023computational}.
For the purposes of this manuscript, however, we aim for computable upper bounds that do not rely on the solution of (global) finite element systems.

There are numerous estimates for Poincar\'e--Friedrichs constants that only rely on locally computable geometric quantities, such as the diameter and volumes of simplices.
Veeser and Verf\"urth~\cite{veeser2012poincare} provide computable upper bounds in the case of the classical Sobolev space $W^{1,p}(\Omega)$ over vertex stars. Naturally, their estimates depend on the shape regularity of the mesh.
A whole class of upper bounds for Poincar\'e--Friedrichs inequalities uses some form of passing through the triangulation and constructs the potentials step-by-step.
The underlying idea is that we first construct a potential for the gradient over an initial simplex.
Every time we have found a potential over a subdomain, we construct a potential over a neighboring simplex or patch:
along the interfacing intersection, the two potentials will only differ by a constant,
and that difference can easily be removed to ensure continuity across that interface.
Cell by cell, the potential is constructed over increasing subdomains, matching along the interfacing intersections until the entire domain is exhausted.
The method is known in the finite element literature~\cite{ern2021finite}.
It was previously used in the context of finite volume methods~\cite{Eym_Gal_Her_00}, broken (weakly continuous) Sobolev spaces~\cite{vohralik2005discrete}, or more recently in continuous--discrete comparison results~\cite{Brae_Pill_Sch_p_rob_09,ern2020stable,Chaum_Voh_p_rob_3D_H_curl_24,Voh_loc_glob_H1_25}.
This sequential procedure applies to general triangulated domains, not only local stars.
Most importantly, these sequential estimates of Poincar\'e--Friedrichs constants generally circumvent the effect of low boundary regularity and only rely on the mesh's shape regularity.

While we thus know Poincar\'e--Friedrichs constants over local stars for scalar-valued Sobolev spaces,
we are not aware of computable estimates for the case of $\bfW^{p}(\curl,\Omega)$ and $\bfW^{p}(\divergence,\Omega)$ over finite element stars.

\subsection{Objectives and methodology}

The main objective of this manuscript is the construction of potentials for the gradient, curl, and divergence operators, and, more generally, the exterior derivative.
The operator norms of our potentials satisfy computable upper bounds, thus yielding computable upper estimates of the Poincar\'e--Friedrichs constants as well.
In what follows, we focus on domains with shellable triangulations, which include local patches (stars) in two and three dimensions as important special cases.
We also devote an important effort specifically to convex domains.

\subsubsection{Convex domains}

Our main result for convex domains in the exterior calculus setting is the construction of regularized Poincar\'e and Bogovski\u{\i} potential operators with explicitly bounded operator norms, as summarized in Theorem~\ref{theorem:bogovpoinc}.
These upper bounds for the Poincar\'e--Friedrichs constants are proportional to the domain diameter and are bounded in terms of the domain's eccentricity.
The bounds are independent of the Lebesgue exponent $1 \leq p \leq \infty$, though the space dimension and the form degree enter the estimates.

The reason for our study of Poincar\'e--Friedrichs constants over convex domains is twofold:
firstly, they are of independent interest, and secondly, we will need them as a component for our main results on triangulations.
Our exposition of regularized Poincar\'e and Bogovski\u{\i}-type potential operators follows the general methodology of Costabel and McIntosh~\cite{costabel2010bogovskiui}.
By comparison, we simplify the potential operators:
we only study them over convex domains instead of domains star-shaped with respect to a ball. Moreover, we use simpler (constant) weight functions.
While the resulting potentials feature generally lower regularity, they are conducive to our purposes.
Crucially, this allows us to estimate their operator norms easily and thus bound Poincar\'e--Friedrichs constants.

\subsubsection{Potentials subject to partial boundary conditions}

We also address the construction of potentials for the gradient, curl, and divergence operators, being special cases of the exterior derivative, over a simplex subject to partial boundary conditions.
We are not aware of explicit estimates or regularized potentials for these boundary conditions in the published literature.
For example, given a divergence-free vector-valued field over a tetrahedron with vanishing normal trace along three of the tetrahedron's faces,
we want to find the vector-valued field potential that not only is a preimage under the curl operator but also has a vanishing tangential trace along the same three faces.
We achieve this by constructing an auxiliary problem subject to full boundary conditions, building upon the regularized Bogovski\u{\i} operators.
We thus obtain Poincar\'e--Friedrichs inequalities over simplices and subject to partial boundary conditions.
Again, we address the entire range $1 \leq p \leq \infty$ of Lebesgue exponents, and our Poincar\'e--Friedrichs constants are explicitly computable.

\subsubsection{Main results}

Our main objective remains to find potentials for the differential operators of vector calculus and exterior calculus, including the curl and divergence, over a broad class of triangulated domains.
The operator norms of these potentials will serve as our computable Poincar\'e--Friedrichs constants.
In light of the different approaches discussed above,
we seek constants that are explicitly computed in terms of mesh geometry, that do not require the solution of global finite element problems,
and that do not depend on the boundary regularity of the domain.

As discussed earlier, the sequential construction of potentials for gradient vector-valued fields is well-established in the literature and serves as our conceptual blueprint.
Gradient potentials are easily computed over each individual simplex, but the constants of integration generally do not match, so the scalar piecewise potential will belong to a broken Sobolev space.
The local potentials will differ by a constant along the simplex boundaries.
However, we can sequentially construct a potential in Sobolev spaces over increasingly larger intermediate subdomains: at each step, we select a simplex that shares at least one face with one of the previously processed simplices and adjust the constant integration of the local potential.
The global potential is built cell by cell until the entire domain is covered.
Our main results in this context are Theorem~\ref{theorem:poincarefriedrichsestimate:grad} and Theorem~\ref{theorem:fullrecursivesum:grad}.

We generalize this sequential construction of potentials to the curl, the divergence, and, more generally, the exterior derivative.
However, we need to overcome new challenges that arise due to the infinite-dimensional kernels of these differential operators, as we now explain in more detail.
The basic inductive strategy remains the same.
We start by constructing, for instance, a curl potential over a single simplex.
Having already defined a potential operator over a subtriangulation, we select a neighboring simplex that shares at least one face with the preceding simplices.
We then construct a potential for the curl operator whose tangential traces along the shared faces match those of the already existing potential.
Repeating this procedure eventually exhausts the original triangulation.
Our main results here are Theorem~\ref{theorem:poincarefriedrichsestimate:exterior} and Theorem~\ref{theorem:fullrecursivesum:exterior}.

However, unlike in the potential construction of the gradient, it is not immediately evident whether constructing the local curl potential with given tangential traces is a well-posed auxiliary problem.
For the case of the gradient, it is sufficient that the sequential traversal of the triangulation satisfies that each new simplex shares at least one face with one of the previous simplices.
For the differential operators of vector calculus and exterior calculus, it seems that we must be more restrictive:
we require the new simplex to intersect with the existing subtriangulation along a boundary submanifold of dimension $n-1$, which means a collection of faces of this simplex but no individual edges or vertices, allowing us to define a well-posed auxiliary problem and to extend the existing curl potential to the new simplex.
Whether a triangulation admits such a particular traversal is a non-trivial condition and defines the class of \emph{shellable} triangulations.

Shellable simplicial complexes, and more generally polytopal complexes, are a well-established notion in discrete geometry and combinatorics; see, e.g., Kozlov~\cite{kozlov2008combinatorial} and Ziegler~\cite{ziegler1995lectures}, and the references therein.
As already highlighted, local patches (stars) in 2D and 3D triangulations are shellable. 
Generally speaking, domains that have a shellable triangulation must be contractible. 
Not all triangulations of contractible domains are shellable, but the exceptions are rare and requiring the triangulation to be shellable poses little restriction in practice.

\subsection{Basic notation}

Whenever $x \in \bbR^{n}$ is a vector, we write $\|x\| = \|x\|_{2}$ for its Euclidean norm, 
and whenever $A \in \bbR^{n \times n}$, we let $\| A \|_{2}$ be its operator norm with respect to the Euclidean norm. 
Furthermore, $\Jacobian F$ always denotes the Jacobian of any mapping $F$.

\subsection{Organization of this manuscript}

The remainder of this manuscript is structured as follows.
We review Poincar\'e--Friedrichs constants for the gradient over convex domains in Section~\ref{section:poincare},
where we also discuss the difference with the Poincar\'e inequality and motivate our interest in linear potentials.
We review basic notions of triangulations in Section~\ref{section:triangulations}.
Subsequently, we develop the computable upper bounds for the Poincar\'e--Friedrichs constants for the gradient over face-connected triangulated domains in Section~\ref{section:gradient}.
Sobolev spaces in vector calculus and the calculus of differential forms are reviewed in Section~\ref{section:calculus}. 
We then introduce our regularized potentials over convex sets in Section~\ref{section:potentialoperator}, 
giving rise to computable Poincar\'e--Friedrichs constants as their operator norms. 
We subsequently review shellable triangulations of manifolds in Section~\ref{section:advancedtriangulations},
and we construct an important geometric reflection operator in Section~\ref{section:extension}. 
Finally, we provide computable upper bounds for Poincar\'e--Friedrichs constants for the exterior derivative over shellable triangulations in Section~\ref{section:poincarefriedrichs} and present numerical examples in Section~\ref{section:numericalexamples}.
We conclude with an outlook in Section~\ref{section:outlook}.

\section{Review of Poincar\'e and Poincar\'e--Friedrichs inequalities}\label{section:poincare}

This section surveys variations of the Poincar\'e--Friedrichs inequalities for the gradient operator, with emphasis on analytical upper bounds over bounded convex domains.
Continuing the discussion from the introduction (Section~\ref{section:intro:potential}), 
we detail the difference between the Poincar\'e--Friedrichs inequality, which addresses the norm-minimizing potential, and the Poincar\'e inequality, which addresses the potential with mean value zero.
This survey serves as a building block in constructing computable constants over triangulated domains in a combinatorial way below (Section~\ref{section:gradient}), but we believe it is also of independent interest.
\\

Let $\Omega \subseteq \bbR^{n}$ be a connected open set. 
Given any $p \in [1,\infty]$, we let $L^{p}(\Omega)$ denote the Lebesgue space over $\Omega$ with integrability exponent $p$, and we write $\bfL^{p}(\Omega) := L^{p}(\Omega)^{n}$ for the corresponding Lebesgue space of vector-valued fields. 
We also write $W^{1,p}(\Omega)$ for the first-order Sobolev space over $\Omega$ with integrability exponent $p$.

We say that a domain $\Omega \subseteq \bbR^{n}$ satisfies the \emph{Poincar\'e--Friedrichs inequality} with exponent $p \in [1,\infty]$
if there exists a constant $C_{\Omega,\grad,p} \geq 0$ such that the following holds:
for every vector-valued field $\bff \in \nabla W^{1,p}(\Omega)$ there exists $u \in W^{1,p}(\Omega)$
such that $\nabla u = \bff$ and 
\begin{align}\label{math:inequalitypf:gradient}
    \| u \|_{L^{p}(\Omega)}
    \leq 
    C_{\Omega,\grad,p} 
    \| \bff \|_{L^{p}(\Omega)}
    .
\end{align}
Since $\Omega$ is connected, this is equivalent to  
\begin{align*}
    \min_{ c \in \bbR } \| u - c \|_{L^{p}(\Omega)}
    \leq 
    C_{\Omega,\grad,p} 
    \| \nabla u \|_{L^{p}(\Omega)},
    \quad 
    \forall 
    u \in W^{1,p}(\Omega)
    .
\end{align*}
We call $C_{\Omega,\grad,p}$ the \emph{Poincar\'e--Friedrichs constant} with exponent $p$.

\subsection{Relationship with Poincar\'e inequalities}

We wish to clarify the relationship between the Poincar\'e--Friedrichs inequality, in the sense introduced above, with other inequalities that are known as Poincar\'e inequality (or also Poincar\'e--Wirtinger or Friedrichs inequality) in the literature~\cite[Remark~3.32]{ern2021finite}. Given $p \in [1,\infty]$ and a domain $\Omega \subseteq \bbR^{n}$ of finite measure, 
we say that $\Omega$ satisfies the Poincar\'e inequality with exponent $p$ 
if there exists $C_{\varnothing,p} \geq 0$ such that 
\begin{align*}
    \| u - u_{\Omega} \|_{L^{p}(\Omega)}
    \leq 
    C_{\Omega,\varnothing,p} 
    \| \nabla u \|_{L^{p}(\Omega)}
    ,
    \quad 
    \forall u \in W^{1,p}(\Omega)
    ,
\end{align*}
where $u_{\Omega}$ is the \emph{average} of $u$ over $\Omega$, that is,
\begin{align*}
    u_{\Omega} := \vol(\Omega)^{-1} \int_{\Omega} u(x) \;dx.
\end{align*}
Clearly, this Poincar\'e inequality implies the Poincar\'e--Friedrichs inequality and we have 
\[
    C_{\Omega,\grad,p} \leq C_{\Omega,\varnothing,p}.
\]
Towards a converse inequality, 
let us first observe that the average of any $u \in W^{1,p}(\Omega)$ with $p < \infty$ satisfies the bound 
\begin{align}\label{math:aver_norm}
    \| u_\Omega \|_{L^{p}(\Omega)}^{p}
    = 
    \int_{\Omega} \left( \vol(\Omega)^{-1} \int_{\Omega} |u(x)| \;dx \right)^{p}
    \leq 
    \int_{\Omega} \vol(\Omega)^{-1} \int_{\Omega} |u(x)|^{p} \;dx
    = 
    \| u \|_{L^{p}(\Omega)}^{p}
    .
\end{align}
Here, we have used H\"older's or Jensen's inequality. 
In the case $p = \infty$, any $u \in L^{\infty}(\Omega)$ satisfies $\| u_\Omega \|_{L^{\infty}(\Omega)} \leq \| u \|_{L^{\infty}(\Omega)}$. 
We conclude that taking the average is a projection within Lebesgue spaces with unit norm. 
The triangle inequality now shows that 
\begin{align*}
    \| u - u_\Omega \|_{L^{p}(\Omega)} 
    \leq
    2
    \| u \|_{L^{p}(\Omega)},
    \quad 
    \forall
    u \in L^{p}(\Omega)
    .
\end{align*}
Thus, the Poincar\'e--Friedrichs inequality implies the Poincar\'e inequality with 
\begin{align}\label{math:p_is_smaller_than_pf}
    C_{\Omega,\grad,p} \leq 2 C_{\Omega,\varnothing,p}.
\end{align}
In the special case $p=2$, taking the average is an orthogonal projection, 
and so this improves to $\| u - u_\Omega \|_{L^{2}(\Omega)} \leq \| u \|_{L^{2}(\Omega)}$ for any $u \in L^{2}(\Omega)$. 
Hence,
\begin{align}\label{math:pf_equals_p_in_hilbertspace}
    C_{\Omega,\varnothing,2} = C_{\Omega,\grad,2}. 
\end{align}
This improvement also follows from the projection estimate (see, e.g.,~\cite{xu2003some}).
Stern's generalized projection estimate~\cite[Theorem~4.1,Remark~5.1]{stern2015banach} implies improved estimate for all $L^{p}$ spaces with $1 \leq p \leq \infty$:
since taking the average is a projection onto the constants functions with unit norm, from~\eqref{math:aver_norm} it now follows that 
\begin{align*}
    \| u - u_\Omega \|_{L^{p}(\Omega)}
    \leq 
    \min\left( 2, 2^{|2/p-1|} \right)
    \| u \|_{L^{p}(\Omega)}
    = 
    2^{|2/p-1|} 
    \| u \|_{L^{p}(\Omega)}
    ,
    \quad 
    \forall 
    u \in L^{p}(\Omega)
    .
\end{align*}
Here, we have used $1 \leq 2^{|2/p-1|} \leq 2$ for $1 \leq p \leq \infty$.
We thus conclude 
\begin{align}\label{math:p_vs_pf}
    C_{\Omega,\grad,p} \leq 2^{|2/p-1|} C_{\Omega,\varnothing,p}.
\end{align}
In the limit cases $p = 1$ and $p = \infty$ we reproduce~\eqref{math:p_is_smaller_than_pf}, 
and in the case $p = 2$ we achieve the identity~\eqref{math:pf_equals_p_in_hilbertspace} once again.
In summary, our notion of Poincar\'e--Friedrichs constant is equivalent to the common notion of Poincar\'e constant, up to a numerical factor that only depends on $1 \leq p \leq \infty$ and that is at most $2$.

As discussed in Section~\ref{section:intro}, the notion of Poincar\'e--Friedrichs inequality suits our discussion better than the Poincar\'e inequality.
The kernel of the gradient is the one-dimensional space of constant functions, and is thus complemented in the Lebesgue spaces with a canonical choice of projection.
By contrast, the curl and divergence operators have infinite-dimensional kernels, and so it is not even trivial whether these kernels are complemented subspaces and admit a projection onto them, not to mention a canonical projection.

\subsection{Analytical constants in Poincar\'e--Friedrichs inequalities over bounded convex domains}\label{subsection:PX_convex}

We collect examples for Poincar\'e and Poincar\'e--Friedrichs inequalities for the important special case of bounded convex domains. 
We have the Poincar\'e inequalities~\cite{Pay_Wei_Poin_conv_60,bebendorf2003note,acosta2004optimal} (or~\cite[Lemma~3.24]{ern2021finite}) 
\begin{align}\label{math:acostaduran}
    \| u - u_{\Omega} \|_{L^{1}(\Omega)}
    \leq 
    \frac{\diam(\Omega)}{2}
    \| \nabla u \|_{L^{1}(\Omega)}
    ,
    \quad 
    \forall 
    u \in W^{1,1}(\Omega)
    ,
    \\
    \| u - u_{\Omega} \|_{L^{2}(\Omega)}\label{math:bebendorf}
    \leq 
    \frac{\diam(\Omega)}{\pi}
    \| \nabla u \|_{L^{2}(\Omega)}
    ,
    \quad 
    \forall 
    u \in W^{1,2}(\Omega)
    ,
\end{align}
where $\diam(\Omega)$ is the diameter of the domain $\Omega$.
These two estimates are the best possible Poincar\'e inequalities in the cases $p=1$ and $p=2$, respectively, in terms of the diameter alone. 
Upper bounds for the Poincar\'e constant over bounded convex domains with $1 < p < \infty$ are known in the literature~\cite[Theorem~1.1, Theorem~1.2]{chua2006estimates}:
\begin{align}\label{math:chuawheeden}
    \| u - u_{\Omega} \|_{L^{p}(\Omega)}
    \leq 
    C_{{\rm CW},p}
    \diam(\Omega)
    \| \nabla u \|_{L^{p}(\Omega)}
    ,
    \quad 
    \forall 
    u \in W^{1,p}(\Omega)
    ,
\end{align}
where we use an upper bound by Chua and Wheeden:
\begin{align}
    C_{{\rm CW},p} 
    := 
    \sup\limits_{ v \in C^\infty([0,1]) \setminus \bbR } 
    \frac{ 
        \| v - v_{[0,1]} \|_{L^{p}([0,1])} 
    }{ 
        \| \nabla v \|_{L^{p}([0,1])} 
    }
    \leq 
    \sqrt[p]{p} 2^{1-\frac 1 p}
    =
    2
    \left( \frac p 2 \right)^{\frac 1 p}
    .
\end{align} 
Note that~\eqref{math:chuawheeden} is generally not optimal among the upper bounds that only depend on the domain diameter and the Lebesgue exponent.
As discussed above, these Poincar\'e inequalities imply Poincar\'e--Friedrichs inequalities.

We know optimal Poincar\'e--Friedrichs constants over bounded convex domains~(\cite[Theorem~1.1]{ferone2012remark},~\cite[Theorem~1.1]{esposito2013poincare}): 
when $1 < p < \infty$, one can show that 
\begin{align}\label{math:esposito:pf}
    \min\limits_{ c \in \bbR }
    \| u - c \|_{L^{p}(\Omega)}
    \leq 
    C_{{\rm EFNT},p}
    \diam(\Omega)
    \| \nabla u \|_{L^{p}(\Omega)}
    ,
    \quad 
    \forall 
    u \in W^{1,p}(\Omega)
    ,
\end{align}
where $C_{{\rm EFNT},p}$ is the best possible constant that only depends on $p$ and equals 
\begin{align*}
    C_{{\rm EFNT},p}
    :=
    \frac{ p \sin(\pi/p) }{ 2\pi \sqrt[p]{p-1} }
    .
\end{align*}
Note that the last inequalities from~\eqref{math:p_vs_pf} imply, again when $1 < p < \infty$, the Poincar\'e inequalities
\begin{align}\label{math:esposito:poincare}
    \| u - u_{\Omega} \|_{L^{p}(\Omega)}
    \leq 
    2^{|1-\frac 2 p|}C_{{\rm EFNT},p}
    \diam(\Omega)
    \| \nabla u \|_{L^{p}(\Omega)}
    ,
    \quad 
    \forall 
    u \in W^{1,p}(\Omega)
    .
\end{align}
When $p=1$, then the optimal Poincar\'e constant also bounds the Poincar\'e--Friedrichs constant:
\begin{align}\label{math:l1estimate}
    \min\limits_{ c \in \bbR }
    \| u - c \|_{L^{1}(\Omega)}
    \leq 
    \frac{ \diam(\Omega) }{2}
    \| \nabla u \|_{L^{1}(\Omega)}
    ,
    \quad 
    \forall 
    u \in W^{1,1}(\Omega)
    .
\end{align}
When $p=\infty$, since bounded convex domains are Lipschitz domains, Rademacher's theorem leads to 
\begin{align}\label{math:lipschitzestimate}
    \min\limits_{ c \in \bbR }
    \| u - c \|_{L^{\infty}(\Omega)}
    \leq 
    \diam(\Omega)
    \| \nabla u \|_{L^{\infty}(\Omega)}
    ,
    \quad 
    \forall 
    u \in W^{1,\infty}(\Omega)
    .
\end{align}

\begin{remark}
Any estimate for the Poincar\'e--Friedrichs constant implies an estimate for the Poincar\'e constant, via~\eqref{math:p_vs_pf}. 
    Let us compare $C_{{\rm CW},p}$ for the Poincar\'e inequality
    with $C_{{\rm EFNT},p}$ for the Poincar\'e--Friedrichs inequality. 
    In the case $2 \leq p$, 
    \begin{align*}
        2^{1-\frac 2 p}
        C_{{\rm EFNT},p}
        =
        \frac{ 2^{1-\frac 2 p} }{2}
        \frac{ \sin(\pi/p) }{ \pi/p }
        \frac{ 1 }{ \sqrt[p]{p-1} }
        \leq 
        4^{-\frac 1 p}
        \leq
        C_{{\rm CW},p} 
        .
    \end{align*}
    In the case $p \leq 2$,
    \begin{align*}
        2^{\frac 2 p - 1}
        C_{{\rm EFNT},p}
        =
        \frac{ 2^{\frac 2 p - 1} }{2}
        \frac{ \sin(\pi/p) }{ \pi/p }
        \frac{ 1 }{ \sqrt[p]{p-1} }
        \leq
        \frac{ 2^{\frac 2 p-1} }{2}
        &
        =
        2^{\frac 2 p - 2}
=
        4^{\frac 1 p - 1}
        \leq
        C_{{\rm CW},p} 
.
    \end{align*}
    It follows that~\eqref{math:esposito:poincare} is generally a tighter estimate than~\eqref{math:chuawheeden} for $1 < p < \infty$.
\end{remark}

\begin{remark}
    The above Poincar\'e and Poincar\'e--Friedrichs constants are optimal for the class of bounded convex domains, 
    but individual bounded convex domains may allow for better constants.
    We refer to~\cite{Liu_Kik_interp_10,Cars_Ged_Rim_expl_cnst_12,matculevich2016explicit} for discussions;
    for example, triangles allow the reduction of the constant by 20\%.
\end{remark}

\section{Basic notions of triangulations}\label{section:triangulations}

We gather basic notions and definitions concerning simplicial meshes.

A ${k}$-dimensional \emph{simplex} $T$ is the convex hull of ${k}+1$ affinely independent points $v_0, v_1, \ldots, v_{{k}} \in \mathbb{R}^{n}$. We call these points the \emph{vertices} of the simplex $T$.
The strictly positive convex combinations of the vertices of the simplex constitute the \emph{ interior} of the simplex,
and its remaining points constitute the \emph{boundary} of the simplex.
If $S$ is a simplex whose vertices are also vertices of another simplex $T$, in which case $S \subseteq T$,
then we call $S$ a \emph{subsimplex} of $T$ and call $T$ a \emph{supersimplex} of $S$.

A finite family of simplices $\calT$ is a \emph{simplicial complex} or \emph{triangulation} if it satisfies the following conditions:
(i) $\calT$ contains all the subsimplices of its members (ii) any non-empty intersection of two members of $\calT$ is a common subsimplex of each.
We say that a simplicial complex $\calT$ has \emph{dimension $n$} or is \emph{$n$-dimensional} if each of its simplices is a subset of an $n$-dimensional member of that triangulation.\footnote{Simplicial complexes that we call $n$-dimensional are called purely $n$-dimensional in the literature on polytopes~(cf.~\cite{ziegler1995lectures}) and simply ``simplicial meshes'' in the finite element literature.}
We also write $\underlying{\calT}$ for the \emph{underlying set} of the simplicial complex $\calT$,
which is the union $\underlying{\calT} = \bigcup \calT$ of all simplices in $\calT$.
One calls any set triangulable if it is the underlying set of some triangulation.

Given any simplex $T$, we write $\subsimplex(T)$ for the simplicial complex that contains all subsimplices of $T$,
and $\subsimplex_{{k}}(T) \subseteq \subsimplex(T)$ denotes the set of $k$-dimensional subsimplices of $T$. 
We write $\Vertices(T) := \subsimplex_{0}(T)$ for the set of vertices of $T$.
Whenever $\calT$ is a simplicial complex, 
the set of $k$-dimensional simplices in $\calT$ is denoted as $\subsimplex_{{k}}(\calT)$. 
Similarly, the notations $\Vertices(\calT) := \subsimplex_{0}(\calT)$ and $\Faces(\calT) := \subsimplex_{n-1}(\calT)$ refer to the vertices and the faces (that is, members with codimension one) of this triangulation.\footnote{
    Our use of the term \emph{face} as is common in classical geometry and the finite element literature~\cite{brenner2008mathematical}
    and is synonymous with \emph{facet} as used in the literature on polyhedral combinatorics~\cite{schrijver1998theory}.
    Notably, this terminology differs from the uses \emph{face} and \emph{facet} in the theory of polyhedra~\cite{ziegler1995lectures}.
}
In practice, we may identify points and singleton simplices.

When $\calT$ is a triangulation and $T \in \calT$, then $\patch_{\calT}(T)$ denotes the \emph{local patch} or \emph{local star} of $T$, 
which is the simplicial subcomplex of $\calT$ that contains all supersimplices of $T$ and their subsimplices. 
We write $\carapace_{\calT}(T)$ for the subset of the local patch whose members do not contain $T$ itself. 
Formally,
\begin{gather*}
\patch_{\calT}(T) := \bigcup_{ \substack{ T' \in \subsimplex_{n}(\calT) \\ T \subseteq T' } } \subsimplex(T'),
    \qquad 
    \carapace_{\calT}(T) := \bigcup_{ \substack{ T' \in \patch_{\calT}(T) \\ T \nsubseteq T' } } \subsimplex(T').
\end{gather*}
We also write $\UPatch_T := | \patch_{\calT}(T) |$ for the closed underlying set of the local patch. 
A crucial structural observation is the following.

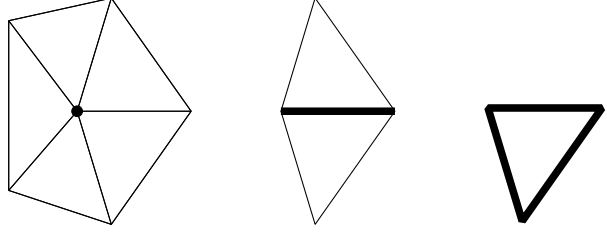
\begin{figure}[t]
\caption{From left to right: local patches around a vertex, and edge, and a triangle. The local patch of any full-dimensional simplex only consists of that simplex itself (and its subsimplices).}\label{figure:patches}
\centering
\begin{tabular}{ccc}
  \begin{tikzpicture}
[line join=bevel,x={( 1.5cm, 0mm)},y={( 0mm, 1.5cm)},z={( 1.5*3.85mm, -1.5*3.85mm)}]
    
    \coordinate (O) at (  0.0,  0.0, 0.0);
    \coordinate (A) at (  1.0,  0.0, 0.0);
    \coordinate (B) at (  0.3,  1.0, 0.0);
    \coordinate (C) at ( -0.6,  0.8, 0.0);
    \coordinate (D) at ( -0.6, -0.7, 0.0);
    \coordinate (E) at (  0.3, -1.0, 0.0);
    \coordinate (F) at (  0.3, -1.0, 0.0);
    
    \draw (O) -- (A) -- (B) -- cycle;
    \draw (O) -- (B) -- (C) -- cycle;
    \draw (O) -- (C) -- (D) -- cycle;
    \draw (O) -- (D) -- (E) -- cycle;
    \draw (O) -- (E) -- (A) -- cycle;
    \draw (A) -- (B) -- (C) -- (D) -- (E) -- cycle;
    \draw (0,0) circle[radius=2pt];
    \fill (0,0) circle[radius=2pt];
  \end{tikzpicture}\qquad&\qquad
  \begin{tikzpicture}
[line join=bevel,x={( 1.5cm, 0mm)},y={( 0mm, 1.5cm)},z={( 1.5*3.85mm, -1.5*3.85mm)}]
    
    \coordinate (O) at (  0.0,  0.0, 0.0);
    \coordinate (A) at (  1.0,  0.0, 0.0);
    \coordinate (B) at (  0.3,  1.0, 0.0);
    \coordinate (C) at ( -0.6,  0.8, 0.0);
    \coordinate (D) at ( -0.6, -0.7, 0.0);
    \coordinate (E) at (  0.3, -1.0, 0.0);
    \coordinate (F) at (  0.3, -1.0, 0.0);
    
    \draw (O) -- (A) -- (B) -- cycle;
    \draw (O) -- (E) -- (A) -- cycle;
    \draw[line width=0.1cm] (O) -- (A);
  \end{tikzpicture}\qquad&\qquad
  \begin{tikzpicture}
[line join=bevel,x={( 1.5cm, 0mm)},y={( 0mm, 1.5cm)},z={( 1.5*3.85mm, -1.5*3.85mm)}]
    
    \coordinate (O) at (  0.0,  0.0, 0.0);
    \coordinate (A) at (  1.0,  0.0, 0.0);
    \coordinate (B) at (  0.3,  1.0, 0.0);
    \coordinate (C) at ( -0.6,  0.8, 0.0);
    \coordinate (D) at ( -0.6, -0.7, 0.0);
    \coordinate (E) at (  0.3, -1.0, 0.0);
    \coordinate (F) at (  0.3, -1.0, 0.0);
    
    \draw[line width=0.1cm] (O) -- (A) -- (E) -- cycle;
  \end{tikzpicture}
\end{tabular} 
\end{figure}

\begin{lemma}
 Let $\calT$ be an $n$-dimensional simplicial complex and let $S,S' \in \calT$.
Then either $\patch_{\calT}(S)$ and $\patch_{\calT}(S')$ 
 only intersect in simplices of dimension at most $n-1$ or there exists $S'' \in \calT$
 such that 
 \begin{align*}
    \patch_{\calT}(S) \cap \patch_{\calT}(S') = \patch_{\calT}(S''),
    \qquad 
    \Vertices(S) \cup \Vertices(S') = \Vertices(S'').
 \end{align*}
\end{lemma}
\begin{proof}
 Let $T \in \calT$ be $n$-dimensional.
 We have $T \in \patch_{\calT}(S )$ if and only if all vertices of $S $ are vertices of $T$.
 Similarly, $T \in \patch_{\calT}(S')$ if and only if all vertices of $S'$ are vertices of $T$.
 Consequently, $T \in \patch_{\calT}(S) \cap \patch_{\calT}(S')$ if and only if $T \in \patch_{\calT}(S'')$,
 where $S'' \in \calT$ is the convex closure of $S$ and $S'$.
\end{proof}

We introduce a specific notion of connectivity when we are given an $n$-dimensional simplicial complex $\calT$. 
We call two $n$-dimensional simplices $S, S' \in \calT$ \emph{face-neighboring} if $S \cap S'$ is a common face of both of them. 
We call $n$-simplices $S,S' \in \calT$ \emph{face-connected in $\calT$} if there exists a sequence $S=S_0,S_1,\dots,S_m=S' \in \calT$ such that $S_{i}$ and $S_{i-1}$ are face-neighboring for all $1 \leq i \leq m$. 
Such a sequence is called a \emph{face path} from $S$ to $S'$ in $\calT$. 
Clearly, face-connected in $\calT$ is an equivalence relation among simplices. 
A \emph{face-connected component} of $\calT$ is an equivalence class under this equivalence relation, 
and we call $\calT$ \emph{face-connected} if it has only one face-connected component.

\subsection{Shape measures and related quantities}

We introduce several quantities that measure the regularity of a triangulation. 
These have in common that they can be computed from purely local information. 

We write $\diam(T)$ and $\vol(T)$ for the diameter and $n$-dimensional volume of any $n$-simplex $T$.
Moreover, $\height(T)$ refers to the smallest height of any vertex of the simplex $T$,
where the height of a vertex is defined as the distance to the affine span of its opposing face.
For the purpose of the usual scaling arguments,
the \emph{$n$-dimensional reference simplex} $\Delta^n \subseteq \bbR^n$ is the convex closure of the origin and the $n$ canonical unit vectors.

Whenever $T$ is any $n$-dimensional simplex $T$,
we define the \emph{aspect shape measure} $\aspectratio(T)$,
and 
the \emph{algebraic shape measure} $\algebraicshapemeasure(T)$
by 
\begin{align}\label{math:shapemeasure}
    \aspectratio(T)
    := 
    \frac{ \diam(T) }{ \height(T) }
    ,
    \qquad 
    \algebraicshapemeasure(T)
    := 
    \sup_{ \varphi : \Delta^n \rightarrow T } 
\matnorm{\Jacobian \varphi} \matnorm{ \Jacobian \varphi^{-1} }
    ,
\end{align}
where the last supremum is taken over all affine transformation from the reference $n$-simplex onto the $n$-simplex $T$. 
When $\calT$ is an $n$-dimensional simplicial complex, we naturally define 
\begin{align}\label{math:shapemeasure:triangulation}
    \aspectratio(\calT) := \sup_{ T \in \subsimplex_{n}(\calT) } \aspectratio(T)
    ,
    \quad 
    \algebraicshapemeasure(\calT) := \sup_{ T \in \subsimplex_{n}(\calT) } \algebraicshapemeasure(T)
    .
\end{align}
We call these the aspect and algebraic shape measure, respectively, of the triangulation. 

\begin{remark}
    The ratio $\aspectratio(T)$ measures the ``shape quality'' of an $n$-dimensional simplex $T$ and is an instance of a so-called \emph{shape measure}.
    For example, the reference triangle has aspect shape measure $2$ and the reference tetrahedron has aspect shape measure $\sqrt{6}$. 
    Numerous alternative shape measures have been used throughout the literature of numerical analysis and computational geometry to quantify the quality of simplices (see~\cite[p.61, Definition 5.1]{braess2001finite}, 
   ~\cite[p.97, Definition {\char`\(}4.2.16{\char`\)} ]{brenner2008mathematical}, \cite[Definition~11.2]{ern2021finite}). 
\end{remark}

We gather a few relationships between geometric and algebraic entities and compare the different shape measures of a single simplex.

\begin{lemma}\label{lemma:measurerelationships}
    Let $T$ be an $n$-simplex and let $\varphi : \Delta^{n} \rightarrow T$ be an affine diffeomorphism from the reference $n$-simplex. Then 
    \begin{gather*} \matnorm{ \Jacobian \varphi } 
        \leq 
        \Ceins{n} \cdot \diam(T),
        \qquad 
        \matnorm{ \Jacobian \varphi^{-1} } 
        \leq 
        \Czwei{n} \cdot \aspectratio(T) \diam(T)^{-1}
        ,
        \\
        \frac{1}{\sqrt{2n}} \aspectratio(T) \leq \algebraicshapemeasure(T) \leq n \aspectratio(T).
    \end{gather*}
    Here, $\Ceins{n}, \Czwei{n} \leq \sqrt n$. \end{lemma}
\begin{proof}
    Let $\varphi : \Delta^{n} \rightarrow T$ be an affine transformation. 
    We abbreviate $J := \Jacobian\phi$ for its Jacobian.
    We begin with observing that the largest $\ell^{2}$-norm of any column of $J$,
    which we denote here by $c_{\max}(J)$,
    equals the maximum of the quotient $\| J x \|_{\ell^{2}} / \| x \|_{\ell^{1}}$ over all non-zero $x \in \bbR^{n}$.
    The diameter of $T$ is the length of its longest edge. 
    The first inequality is implied by the following comparisons:
    \begin{align*}
        \frac{\diam(T)}{\sqrt 2} \leq \matnorm{ J } \leq \sqrt{n} \cdot c_{\max}(J) \leq \sqrt{n} \cdot \diam(T).
    \end{align*}
    The columns of the matrix $J^{-1}$ are the gradients of the barycentric coordinates of the vertices of $T$,
    except for the vertex $\varphi(0) \in T$. It immediately follows that 
    \begin{align*}
        \matnorm{ J^{-1} } \leq \sqrt{n} \cdot c_{\max}(J^{-1}) \leq \sqrt{n} \cdot \height(T)^{-1}
        .
    \end{align*}
    The smallest height in the reference simplex $\Delta^{n}$ is $\height_{\Delta} = 1/\sqrt{n}$, 
    whence $\height(T) \geq \matnorm{J^{-1}}^{-1} / \sqrt{n}$. This yields the second inequality.
Lastly, the last comparison is now obvious.
\end{proof}

We will also need the maximal ratio of volumes between face-neighboring $n$-simplices, written $\volumeratio(\calT)$,
and the ratio of the diameters of any intersecting simplices, written $\diameterratio(\calT)$. Formally, 
\begin{align}\label{math:volumeratio}
    \volumeratio(\calT) 
    := 
    \sup_{ \substack{ T, T' \in \subsimplex_{n}(\calT) \\ T \cap T' \in \subsimplex_{n-1}(\calT) } } 
\vol(T) / \vol(T')
    ,
\end{align}
\begin{align}\label{math:diameterratio}
    \diameterratio(\calT) 
    := 
    \sup_{ \substack{ T, T' \in \subsimplex_{n}(\calT) \\ T \cap T' \neq \emptyset } }
\diam(T) / \diam(T')
    .
\end{align}
Finally, whenever $T, T'$ are two $n$-simplices that share a common face $F$ of codimension $1$, we let $\Xi_{T,T'} : T \rightarrow T'$ denote the affine diffeomorphism 
that preserves $F$. We then define 
\begin{align}\label{math:reflectionmeasure}
   ~\reflectionestimate(\calT) 
    := 
    \sup_{ \substack{ T, T' \in \subsimplex_{n}(\calT) \\ T \cap T' \in \subsimplex_{n-1}(\calT) } }
    \matnorm{ \Jacobian \Xi_{T,T'} }
\end{align}
to be the maximum of the operator norm of the Jacobian of any such diffeomorphism. 
This indicator quantifies how much reflection across the shared face distorts the geometry. 

\begin{lemma}\label{lemma:volumecomparison}
    Let $T_1$ and $T_2$ be two $n$-simplices that share a common face $F$. Then 
    \begin{align*}
        \diam(T_1)
        \leq 
\aspectratio(T_1)
        \diam(F)
        \leq 
        \aspectratio(T_1)
        \diam(T_2)
        ,
        \qquad 
        \frac{ \vol(T_1) }{ \vol(T_2) }
\leq 
        \frac{ \diam(T_1) }{ \diam(T_2) } \aspectratio(T_2)
        \leq 
        \aspectratio(T_1) \aspectratio(T_2)
        .
    \end{align*}
    If $\Xi : T_1 \rightarrow T_2$ is the affine diffeomorphism that is the identity over $F$, then 
    at least $n-2$ of its Jacobian's singular values equal $1$, and we have 
    \begin{align*}
        \matnorm{ \Jacobian \Xi^{  } } 
        &
        \leq 
        \frac{1}{2} 
\sqrt{ \left( \frac{ \diam(T_2) }{ \diam(T_1) } \aspectratio(T_{1}) + 1 \right)^{2} + \aspectratio(T_{1})^{2} }
        +
        \frac{1}{2} 
        \sqrt{ \left( \frac{ \diam(T_2) }{ \diam(T_1) } \aspectratio(T_{1}) - 1 \right)^{2} + \aspectratio(T_{1})^{2} }
        ,
        \\
        \matnorm{ \Jacobian \Xi^{-1} } 
        &
        \leq 
        \frac{1}{2} 
\sqrt{ \left( \frac{ \diam(T_1) }{ \diam(T_2) } \aspectratio(T_{2}) + 1 \right)^{2} + \aspectratio(T_{2})^{2} }
        +
        \frac{1}{2} 
        \sqrt{ \left( \frac{ \diam(T_1) }{ \diam(T_2) } \aspectratio(T_{2}) - 1 \right)^{2} + \aspectratio(T_{2})^{2} }
        ,
        \\
        \det(\Jacobian \Xi)
        &
        =
        \frac{ \vol(T_2) }{ \vol(T_1) }
        \leq 
        \frac{ \diam(T_2) }{ \diam(T_1) } \aspectratio(T_1)
        .
    \end{align*}
\end{lemma}
\begin{proof}
    The diameter of $F$ is at least as large as the height $h_S$ of some other vertex of $F$ in $T_{1}$.
    To see this, recall that the height of a vertex is minimum distance to the affine hull of the opposing subsimplex.
    Naturally, the height of the vertex in $T_{1}$ will be smaller than the height in $F$.
    Now, 
    \begin{align*}
\diam(T_{1}) \aspectratio(T_{1})^{-1} = \height(T_{1}) \leq h_S \leq \diam(F).
    \end{align*}
    The first estimate follows. 
    As for the second estimate,
    let $h_1, h_2 > 0$ be the heights of $F$ in the simplices $T_1$ and $T_2$, respectively.
    By the volume formula for simplices, $\vol(T_1) = h_1 \vol(F)/n$ and $\vol(T_2) = h_2 \vol(F)/n$, and thus $\vol(T_1)/\vol(T_2) = h_1/h_2$. Thus follows the second estimate: 
    \begin{align*}
        \frac{ \vol(T_1) }{ \vol(T_2) }
        = 
        \frac{ h_1 }{ h_2 }
        \leq
        \frac{ \diam(T_1) }{ h_2 }
        \leq
        \aspectratio(T_{1}) \frac{ \diam(F) }{ h_2 }
        \leq
        \aspectratio(T_{1}) \aspectratio(T_{2})
        ,
        \qquad 
        \frac{ \diam(T_1) }{ h_2 }
        =
        \frac{ \diam(T_1) }{ \diam(T_2) } \aspectratio(T_{2})
        .
    \end{align*}
    Lastly, let $\Xi : T_1 \rightarrow T_2$ be as stated. 
    To estimate the Lipschitz constant of $\Xi$, we study the singular values of its Jacobian. 
    Without loss of generality, $F$ lies in the span of the first $n-1$ coordinates and contains the origin. 
    Write $z_1 \in T_1$ and $z_2 \in T_2$ for the two vertices not contained in $F$.
    There exists a unit height vector $\hat h_0$ on the $n$-th coordinate direction
    such that there exist $b_1, b_2 \in \bbR^{n}$ in the affine hull of $F$
    satisfying $z_1 = b_1 + h_1 \hat h_0$ and $z_2 = b_2 - h_2 \hat h_0$.
    The mapping $\Xi$ is linear, being the identity over $F$ and mapping $z_1$ to $z_2$. Hence,
    \begin{align*}
        \Xi( \hat h_0 ) = h_1^{-1} ( b_2 - b_1 ) - h_2 h_{1}^{-1} \hat h_0.
    \end{align*}
    We see that $\Xi$ equals the identity over the orthogonal complement of the span of $\hat h_0$ and $b_2 - b_1$.
    The only singular values of its Jacobian are the two singular values $\sigma_{-} \leq 1 \leq \sigma_{+}$ of the matrix 
    \begin{align*}
        \begin{pmatrix}
            a & 0 
            \\
            c & 1
        \end{pmatrix},
        \qquad 
        a = - h_2 / h_1,
        \qquad 
        c = \vecnorm{ b_2-b_1 } / h_1
        .
    \end{align*}
    These are 
    \begin{align*}
        \sigma_{\pm} 
        &
        = 
        \frac{1}{\sqrt{2}} \sqrt{ 1 + a^2 + c^2 \pm \sqrt{ \left( 1 + a^2 + c^2 \right)^{2} - 4a^{2} } } 
        \\&
        =
        \frac{1}{\sqrt{2}} \sqrt{ 1 + a^2 + c^2 \pm \sqrt{ \left( (1+a)^2 + c^2 \right)\left( (1-a)^2 + c^2 \right) } } 
        \\&
        =
        \frac{1}{2} 
        \left( \sqrt{ (1 + a)^2 + c^2 } \pm \sqrt{ (1 - a)^2 + c^2 } \right)
        .
    \end{align*}
    We notice that this provides the upper bound
    \begin{align*}
        \sigma_{+} 
        &
        \leq 
        \frac{1}{2} 
\sqrt{ \left( \frac{ \diam(T_2) }{ \diam(T_1) } \aspectratio(T_{1}) + 1 \right)^{2} + \aspectratio(T_{1})^{2} }
        +
        \frac{1}{2} 
        \sqrt{ \left( \frac{ \diam(T_2) }{ \diam(T_1) } \aspectratio(T_{1}) - 1 \right)^{2} + \aspectratio(T_{1})^{2} }
        .
    \end{align*}
    The desired estimates are shown.
\end{proof}

\begin{remark}
    While we will utilize $\diameterratio(\calT)$ at numerous places throughout the manuscript, 
    $\volumeratio(\calT)$ and $\reflectionestimate(\calT)$ will only be used throughout the following Section~\ref{section:gradient}.
    Lemma~\ref{lemma:volumecomparison} expresses that $\volumeratio(\calT)$ of~\eqref{math:volumeratio} is controlled by the shape measure.
    In a face-connected triangulation where we have an upper bound for the number of simplices sharing a vertex,
    this Lemma also indicates, at least in principle, control of $\diameterratio(\calT)$ of~\eqref{math:diameterratio}.
\end{remark}

\section{Poincar\'e--Friedrichs inequalities over triangulated domains}\label{section:gradient}

In this section, we develop stepwise computable estimates for Poincar\'e--Friedrichs constants of triangulated domains.
The following very classical procedure serves as our inspiration: given a gradient vector field, we can reconstruct a gradient potential by fixing a starting point and integrating the gradient field along lines emanating from that point; the potential is unique up to an additive constant.
We perform a discrete analogue of this procedure over triangulated domains:
having fixed a starting simplex, we first define the scalar potential on this simplex.
Step by step, we traverse the triangulation along face-neighboring simplices, each time extending the scalar potential to the next simplex and controlling its Lebesgue norm.
As we construct the scalar potential over increasingly larger subdomains, we always pick a gradient potential on the next simplex and fix the integration constant using the value already known along the connecting face, thereby ensuring the correct Sobolev continuity.
This basic idea has appeared in various forms before, for instance recently in~\cite{Brae_Pill_Sch_p_rob_09,ern2020stable,Chaum_Voh_p_rob_3D_H_curl_24,Voh_loc_glob_H1_25}.

We begin with an auxiliary result with independent relevance, where we estimate the Poincar\'e--Friedrichs inequality when homogeneous boundary conditions hold along at least one face of the boundary.

\begin{lemma}\label{lemma:mixedbconsimplex}
    Let $T$ be an $n$-simplex with a face $F$ and $p \in [1,\infty]$. 
    If $u \in W^{1,p}(T)$ with $\trace_{F} u = 0$, then 
    \begin{align*}
        \| u \|_{L^{p}(T)}
        &
        \leq 
        C_{\PF,T,F,p} \| \nabla u \|_{L^{p}(T)}
        ,
    \end{align*}
    where $C_{\PF,T,F,p} = p^{-\frac 1 p} \diam(T)$ for $p < \infty$ and $C_{\PF,T,F,\infty} = \diam(T)$. 
\end{lemma}
\begin{proof}
Since the inequality follows from Rademacher's theorem in the limit case $p = \infty$,
    we assume $1 \leq p < \infty$. 
    Let $u \in C^{\infty}(T)$ have support disjoint from $F$.
    We tacitly extend this by zero to a function $u \in \Lebesgue^{\infty}(\bbR^{n})$. 
Without loss of generality, 
    the segment from the midpoint of $F$ to the opposing vertex lies on the first coordinate axis,
    and the minimal first coordinate among all the points of $F$ equals $0$. 
    We write $\bfg$ for the trivial extension of $\nabla u$ over the entire $\bbR^{n}$.
Using the fundamental theorem of calculus and H\"older's inequality, 
    \begin{align*}
        \int_{T} |u(x)|^{p} \diff x \;dx
        &\leq
        \int_{\bbR^{n-1}} \int_{0}^{\diam(T)} |u( x_{1}, \overline x )|^{p} \;dx_{1} \;d\overline{x}
        \\&
        \leq
        \int_{\bbR^{n-1}} \int_{0}^{\diam(T)} \left| \int_{0}^{x_{1}} |\bfg( y, \overline x )| \,dy \right|^{p} \;dx_{1} \,d\overline{x}
        \\&
        \leq
        \int_{0}^{\diam(T)} x_{1}^{p-1} \int_{\bbR^{n-1}} \int_{0}^{x_{1}} \left| \bfg( y, \overline x ) \right|^{p} \;dy \,d\overline{x} \,dx_{1}
        \\&
        \leq
        \int_{0}^{\diam(T)} x_{1}^{p-1} \;dx_1 
        \cdot 
        \int_{T} \left| \bfg( y, \overline x ) \right|^{p} \;dy \,d\overline{x} 
        \leq
        \frac{\diam(T)^{p}}{p} \int_{T} \left| \nabla u( x ) \right|^{p} \,dx
.
    \end{align*}
    If $u \in W^{1,p}(T)$ has vanishing trace along $F$ but is not necessarily smooth, 
    then we conclude $\|u\|_{L^{p}(T)} \leq \diam(T) p^{-\frac 1 p} \|\nabla u\|_{L^{p}(T)}$ 
    from approximation via members of $C^{\infty}(T)$ whose support is disjoint from $F$. 
    We very briefly verify that density argument: 
    There exists an affine diffeomorphism $\varphi : \Delta^{n} \rightarrow T$ from the reference simplex onto $T$ that maps the convex closure of the $n$ unit vectors onto the face $F$.
    We let $\hat u := u \circ \varphi$.
    Let $\hat U$ be the unit ball of the $\ell^1$ metric, which contains $\Delta^{n}$.
    We let $\tilde u$ be the extension of $\hat u$ onto $\hat U$ by reflection across the coordinate axes.
    Then $\tilde u \in W^{1,p}_{0}(\hat U)$, and $\tilde u$ is the limit of a sequence $u_{m} \in C^{\infty}_{c}(\hat U)$.
    Now $u_{m} \circ \varphi^{-1} \in C^{\infty}(T)$ approximates $u$ within the Banach space $W^{1,p}(T)$
    and has the desired support property. 
\end{proof}

\begin{remark}\label{remark:mixedbconsimplex:hilbert}
    We can improve Lemma~\ref{lemma:mixedbconsimplex} in the special case $p=2$.
    The variational formulation of the Poincar\'e constant over a convex domain reveals that $C_{\PF,T,F,p}$ lies between the Poincar\'e constant without boundary conditions and with full boundary conditions. 
    In particular,
    \begin{align*}
        C_{\PF,T,F,2} 
        \leq 
        \frac{ \diam(T) }{ \pi }
    \end{align*}
    is an improved Poincar\'e inequality.
\end{remark}

The next auxiliary result establishes Poincar\'e--Friedrichs constants over face patches within simplicial triangulations. 
We emphasize that face patches are not necessarily convex, but we can still extend the results on convex domains from Section~\ref{subsection:PX_convex}. The following result is reasonably sharp when the two simplices have similar volumes and diameters.

\begin{lemma}\label{lemma:poincarefriedrichsoverfacepatch}
    Let $\calT$ be a triangulation.
    Let $T_{1}, T_{2} \in \calT$ be two $n$-simplices whose intersection is a common face $F := T_1 \cap T_2$. 
    Write $U_{F} := T_1 \cup T_2$.
    If $p \in [1,\infty]$ and $u \in W^{1,p}(\Omega)$,
    then 
    \begin{align*}
        \min\limits_{ c \in \bbR }
        \| u - c \|_{L^{p}(U_{F})}
        &
\leq 
        C_{{\PF},T_1 \cup T_2,p}
        \| \nabla u \|_{L^{p}(U_{F})}
        .
    \end{align*}
    Here, 
    \[ 
        C_{{\PF},T_1 \cup T_2,p} \leq 2 \Ceins{n} C_{{\rm EFNT},p} \volumeratio(\calT)^{\frac 1 p} \max\left( \diam(T_1), \diam(T_2) \right).
    \]
    If $U_{F}$ is convex, then $C_{{\PF},T_1 \cup T_2,p} \leq C_{{\rm EFNT},p} \diam( U_{F} )$.
\end{lemma}
\begin{proof}
    The convex case follows immediately from~\eqref{math:esposito:pf}.
    Otherwise, without loss of generality, $F$ has the vertices $v_0, \dots, v_{n-1}$,
    and $z_{1} \in T_{1}$ and $z_{2} \in T_{2}$ are the remaining vertices of the two triangles. 
    Let $\Delta_{1} = \Delta^{n}$ be the reference $n$-simplex and let $\Delta_{2}$ be obtained from it by flipping the $n$-th coordinate.
    We let $\varphi_{1} : \Delta_{1} \rightarrow T_{1}$ and $\varphi_{2} : \Delta_{2} \rightarrow T_{2}$
    be affine transformations 
    that map the origin to $v_0$, that map each unit vector $e_{i}$ to $v_{i}$ for $i = 1, \dots, n-1$,
    and that satisfy $\varphi_{1}(e_{n}) = z_{1}$ and $\varphi_{2}(-e_{n}) = z_{2}$.
    Write $\hat U := \Delta_1 \cup \Delta_2$.
    We have a bi-Lipschitz mapping $\varphi : \hat U \rightarrow U_{F}$.
    
    Suppose that $u \in W^{1,p}(U_{F})$. Then $\hat u := u \circ \varphi \in W^{1,p}(\hat U)$.
    We observe 
\begin{align*}
        \| \nabla \hat u \|_{L^{p}(\hat U)}
        &\leq 
        \max\left( 
            |\det(\Jacobian \varphi_1)|^{-\frac 1 p} \matnorm{ \Jacobian \varphi_1 }
            ,
            |\det(\Jacobian \varphi_2)|^{-\frac 1 p} \matnorm{ \Jacobian \varphi_2 }
        \right)
        \| \nabla u \|_{L^{p}(U_{F})}
        .
    \end{align*}
    Notice that $\hat U$ has diameter $2$ and is (crucially) convex. 
    Thus, due to Poincar\'e--Friedrichs inequality~\eqref{math:esposito:pf}, there exists $\hat w \in W^{1,p}(\hat U)$ such that $\nabla \hat w = \nabla \hat u$ and 
    \begin{align*}
        \| \hat w \|_{L^{p}(\hat U)}
        \leq 
        2 C_{{\rm EFNT},p}
        \| \nabla \hat u \|_{L^{p}(\hat U)}
        .
    \end{align*}
    Next, setting $w := \hat w \circ \varphi^{-1}$, we find $\nabla w = \nabla u$ and 
    \begin{align*}
        \| w \|_{L^{p}(U_{F})}
        \leq 
        \max\left( 
            |\det(\Jacobian \varphi_1)|
            ,
            |\det(\Jacobian \varphi_2)|
        \right)^{\frac 1 p}
        \| \hat w \|_{L^{p}(\hat U)}
        .
    \end{align*}
    For both $i = 1,2$, we now recall the well-known equation $|\det(\Jacobian\varphi_i)| = n! \vol(T_i)$, 
and the estimate $\matnorm{ \Jacobian \varphi_i } \leq \Ceins{n} \diam(T_i)$, which is given in Lemma~\ref{lemma:measurerelationships}. We obtain the desired result.
\end{proof}

The main result of this section constructs a potential and gives an upper bound for the Poincar\'e--Friedrichs constant.  
It follows the same underlying principle as the ``discrete mean Poincar\'e inequality'' of~\cite[Lemma~3.7]{Eym_Gal_Her_00}.
This procedure serves as the blueprint for constructing potentials of the curl and divergence operators in later sections. 

Two different variations of the underlying idea are analyzed, yielding slightly different estimates. 
In the sequential sweep over the simplices in the triangulation $\calT$, we can either employ Poincar\'e--Friedrichs inequalities over the individual simplices while relying on gluing with the previously visited face neighbor via Lemma~\ref{lemma:mixedbconsimplex}, or we consider the face patch (star) formed by the current and the previous simplex and employ Lemma~\ref{lemma:poincarefriedrichsoverfacepatch}.
Both estimates of Poincar\'e--Friedrichs constants capture the correct asymptotic behavior as $p$ grows to infinity.
We first prepare a recursion result in the forthcoming Theorem~\ref{theorem:poincarefriedrichsestimate:grad} that we subsequently convert in our final estimate in Theorem~\ref{theorem:fullrecursivesum:grad}.

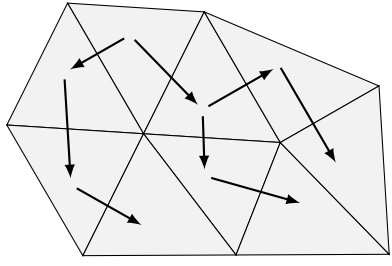
\begin{figure}[t]
\caption{Face-connected triangulation of a domain. The arrows depict a spanning tree in the face-connection graph.}\label{figure:spanningtree}
\begin{center}
\begin{tikzpicture}[rotate=-60]
    \coordinate (A) at (0,0);
    \coordinate (B) at (2,0);
    \coordinate (C) at (1,1.5);
    \coordinate (D) at (3,1.5); 
    \coordinate (E) at (3,3); 
    \coordinate (F) at (4,0.25); 
    \coordinate (X) at (1,-1.5);
    \coordinate (Y) at (3,-1.5); 
    \coordinate (T) at (5,2); 
    
    \coordinate (CentroidABC) at ($(A)!0.3333!(B)!0.3333!(C)$);
    \coordinate (CentroidBCD) at ($(B)!0.3333!(C)!0.3333!(D)$);
    \coordinate (CentroidCDE) at ($(C)!0.3333!(D)!0.3333!(E)$);
    \coordinate (CentroidBDF) at ($(B)!0.3333!(D)!0.3333!(F)$);
    \coordinate (CentroidABX) at ($(A)!0.3333!(B)!0.3333!(X)$);
    \coordinate (CentroidBXY) at ($(Y)!0.3333!(B)!0.3333!(X)$);
    \coordinate (CentroidBFY) at ($(Y)!0.3333!(B)!0.3333!(F)$);
    \coordinate (CentroidDFT) at ($(D)!0.3333!(F)!0.3333!(T)$);
    \coordinate (CentroidDET) at ($(D)!0.3333!(E)!0.3333!(T)$);
    
    \draw[fill=gray!10] (A) -- (B) -- (C) -- cycle; 
    \draw[fill=gray!10] (B) -- (C) -- (D) -- cycle; 
    \draw[fill=gray!10] (C) -- (D) -- (E) -- cycle; 
    \draw[fill=gray!10] (B) -- (D) -- (F) -- cycle; 
    \draw[fill=gray!10] (A) -- (B) -- (X) -- cycle; 
    \draw[fill=gray!10] (Y) -- (B) -- (X) -- cycle; 
    \draw[fill=gray!10] (Y) -- (B) -- (F) -- cycle; 
    \draw[fill=gray!10] (T) -- (D) -- (F) -- cycle; 
    \draw[fill=gray!10] (T) -- (D) -- (E) -- cycle; 
    
    \draw[-latex, thick, shorten <=2.5pt, shorten >=2.5pt] (CentroidABC) -- (CentroidBCD); 
    \draw[-latex, thick, shorten <=2.5pt, shorten >=2.5pt] (CentroidBCD) -- (CentroidCDE); 
    \draw[-latex, thick, shorten <=2.5pt, shorten >=2.5pt] (CentroidBCD) -- (CentroidBDF); 
    \draw[-latex, thick, shorten <=2.5pt, shorten >=2.5pt] (CentroidABC) -- (CentroidABX); 
    \draw[-latex, thick, shorten <=2.5pt, shorten >=2.5pt] (CentroidABX) -- (CentroidBXY); 
    \draw[-latex, thick, shorten <=2.5pt, shorten >=2.5pt] (CentroidBXY) -- (CentroidBFY); 
    \draw[-latex, thick, shorten <=2.5pt, shorten >=2.5pt] (CentroidBDF) -- (CentroidDFT); 
    \draw[-latex, thick, shorten <=1.0pt, shorten >=2.0pt] (CentroidCDE) -- (CentroidDET); 
\end{tikzpicture}
\end{center}
\end{figure}

\begin{theorem}\label{theorem:poincarefriedrichsestimate:grad}
    Let $\calT$ be a face-connected $n$-dimensional finite triangulation and let the domain $\Omega$ be the interior of the underlying set of $\calT$.
    Let $1 \leq p, q \leq \infty$ with $1 = 1/p + 1/q$.
    Then for any $u \in W^{1,p}(\Omega)$
    there exists $w \in W^{1,p}(\Omega)$ with $\nabla w = \nabla u$ 
    and satisfying the following estimates:
    \begin{enumerate}
    \item 
    There exists an $n$-simplex $T_0 \in \calT$ with 
    \begin{gather}\label{math::poincarefriedrichsestimate:grad:start}
        \| w \|_{L^{p}(T_{0})} \leq C_{{\PF},T_{0},p} \| \nabla u \|_{L^{p}(T_{0})}.
    \end{gather}
    \item 
    For any $n$-simplex $T_M \in \calT$, there exists a face path $T_0, T_1, \dots, T_M$
    such that for all $1 \leq m \leq M$, we have the two recursive estimates:
    \begin{align}
        \label{math::poincarefriedrichsestimate:grad:glue}
        \| w \|_{L^{p}(T_{m})}
        &
        \leq  
        A_{m}
        \| w \|_{L^{p}(T_{m-1})} 
        +
        B'_{m}
        \| \nabla u \|_{L^{p}(T_{m  })} 
        + 
        B''_{m}
        \| \nabla u \|_{L^{p}(T_{m-1})} 
        ,
        \\
        \label{math::poincarefriedrichsestimate:grad:patch}
        \| w \|_{L^{p}(T_{m})}
        &
        \leq  
        A_{m}
        \| w \|_{L^{p}(T_{m-1})} 
        +
        B^\star_{m}
        \| \nabla u \|_{L^{p}(\UPatch_{F_m})} , 
    \end{align}
    with the constants  
    \begin{gather*}
        A_{m} 
        \leq
        \volumeratio(\calT)^{\frac 1 p} 
        ,
        \quad 
        B'_{m} 
        \leq 
        C_{\PF,T_{m},F_{m},p} 
        ,
        \quad 
        B''_{m} 
        \leq 
        C_{\PF,T_{m},F_{m},p} \volumeratio(\calT)^{\frac 1 p}~\reflectionestimate(\calT)  ,
        \\
        B_{m}^\star
        \leq 
        \left( 1 + \volumeratio(\calT)^{\frac q p} \right)^{\frac 1 q}
        C_{{\PF},\UPatch_{F_\ell},p}  .
    \end{gather*}
    \end{enumerate}
    Here, $F_m := T_m \cap T_{m-1}$ and $U_m := T_m \cup T_{m-1}$ for any $1 \leq m \leq M$.
\end{theorem}

\begin{proof}
    Let $u \in W^{1,p}(\Omega)$. 
    We start with the Poincar\'e--Friedrichs inequality on the first simplex $T_{0}$. 
    Via~\eqref{math:inequalitypf:gradient}, 
    there exists $w_0 \in W^{1,p}(T_{0})$ satisfying $\nabla w_0 = \nabla u$ over $T_{0}$ together with 
    \begin{gather*}
        \| w_0 \|_{L^{p}(T_{0})} \leq C_{{\PF},T_{0},p} \| \nabla u \|_{L^{p}(T_{0})}.
    \end{gather*}
    In particular, $c_{0} := ( w_0 - u ) \restriction_{T_{0}}$ is a constant function. 
    We then define 
    \begin{align*}
        w := u + c_0
    \end{align*} 
    as our potential.
    Clearly, $w \in W^{1,p}(\Omega)$ with $\nabla w = \nabla u$.
    By construction $w\restriction_{T_{0}} = w_0$ and~\eqref{math::poincarefriedrichsestimate:grad:start} holds.
    In the rest of the proof, we verify that $w$ satisfies the desired recursive estimates~\eqref{math::poincarefriedrichsestimate:grad:glue} and~\eqref{math::poincarefriedrichsestimate:grad:patch}.
    Suppose that $T_0, T_1, \dots, T_M$ is a face path in $\calT$
    and that $1 \leq m \leq M$.
    Recall that we write $F_m := T_m \cap T_{m-1}$, which is a face of dimension $n-1$ shared by the $n$-simplices $T_{m}$ and $T_{m-1}$,
    and that we write $\UPatch_{F_m} := T_m \cup T_{m-1}$. 
    
    We study two constructions, beginning as follows. 
    We define $w'_{m} := w\restriction_{T_{m-1}} \circ \Xi \in W^{1,p}(T_{m})$,
    where $\Xi : T_{m} \rightarrow T_{m-1}$ is the unique affine diffeomorphism that leaves $F_{m}$ invariant. 
    By construction, $w'_{m} \in W^{1,p}(T_{m})$ with 
    \begin{align*}
        \trace_{F_{m}} w'_{m} = \trace_{F_{m}} w\restriction_{T_{m-1}}
        .
    \end{align*}
    We now define $w''_{m} \in W^{1,p}(T_{m})$ via 
    \begin{align}\label{equation:w''}
        w''_{m} 
        := 
w\restriction_{T_{m}} - w'_{m} = u\restriction_{T_{m}} - u\restriction_{T_{m-1}} \circ \Xi.
    \end{align}
    We crucially note that $w''_{m}$ is trace-free along $F_{m}$ since
    \begin{align*}
        \trace_{F_{m}} w''_{m} 
        = 
        \trace_{F_{m}} \left( 
            w\restriction_{T_{m}} - w'_{m} 
        \right) 
        =
        \trace_{F_{m}} w\restriction_{T_{m}}
        -
        \trace_{F_{m}} w\restriction_{T_{m-1}}
        =
        \trace_{F_{m}} u\restriction_{T_{m}}
        -
        \trace_{F_{m}} u\restriction_{T_{m-1}}
= 0.
    \end{align*}
    An application of Lemma~\ref{lemma:mixedbconsimplex} to the first expression in~\eqref{equation:w''} gives 
    \begin{align*}
        \| w''_{m} \|_{L^{p}(T_{m})} 
        &
        \leq 
        C_{\PF,T_{m},F_{m},p} 
        \left( 
            \| \nabla w \|_{L^{p}(T_{m})} 
            + 
            \| \nabla w'_{m} \|_{L^{p}(T_{m})} 
        \right) 
        \\&= 
        C_{\PF,T_{m},F_{m},p} 
        \left( 
            \| \nabla u \|_{L^{p}(T_{m})} 
            + 
            \| \nabla w'_{m} \|_{L^{p}(T_{m})} 
        \right) 
        .
    \end{align*}
    Using Lemma~\ref{lemma:volumecomparison} as well as Definitions~\eqref{math:volumeratio} and~\eqref{math:reflectionmeasure}, we find 
    \begin{align*}
        \| \nabla w'_{m} \|_{L^{p}(T_{m})}
        &
        \leq 
        |\det(\Jacobian \Xi  )|^{-\frac 1 p} 
        \matnorm{ \Jacobian \Xi }
        \| \nabla w \|_{L^{p}(T_{m-1})}
        \\&
        \leq 
        \left( \frac{ \vol(T_m) }{ \vol(T_{m-1}) } \right)^{\frac 1 p}
       ~\reflectionestimate(\calT)
        \| \nabla w \|_{L^{p}(T_{m-1})}
=
        \volumeratio(\calT)^{\frac 1 p}~\reflectionestimate(\calT)
        \| \nabla u \|_{L^{p}(T_{m-1})}
        .
    \end{align*}
Since $w\restriction_{T_{m}} = w''_{m} + w'_{m}$, we finally find 
    \begin{align*}
        \| w \|_{L^{p}(T_{m})}
        &
        \leq  
        \| w'_{m} \|_{L^{p}(T_{m})}
        + 
        \| w''_{m} \|_{L^{p}(T_{m})}
        \\&
        \leq  
        \volumeratio(\calT)^{\frac 1 p} 
        \| w \|_{L^{p}(T_{m-1})} 
        + 
        C_{\PF,T_{m},F_{m},p} 
        \left( 
            \| \nabla u \|_{L^{p}(T_{m})} 
            + 
            \volumeratio(\calT)^{\frac 1 p}~\reflectionestimate(\calT)
        \| \nabla u \|_{L^{p}(T_{m-1})}
        \right). 
    \end{align*}
    Therefrom, the first recursive estimate follows.

    Now we discuss the second recursive estimate. 
    Suppose again that $1 \leq m \leq M$. 
    We use the Poincar\'e--Friedrichs inequality over $\UPatch_{F_{m}}$,
    as given in Lemma~\ref{lemma:poincarefriedrichsoverfacepatch}, 
    to find $w_{F_m} \in W^{1,p}(\UPatch_{F_m})$ such that $\nabla w_{F_m} = \nabla u$ over $\UPatch_{F_m}$ and 
    \begin{align}\label{math:wFm_inequality}
        \| w_{F_m} \|_{L^{p}(\UPatch_{F_m})} \leq C_{{\PF},\UPatch_{F_m},p} \| \nabla w_{F_m} \|_{L^{p}(\UPatch_{F_m})}.
    \end{align}
We can define the constant
    \begin{align*}
    c_{m} := w\restriction_{\UPatch_{F_m}} - w_{F_{m}}.
    \end{align*}
    Now we observe that 
    \begin{gather*}
        \| w \|_{L^{p}(T_m)}
        \leq 
        \| w_{F_m} \|_{L^{p}(T_m)}
        +
        \| c_{m} \|_{L^{p}(T_m)},
        \\
        \| c_{m} \|_{L^{p}(T_m)}
        = 
        \frac{ \vol(T_m)^{\frac 1 p} }{ \vol(T_{{m-1}})^{\frac 1 p} }
        \| c_{m} \|_{L^{p}(T_{{m-1}})},
        \\ 
        \| c_{m} \|_{L^{p}(T_{{m-1}})}
        \leq 
        \| w \|_{L^{p}(T_{{m-1}})} + \| w_{F_{m}} \|_{L^{p}(T_{{m-1}})} 
        .
    \end{gather*}
    In combination, 
    \begin{align*}
        \| w \|_{L^{p}(T_m)}
        &
        \leq 
        \| w_{F_m} \|_{L^{p}(T_m)}
        \\&\qquad 
        +
        \frac{ \vol(T_m)^{\frac 1 p} }{ \vol(T_{{m-1}})^{\frac 1 p} }
        \| w_{F_{m}} \|_{L^{p}(T_{{m-1}})}
        +
        \frac{ \vol(T_m)^{\frac 1 p} }{ \vol(T_{{m-1}})^{\frac 1 p} }
        \| w \|_{L^{p}(T_{{m-1}})}
        .
    \end{align*}
    We sum the two integrals of $w_{F_m}$. 
When $1 < p < \infty$, recalling the complementary exponent $q = p/(p-1) \in (1,\infty)$, 
    we use H\"older's inequality to verify 
    \begin{align*}
        \| w \|_{L^{p}(T_m)}
        &
        \leq 
        \left( 1 + \frac{ \vol(T_m)^{\frac q p} }{ \vol(T_{{m-1}})^{\frac q p} } \right)^{\frac 1 q}
        \left( 
            \| w_{F_{m}} \|_{L^{p}(T_m)}^{p}
            +
            \| w_{F_{m}} \|_{L^{p}(T_{{m-1}})}^{p}
        \right)^{\frac 1 p}
        +
        \frac{ \vol(T_m)^{\frac 1 p} }{ \vol(T_{{m-1}})^{\frac 1 p} }
        \| w \|_{L^{p}(T_{{m-1}})}
        \\&
        = 
        \left( 1 + \frac{ \vol(T_m)^{\frac q p} }{ \vol(T_{{m-1}})^{\frac q p} } \right)^{\frac 1 q}
        \| w_{F_{m}} \|_{L^{p}(\UPatch_{F_m})} 
        +
        \frac{ \vol(T_m)^{\frac 1 p} }{ \vol(T_{{m-1}})^{\frac 1 p} }
        \| w \|_{L^{p}(T_{{m-1}})}
        .
    \end{align*}
    Note that in the limit cases $p=1$ and $p=\infty$ we get, respectively, 
    \begin{align*}
        \| w\|_{L^{1}(T_m)}
        &\leq 
        \max\left(
            1, \frac{ \vol(T_m) }{ \vol(T_{{m-1}}) } 
        \right)
        \| w_{F_m} \|_{L^{1}(\UPatch_{F_m})}
        +
        \frac{ \vol(T_m) }{ \vol(T_{{m-1}}) }
        \| w \|_{L^{1}(T_{{m-1}})},
        \\
        \| w \|_{L^{\infty}(T_m)}
        &\leq 
        2
        \| w_{F_m} \|_{L^{\infty}(\UPatch_{F_m})}
        +
        \| w \|_{L^{\infty}(T_{{m-1}})}
        .
    \end{align*}
    The local inequality~\eqref{math:wFm_inequality} now provides the second recursive estimate. 
    The proof is complete.
\end{proof}

The recursive construction of the gradient potential in the previous theorem, marching from simplex to simplex, can be associated with a concept of graph theory. Indeed, the face-neighbor relationship between adjacent $n$-simplices gives rise to an undirected graph that we call \emph{face-connection graph}.
Starting from an initial $n$-simplex $T_0$, any sequence of simplices corresponds to a path in that graph. In practice, we will pick a spanning tree for the undirected graph to describe the construction of the potentials.

We now use this formalism to describe an estimate for the Poincar\'e--Friedrichs constant of the gradient potential.
The formulation of the result is somewhat technical, but the underlying idea is as follows.
Unwrapping the recursion, we estimate the potential over a simplex in terms of the norm of the gradient vector field over numerous other simplices. 
Having estimated the potential on each simplex, a global estimate follows naturally and implies the desired Poincar\'e--Friedrichs inequality.

\begin{theorem}\label{theorem:fullrecursivesum:grad}
    Let $\calT$ be a face-connected $n$-dimensional finite triangulation
    and let the domain $\Omega$ be the interior of the underlying set of $\calT$. 
    Let $1 \leq p,q \leq \infty$ with $1 = 1/p + 1/q$,
    and suppose further that $u, w \in W^{1,p}(\Omega)$ with $\nabla w = \nabla u$.

    Suppose that we have an estimate of the form 
    \begin{align*}
        \| w \|_{L^{p}(T_{0})} &\leq A_{0} \| \nabla u \|_{L^{p}(T_{0})}
        .
    \end{align*}
    Suppose that the $n$-simplices are enumerated as $T_0, T_1, \dots, T_M$, 
    and for each $1 \leq m \leq M$ we have $0 \leq \varpi(m) \leq m-1$ 
    such that 
    \begin{align*}
        \| w \|_{ L^{p}(T_{m}) }
        &\leq  
        A_{m}
        \| \nabla u \|_{ L^{p}(T_{m}) }
        + 
        A'_{m}
        \| \nabla u \|_{ L^{p}(T_{\varpi(m)}) }
        +
        B_{m}
        \| w \|_{ L^{p}(T_{\varpi(m)}) }
        ,
        \quad 1 \leq m \leq M
        .
    \end{align*}
Then 
    \begin{gather*}
        \| w \|_{ L^{p}(\Omega) }
        \leq 
        \left(
            \sum_{m=0}^{M}
            \left( \sum_{\ell=0}^{L(m)} C_{m,\ell}^{q} \right)^{\frac p q}
        \right)^{\frac 1 p}
        \| \nabla u \|_{ L^{p}(\Omega) }
        ,
    \end{gather*}
    where $C_{m,\ell} \geq 0$ is as in the proof, with the obvious modifications if $p=1$ or $p=\infty$. 
\end{theorem}

\begin{proof}
    Suppose that $1 \leq m \leq M$. 
    Starting from index $m$, we repeatedly apply the predecessor function $\varpi$ until we come back to index zero. 
    This produces a strictly ascending sequence of indices 
    \begin{align}
        0 = \recindex_{m,0}, \; \recindex_{m,1}, \; \dots, \; \recindex_{m,L(m)} = m
    \end{align}
    for some $L(m) \geq 1$ such that $\recindex_{m,\ell-1} = \varpi(\recindex_{m,\ell})$. 
Unwrapping the recursion for the norm over the $m$-th simplex, 
    the total expression is now written in the form
    \begin{align*}
        \begin{split}
            \| w \|_{L^{p}(T_{m})} 
            &
            \leq 
            \sum_{\ell=0}^{L(m)}   \left( B_{\recindex_{m,L(m)}} \cdots B_{\recindex_{m,\ell+1}} \right) A_{\recindex_{m,\ell  }} \| \nabla u \|_{L^{p}(T_{\recindex_{m,\ell}})}
            \\&\qquad 
            +
            \sum_{\ell=0}^{L(m)-1} \left( B_{\recindex_{m,L(m)}} \cdots B_{\recindex_{m,\ell+2}} \right) A'_{\recindex_{m,\ell+1}} \| \nabla u \|_{L^{p}(T_{\recindex_{m,\ell}})}
            \\&
            = 
            A_{m} \| \nabla u \|_{L^{p}(T_{\recindex_{m,L(m)}})}
            \\&\qquad 
            +
            \sum_{\ell=0}^{L(m)-1} 
            \left( B_{\recindex_{m,L(m)}} \cdots B_{\recindex_{m,\ell+2}} \right) 
            \left( B_{\recindex_{m,\ell+1}} A_{\recindex_{m,\ell}} + A'_{\recindex_{m,\ell+1}} \right) 
            \| \nabla u \|_{L^{p}(T_{\recindex_{m,\ell}})}
            \\&
            =:
            \sum_{\ell=0}^{M} C_{m,\ell} \| \nabla u \|_{L^{p}(T_{\ell})}
            .
        \end{split}
    \end{align*}
    Here, $C_{m,\ell}$ is the coefficient of $\| \nabla u \|_{L^{p}(T_{\ell})}$, possibly zero, as it appears in the unwrapped recursive estimate of $\| w \|_{L^{p}(T_{m})}$. 
    The global Poincar\'e--Friedrichs inequality follows via H\"older's inequality:
    \begin{align*}
        \begin{split}
            \| w \|_{L^{p}(\Omega)}^{p}
            &
            \leq 
            \sum_{m=0}^{M}
            \| w \|_{L^{p}(T_{m})}^{p}
            \\&
            \leq 
            \sum_{m=0}^{M}
            \left( \sum_{\ell=0}^{M} C_{m,\ell} \| \nabla u \|_{L^{p}(T_{\ell})} \right)^{p}
            \\&
            \leq 
            \sum_{m=0}^{M}
            \left( \sum_{\ell=0}^{M} C_{m,\ell}^{q} \right)^{\frac p q}
            \sum_{\ell'=0}^{M} \| \nabla u \|_{L^{p}(T_{\ell'})}^{p} 
            \leq 
            \left(
                \sum_{m=0}^{M}
                \left( \sum_{\ell=0}^{M} C_{m,\ell}^{q} \right)^{\frac p q}
            \right)
            \| \nabla u \|_{L^{p}(\Omega)}^{p} 
            ,
        \end{split}
    \end{align*}
    where $q \in [1,\infty]$ satisfies $1 = 1/p + 1/q$ and with obvious modifications if $p=1$ or $p=\infty$. 
\end{proof}

\begin{remark}
    The computable Poincar\'e--Friedrichs constants obtained in Theorem~\ref{theorem:fullrecursivesum:grad} 
    depend on only a few parameters of the given triangulation: 
    the length of any traversal from the root simplex, 
    the volume ratios of any pair of adjacent simplices along this path,
    and the Poincar\'e--Friedrichs constants on each simplex (Lemma~\ref{lemma:mixedbconsimplex}) or on face patch (Lemma~\ref{lemma:poincarefriedrichsoverfacepatch}), then possibly including shape regularity parameters.

These computable Poincar\'e--Friedrichs constants increasingly overestimate the best one as the number of $n$-simplices in the triangulation $\calT$ increases. 
    Hence, we conceive their main application to be local patches (stars), in particular non-convex boundary stars.
    The latter occur inevitably at reentrant corners. Clearly, the same building principle applies whenever we have any non-overlapping partition of $\overline \Omega$ into convex local patches $\{ \UPatch_m \}$ of $n$-simplices with an appropriate notion of connectivity. 
    The proof proceeds verbatim, where we merely replace the simplices $\{ T_{m} \}$ by the convex local patches $\{ \UPatch_m \}$. 
    This may allow for partitions of $\overline \Omega$ with significantly fewer elements, which enables a largely improved estimate of the best Poincar\'e--Friedrichs constant. 
\end{remark}

\section{Review of vector calculus and exterior calculus}\label{section:calculus}

We review in this section the Sobolev spaces of vector and exterior calculus with particular emphasis on their behavior under bi-Lipschitz transformations.
We refer the reader to Ern and Guermond~\cite{ern2021finite} and Hiptmair~\cite{hiptmair2002finite} for background material on Sobolev vector analysis and to Greub~\cite{greub1967multilinear} and Lee~\cite{lee2012smooth} for exterior algebra and exterior products. 

\subsection{Vector calculus}

Let $\Omega \subseteq \bbR^{3}$ be a bounded open set. 
We recall that $L^{p}(\Omega)$ denotes the space of scalar-valued $p$-integrable functions defined on $\Omega$
and that $\bfL^{p}(\Omega) := L^{p}(\Omega)^{n}$ stands for vector-valued functions with each component in $L^{p}(\Omega)$. 
In the three-dimensional setting, we are particularly interested in the Sobolev vector analysis. 
The space of scalar-valued $L^{p}(\Omega)$ functions with weak gradients in $\bfL^{p}(\Omega)$ is 
\begin{gather*}
    W^{p}(\grad,\Omega) := W^{1,p}(\Omega) = \{ u \in L^{p}(\Omega) \suchthat \grad v \in \bfL^{p}(\Omega) \}
    .
\end{gather*}
The space $\bfW^{p}(\curl,\Omega)$ of vector-valued $\bfL^{p}(\Omega)$ functions with weak curls in $\bfL^{p}(\Omega)$
and the space of vector-valued $\bfL^{p}(\Omega)$ functions with weak divergences in $L^{p}(\Omega)$ are then written 
\begin{gather*}
    \bfW^{p}(\curl,\Omega) = \{ \bfu \in \bfL^{p}(\Omega) \suchthat \curl \bfu \in \bfL^{p}(\Omega) \},
    \\ 
    \bfW^{p}(\divergence,\Omega) = \{ \bfu \in \bfL^{p}(\Omega) \suchthat \divergence \bfu \in L^{p}(\Omega) \}.
\end{gather*}
We are interested in transformations of these Sobolev vector fields from one domain onto another. 
Suppose that $\Omega, \Omega' \subset \mathbb{R}^3$ are open sets and suppose that $\phi: \Omega \to \Omega'$ is a bi-Lipschitz mapping.
We introduce the gradient-, curl-, and divergence-conforming Piola transformations, respectively, as the mappings 
$\phi^{\grad}: L^{p}(\Omega') \to L^{p}(\Omega)$,
$\phi^{\curl}: \bfL^{p}(\Omega') \to \bfL^{p}(\Omega)$, and 
$\phi^{\divergence}: \bfL^{p}(\Omega') \to \bfL^{p}(\Omega)$.
We also introduce 
$\phi^{\rm b}: L^{p}(\Omega') \to L^{p}(\Omega)$. 
These are defined 
for any $v \in L^{p}(\Omega')$ and $\bfw \in \bfL^{p}(\Omega')$ by setting 
\begin{subequations}\label{eq_definition_piola}
\begin{align*}
    \phi^{\grad}(v) &= v \circ \phi, \\
    \phi^{\curl}(\bfw) &= \Jacobian\phi^{T} (\bfw \circ \phi), \\
    \phi^{\divergence}(\bfw) &= \adj(\Jacobian\phi) \left( \bfw \circ \phi \right), \\  
    \phi^{\rm b}(v) &= \det(\Jacobian\phi) \left( v \circ \phi \right),
\end{align*}
\end{subequations}
Here, $\Jacobian\phi$ is the Jacobian matrix of $\phi$ (see also~\cite[Definition~9.8]{ern2021finite}), and $\adj(\Jacobian\phi) := \det(\Jacobian\phi) (\Jacobian\phi)^{\inv}$ denotes taking its adjugate matrix.
These transformations are invertible. 
Bounds on the Lebesgue norms will follow from a more general result below. 
We use the commutativity relations 
\begin{subequations}\label{eq_piola_commute}
\begin{align}
    \grad \phi^{\grad} (v) &= \phi^{\curl}(\grad v), 
    \\
    \curl \phi^{\curl}(\bfv) &= \phi^{\divergence}(\curl \bfv), 
    \\
    \divergence \phi^{\divergence}(\bfw) &= \phi^{\rm b}(\divergence \bfw),
\end{align}
\end{subequations}
where $v \in W^{p}(\grad,\Omega)$, $\bfv \in \bfW^{p}(\curl,\Omega)$, and $\bfw \in \bfW^{p}(\divergence,\Omega)$.  
We summarize this as a commuting diagram:
\begin{align*}
    \begin{CD}
        W^{p}(\grad,\Omega') @>{\grad}>> \bfW^{p}(\curl,\Omega') @>{\curl}>> \bfW^{p}(\divergence,\Omega') @>{\divergence}>> L^{p}(\Omega')
        \\
        @VV{\phi^{\grad}}V 
        @VV{\phi^{\curl}}V 
        @VV{\phi^{\divergence}}V 
        @VV{\phi^{\rm b}}V 
        \\
        W^{p}(\grad,\Omega ) @>{\grad}>> \bfW^{p}(\curl,\Omega ) @>{\curl}>> \bfW^{p}(\divergence,\Omega ) @>{\divergence}>> L^{p}(\Omega ).
    \end{CD}
\end{align*}

\begin{remark}
    The Piola transform goes into the opposite direction of the mapping $\phi : \Omega \rightarrow \Omega'$:
    scalar and vector fields over $\Omega'$ are transformed into scalar and vector fields over $\Omega$.
    This definition is in accordance with the notion of pullback, which we will review shortly. 
    One advantage of that definition is that it also makes sense whenever the transformation is not bijective. 
    However, the literature also knows the Piola transform in the direction of the original mapping. 
\end{remark}

\subsection{Exterior calculus}

We now move the discussion to exterior calculus, beginning with exterior algebra.
Let $V$ be a real vector space. 
Given an integer $k \geq 0$, we let $\Alt^{k}(V)$ denote the space of scalar-valued antisymmetric $k$-linear forms over $V$. 
Recall that any $k$-linear scalar-valued form $u$ over $V$ is called antisymmetric
if 
\begin{gather*} 
    u( v_{\pi(1)}, v_{\pi(2)}, \ldots, v_{\pi(k)} ) 
    = 
    \text{sign}(\pi) 
    u( v_1, v_2, \ldots, v_k ) 
\end{gather*}
for any $v_1, v_2, \dots, v_k \in V$ and any permutation $\pi$ of the indices \(\{1, 2, \ldots, k\}\). 
By definition, $\Alt^{1}(V)$ is just the dual space of $V$, and $\Alt^{0}(V)$ is the space of real numbers. 
Formally, we define $\Alt^{k}(V)$ to be the zero vector space when $k < 0$. 

The wedge product (or exterior product) of alternating multilinear forms is a fundamental operation in exterior algebra
(see~Chapter 14 in~\cite{lee2012smooth}). 
Given two alternating multilinear forms \( u_1 \in \Alt^{k}(V) \) and \( u_2 \in \Alt^{l}(V) \), 
their wedge product \( u_1 \wedge u_2 \) is a member of $\Alt^{k+l}(V)$
defined by the formula 
\begin{gather*}
    (u_1 \wedge u_2)(v_1, v_2, \ldots, v_{k+l}) 
    = 
    \frac{1}{k! l!} 
\sum_{ \pi } 
    \signum(\pi) 
    u_1( v_{\pi(  1)}, \ldots, v_{\pi(k  )} ) 
    u_2( v_{\pi(k+1)}, \ldots, v_{\pi(k+l)} ),
\end{gather*}
for any \( v_1, v_2, \ldots, v_{k+l} \in V \).
Here, the sum runs over all permutations $\pi$ of the index set \(\{ 1, 2, \ldots, k+l \}\).
The exterior product is bilinear and associative, and satisfies 
\begin{align*}
    u_1 \wedge u_2 = (-1)^{kl} u_2 \wedge u_1,
    \quad 
    \forall u_1 \in \Alt^{k}(V),
    \quad 
    \forall u_2 \in \Alt^{l}(V).
\end{align*}
The interior product is in some sense dual to the exterior product. Given $v \in V$ and $u \in \Alt^{k}(V)$, we define the interior product $v \lrcorner u \in \Alt^{k-1}(V)$ via 
\begin{align*}
    (v \lrcorner u)( v_1, v_2, \ldots, v_{k-1} ) = u( v, v_1, v_2, \ldots, v_{k-1} ),
    \quad 
    \forall v_1, v_2, \ldots, v_{k-1} \in V.
\end{align*}

We employ the exterior algebra only in the special case $V = \bbR^{n}$ of alternating forms over the $n$-dimensional Euclidean space. 
Here, it is customary to identify $\Alt^{k}(V)$ with the space of antisymmetric tensors in $k$ indices. 
Moreover, this particular setting comes with a canonical basis. 
We let \(\{\cartanx^1, \cartanx^2, \ldots, \cartanx^{n}\}\) be the basis dual to the canonical unit vectors.
This is a canonical basis of $\Alt^{1}(\bbR^{n})$. 
To define a canonical basis of $\Alt^{k}(\bbR^{n})$, 
we first introduce $\Sigma(k,n)$, the set of strictly ascending mappings $\sigma : \{1,\dots,k\} \rightarrow \{1,\dots,n\}$, where $k, n \in \bbZ$, 
and introduce the basic $k$-alternators 
\begin{align*}
    \cartanx^{\sigma} := \cartanx^{\sigma(1)} \wedge \dots \wedge \cartanx^{\sigma(k)}, 
    \quad 
    \forall \sigma \in \Sigma(k,n). 
\end{align*}
These define a basis of $\Alt^{k}(\bbR^{n})$.
Note that 
$\dim \Alt^{k}(\bbR^{n}) = \binom{n}{k}$. 
In particular, $\Alt^{k}(\bbR^{n})$ is the zero vector space whenever $k > n$.

We notice that the canonical scalar product on $\bbR^{n}$ gives rise to a scalar product on $\Alt^{1}(\bbR^{n})$,
which induces a scalar product on $\Alt^{k}(\bbR^{n})$.
The basic $k$-alternators are an orthonormal basis of $\Alt^{k}(\bbR^{n})$ with respect to that inner product. 

\subsection{Smooth differential forms}

We let $\Omega \subset \bbR^{n}$ be any bounded open set.
We write $C^{\infty}\Alt^{k}(\Omega)$ for the space of smooth differential $k$-forms over $\Omega$,
which is the vector space of smooth mappings from $\Omega$ into $\Alt^{k}(\bbR^{n})$.
The exterior derivative \( \cartan \) is an operator that takes a \( k \)-form \( \omega \in C^{\infty}\Alt^{k}(\Omega) \) 
to a \((k+1)\)-form \( \cartan\omega \in C^{\infty}\Alt^{k+1}(\Omega) \). 
Every \( k \)-form \( \omega \in C^{\infty}\Alt^{k}(\Omega) \) can be written 
\begin{gather}\label{math:canonicalrepresentation}
    \omega 
    = 
    \sum_{ \sigma \in \Sigma(k,n) } 
    \omega_{\sigma} \, 
    \cartanx^{\sigma}
    = 
    \sum_{ \sigma \in \Sigma(k,n) } 
    \omega_{\sigma} \, 
    \cartanx^{\sigma(1)} \wedge \cartanx^{\sigma(2)} \wedge \cdots \wedge \cartanx^{\sigma(k)},
\end{gather}
where \( \omega_{\sigma} : \Omega \rightarrow \bbR \) are smooth functions.
The exterior derivative \( \cartan\omega \) is defined by
\begin{gather*}
    \cartan\omega = 
    \sum_{ \sigma \in \Sigma(k,n) } 
    \sum_{j=1}^{n} 
    \frac{\partial \omega_{\sigma}}{\partial x^j} 
    \, \cartanx^j \wedge 
    \cartanx^{\sigma(1)} \wedge \cartanx^{\sigma(2)} \wedge \cdots \wedge \cartanx^{\sigma(k)}.
\end{gather*}
The exterior derivative is linear and nilpotent, which means 
\( \cartan(\cartan\omega) = 0 \) for any \( \omega \in C^{\infty}\Alt^{k}(\Omega) \).
Moreover, it satisfies the Leibniz rule:
\begin{gather*} 
    \cartan(\omega \wedge \eta) 
    = 
    \cartan\omega \wedge \eta + (-1)^{k} \omega \wedge \cartan\eta, 
    \qquad \forall \omega \in C^{\infty}\Alt^{k}(\Omega), 
    \qquad \forall \eta \in C^{\infty}\Alt^{l}(\Omega).
\end{gather*}
The integral of a differential $n$-form is uniquely defined via 
\begin{gather*}
    \int_{\Omega} \omega \cartanx^{1} \wedge \dots \wedge \cartanx^{n} = \int_{\Omega} \omega(x) \;dx
    . 
\end{gather*}

\begin{remark}
    In three dimensions, 
    the calculus of differential forms is in correspondence with classical vector calculus. 
    This is expressed formally as the commuting diagram 
    \begin{gather*}
    \begin{CD}
        C^\infty\Alt^0(\Omega) @>{\cartan}>> C^\infty\Alt^1(\Omega) @>{\cartan}>> C^\infty\Alt^2(\Omega) @>{\cartan}>> C^\infty\Alt^3(\Omega) 
        \\
        @VV{\varpi^{0}}V 
        @VV{\varpi^{1}}V 
        @VV{\varpi^{2}}V 
        @VV{\varpi^{3}}V 
        \\
        C^\infty(\Omega) @>{\grad}>> C^\infty(\Omega)^{3} @>{\curl}>> C^\infty(\Omega)^{3} @>{\divergence}>> C^\infty(\Omega),
    \end{CD}
\end{gather*}
    where $\varpi^{0}$ and $\varpi^{3}$ are the identity mappings and where 
    \begin{align*}
     \varpi^{1}\left( u_1 \cartanx^{1} + u_2 \cartanx^{2} + u_3 \cartanx^{3} \right) 
     &= 
     \left( u_1, u_2, u_3 \right), 
     \\
     \varpi^{2}\left( u_{12} \cartanx^{1} \wedge \cartanx^{2} + u_{13} \cartanx^{1} \wedge \cartanx^{3} + u_{23} \cartanx^{2} \wedge \cartanx^{3} \right) 
     &= 
     \left( u_{23}, - u_{13}, u_{12} \right).   
    \end{align*}
    In two dimensions, 
    the calculus of differential forms can be translated into 2D vector calculus in two different ways. 
    To the authors' best knowledge, neither convention is dominant over the other in the literature.
    We summarize the situation in the following commuting diagram: 
    \begin{gather*} 
    \begin{CD}
        C^\infty(\Omega) @>{\curl}>> C^\infty(\Omega)^{2} @>{\divergence}>> C^\infty(\Omega)
        \\
        @AA{\varkappa^{0}}A 
        @AA{\varkappa^{1}}A 
        @AA{\varkappa^{2}}A 
        \\
        C^\infty\Alt^0(\Omega) @>{\cartan}>> C^\infty\Alt^1(\Omega) @>{\cartan}>> C^\infty\Alt^2(\Omega) 
        \\
        @VV{\varpi^{0}}V 
        @VV{\varpi^{1}}V 
        @VV{\varpi^{2}}V 
        \\
        C^\infty(\Omega) @>{\grad}>> C^\infty(\Omega)^{2} @>{\rot}>> C^\infty(\Omega)
        .
    \end{CD}
\end{gather*}
    Here, $\varpi^{1}\left( u_1 \cartanx^{1} + u_2 \cartanx^{2} \right) = \left( u_1, u_2 \right)$ is the lower middle isomorphism. We introduce the rotation operator $J(x,y) = (y,-x)$ and define $\varkappa^{1} = J \varpi^{1}$ and $\rot = \divergence J$.
    The other vertical arrows are the identity. 
    The utility of exterior calculus is that the operators of vector calculus can be translated into a common framework that does not depend on the dimension.
\end{remark}

\subsection{Sobolev spaces of differential forms}

Let us now turn our attention to Sobolev spaces of differential forms. 
Since the exterior product space $\Alt^{k}(\bbR^{n})$ carries a norm, induced from the Euclidean norm on $\bbR^{n}$, there are pointwise norms of differential $k$-forms. 
We let $L^{p}\Alt^{k}(\Omega)$ be the space of differential $k$-forms over $\Omega$ with locally integrable coefficients 
such that its pointwise norm is $p$-integrable. 
The exterior derivative is defined in the sense of distributions and we introduce 
\begin{gather*}
    W^{p}\Alt^{k}(\Omega) 
    := 
    \{ u \in L^{p}\Alt^{k}(\Omega) \suchthat \cartan u \in L^{p}\Alt^{k+1}(\Omega) \}
    .
\end{gather*}
We observe that $u \in L^{p}\Alt^{k}(\Omega)$ has weak exterior derivative $f \in L^{p}\Alt^{k+1}(\Omega)$
if and only if for all $v \in C^{\infty}_{c}\Lambda^{n-k-1}(\Omega)$ we have the integration-by-parts formula
\begin{align*}
    \int_{\Omega} \cartan v \wedge u
    &=
    (-1)^{k(n-k)+1}
    \int_{\Omega} v \wedge f 
    .
\end{align*}
Lastly, we are also interested in differential forms whose trace vanishes along a part of the boundary. 
Suppose that $\Gamma \subseteq \partial\Omega$ is a relatively open subset of the boundary. 
We say that $u \in W^{p}\Alt^{k}(\Omega)$ has vanishing trace along $\Gamma$ if for all $x \in \Gamma$ there exists $r > 0$
such that for all $v \in C^{\infty}_{c}\Lambda^{n-k-1}(\bbR^{n})$ 
whose support lies in the open ball $B_r(x)$
we have the integration-by-parts formula
\begin{align*}
    \int_{B_r(x)} \cartan v \wedge \tilde u
    &=
    (-1)^{k(n-k)+1}
    \int_{B_r(x)} v \wedge \widetilde{\cartan u}
    ,
\end{align*} where $\tilde u : \bbR^{n} \rightarrow \Alt^{k}(\bbR^{n})$ and $\widetilde{\cartan u} : \bbR^{n} \rightarrow \Alt^{k+1}(\bbR^{n})$
are the extensions of $u$ and $\cartan u$ by zero.
If that condition is satisfied, we also write 
\begin{align*}
    \trace_{\Gamma} u = 0
    .
\end{align*}
Accordingly, we write $\trace_{\Gamma} u = \trace_{\Gamma} u'$ for $\trace_{\Gamma} (u-u') = 0$ whenever $u, u' \in W^{p}\Alt^{k}(\Omega)$.
Lastly, we introduce the closed subspaces 
\begin{gather*}
    W^{p}_{0}\Alt^{k}(\Omega) 
    := 
    \{ u \in W^{p}\Alt^{k}(\Omega) \suchthat \trace_{\partial \Omega} u = 0 \}
    .
\end{gather*}
We know that $u \in W^{p}_{0}\Alt^{k}(\Omega)$ if and only if its extension by zero $\tilde u : \bbR^{n} \rightarrow \Alt^{k}(\bbR^{n})$ is a member of $\tilde u \in W^{p}\Alt^{k}(\bbR^{n})$. Moreover, $\cartan W^{p}_{0}\Alt^{k}(\Omega) \subseteq W^{p}_{0}\Alt^{k+1}(\Omega)$. We also observe that $W^{p}\Alt^{n}(\Omega) = W^{p}_{0}\Alt^{n}(\Omega) = L^{p}(\Omega)$. We use the abbreviation $W^{1,p}_{0}(\Omega) := W^{p}\Lambda^{0}(\Omega)$.

\subsection{Transformations by bi-Lipschitz mappings}

We are interested in transformations of Sobolev tensor fields from one domain onto another. 
Suppose that $\Omega,\Omega' \subset \bbR^{n}$ are open sets and suppose that $\phi: \Omega \to \Omega'$ is a bi-Lipschitz mapping.
The pullback of $u \in L^{p}\Alt^{k}(\Omega')$ along $\phi$ is the (measurable) differential form 
\begin{align}\label{equation:pullback}
    \phi^{\ast}u\restriction_{x}( v_{1}, v_{2}, \ldots, v_{k} ) 
    := 
    u\restriction_{\phi(x)}( \Jacobian\phi\restriction_{x} \cdot v_1, \Jacobian\phi\restriction_{x} \cdot v_2, \ldots, \Jacobian\phi\restriction_{x} \cdot v_k ) 
\end{align}
for any $v_1, v_2, \dots, v_k \in \bbR^{n}$ and any $x \in \Omega$. 
One can show that $\phi^{\ast}u \in L^{p}\Alt^{k}(\Omega)$ and the following estimates.

\begin{proposition}\label{proposition:pullbackestimate}
    Let $\phi : \Omega \rightarrow \Omega'$ be a bi-Lipschitz mapping between open sets $\Omega, \Omega' \subseteq \bbR^{n}$.
    Let $p \in [1,\infty]$ and $u \in L^{p}\Lambda^{k}(\Omega')$. 
    Then $\phi^{\ast} u \in L^{p}\Lambda^{k}(\Omega)$ and 
    \begin{align}
        \| \phi^{\ast} u \|_{L^{p}\Lambda^{k}(\Omega)}
        &\leq 
        \| \Jacobian \phi \|^{k}_{L^{\infty}(\Omega)}
        \| \det\Jacobian \phi^{-1} \|^{\frac{1}{p}}_{L^{\infty}(\Omega')}
        \| u \|_{L^{p}\Lambda^{k}(\Omega')}
        .
    \end{align}
    If $\phi$ is affine and $\sigma_1 \geq \sigma_2 \geq \cdots \geq \sigma_n$ are the singular values of $\Jacobian\phi$, then 
    \begin{align}
        \| \phi^{\ast} u \|_{L^{p}\Lambda^{k}(\Omega)}
        &\leq 
        \sigma_1 \sigma_2 \cdots \sigma_k \cdot 
        \| \det\Jacobian \phi^{-1} \|^{\frac{1}{p}}_{L^{\infty}(\Omega')}
        \| u \|_{L^{p}\Lambda^{k}(\Omega')}
        .
    \end{align}
    Moreover, if $u \in W^{p}\Lambda^{k}(\Omega')$, then 
    \begin{align}\label{math:pullbackcommutes}
        \phi^{\ast} u \in W^{p}\Lambda^{k}(\Omega),
        \qquad 
        \cartan \phi^{\ast} u = \phi^{\ast} \cartan u.
    \end{align}
\end{proposition}
\begin{proof}
    See~\cite{licht2019smoothed} and Corollary~6 in~\cite{stern2013lp}.
\end{proof}

\subsection{Some approximation properties}

We review a few approximation properties. 
Let $\mollifier$ be a non-negative scalar function whose integral equals one and whose support lies in the unit ball around the origin.
Define $\mollifier_{\eps}(x) := \eps^{-n} \mollifier(x/\eps)$. We collect several approximation results that involve convolution with a mollifier.
In what follows, convolutions of functions defined over domains tacitly assume that these functions have been extended by zero to all of Euclidean space.

\begin{lemma}
Let $\Omega \subseteq \bbR^{n}$ be a bounded open set and let $1 \leq p < \infty$. 
    If $u \in L^{p}(\Omega)$, then the convolution $\mollifier_{\eps} \star u \rightarrow u$ in $L^{p}(\Omega)$ as $\epsilon \rightarrow 0$.
\end{lemma}

\begin{lemma}
    Let $\Omega \subseteq \bbR^{n}$ be a bounded open set and let $1 \leq p < \infty$. 
    Smooth forms are dense in $W^{p}\Lambda^{k}(\Omega)$.
    If $\Omega$ is convex, then $C^{\infty}_{c}\Lambda^{k}(\Omega)$ is dense in $W^{p}_{0}\Lambda^{k}(\Omega)$. 
\end{lemma}
\begin{proof}
    We notice $\mollifier_{\epsilon} \star u \in C^{\infty}\Lambda^{k}(\bbR^{n})$. 
    By the dominated convergence theorem and because $u$ has a weak derivative,
    for any $v \in C^{\infty}_{c}\Lambda^{n-k-1}(\Omega)$
    \begin{align*}
        \int_{\bbR^{n}} ( \mollifier_{\epsilon} \star u ) \wedge \cartan v 
        &= 
        \int_{\bbR^{n}} \int_{\bbR^{n}} \mollifier(x-y) \wedge u(y) \wedge \cartan_{x} v(x) \;dx\,dy
        \\&= 
        \int_{\bbR^{n}} \int_{\bbR^{n}} \cartan_x\mollifier(x-y) \wedge u(y) \wedge v(x) \;dx\,dy
        \\&= 
        - \int_{\bbR^{n}} \int_{\bbR^{n}} \cartan_y\mollifier(x-y) \wedge u(y) \wedge v(x) \;dx\,dy
        \\&= 
        \int_{\bbR^{n}} \int_{\bbR^{n}} \mollifier(x-y) \wedge \cartan_y u(y) \wedge v(x) \;dx\,dy
        = 
        \int_{\bbR^{n}} ( \mollifier_{\epsilon} \star \cartan u ) \wedge v 
        .
    \end{align*}
    Hence, $C^{\infty}\Lambda^{k}(\Omega)$ is dense in $W^{p}\Lambda^{k}(\Omega)$.
    
    Next, suppose that $\Omega$ is convex. Without loss of generality, $0 \in \Omega$. 
    Let $u \in W^{p}_{0}\Lambda^{k}(\Omega)$ and extend $u$ trivially onto $\bbR^{n}$. 
    Define $\varphi_t(x) = tx$ for $t > 1$. 
    Then $\varphi_{t}^{\ast} u \in W^{p}_{0}\Lambda^{k}(t^{-1}\Omega) \subseteq W^{p}_{0}\Lambda^{k}(\Omega)$ 
    and $\varphi_{t}^{\ast} u$ converges to $u$ as $t$ decreases towards $1$. 
    Given any $t > 1$, by taking the convolution with $\mollifier_{\epsilon}$ for $\epsilon > 0$ small enough,
    we approximate $\varphi_{t}^{\ast} u$ through members of $C^{\infty}_{c}\Lambda^{k}(\Omega)$.
    The desired result follows. 
\end{proof}

\section{Regularized potentials over bounded convex sets}\label{section:potentialoperator}

We now develop bounds for Poincar\'e--Friedrichs constants for the exterior derivative over convex domains.
Here, we consider two special cases: 
either the $L^{p}$ de~Rham complex without boundary conditions, or the $L^{p}$ de~Rham with full boundary conditions. 
The corresponding linear potentials are known as the regularized Poincar\'e and regularized Bogovski\u{\i} potentials in the literature. 
We build upon the discussion spearheaded by Costabel and McIntosh~\cite{costabel2010bogovskiui},
who analyze them as pseudo-differential operators over domains star-shaped with respect to a ball. 
In comparison to their extensive work, our discussion is more modest:
we study potentials merely over convex sets, and we are only interested in their operator norms between Lebesgue spaces.
However, our goal is explicit bounds for the operator norms, giving the Poincar\'e--Friedrichs constants. 

In the remainder of this section, $\Omega \subseteq \bbR^{n}$ is a bounded convex open set with diameter $\diam(\Omega) > 0$.

\subsection{Regularized Poincar\'e and Bogovski\u{\i} operators}

We begin by introducing the Costabel--McIntosh kernel.
For any $\mia \in \{0,\dotsc,n\}$, we define the kernel $\calG_{\mia} : \bbR^{n} \times \bbR^{n} \rightarrow \bbR$ by
\begin{equation}\label{math:mainkernel}
  \calG_{\mia}(x,y) = \int_{1}^\infty (t-1)^{n-\mia}t^{\mia-1} \vol(\Omega)^{-1}\chi_{\Omega} \left(y+t(x-y)\right)\,dt
  ,
\end{equation}
where $\chi_{\Omega} : \Omega \rightarrow \{0,1\}$ denotes the characteristic function of the domain $\Omega$.
Given a differential form $u \in {C^\infty_c}(\bbR^n,\Alt^\mia)$, where \(1 \leq \mia \leq n\), 
we then define the integral operators
\begin{align*}
  \Poinc_{\mia} u(x) &= \int_{\Omega} \calG_{n-\mia+1}(y,x) \,(x-y)\lrcorner u(y)\,dy,
  \\
  \Bogov_{\mia} u(x) &= \int_{\Omega} \calG_{\mia}(x,y) \,(x-y)\lrcorner u(y)\,dy.
\end{align*}
We call $\Poinc_{\mia}$ the \emph{Poincar\'e operator} and $\Bogov_{\mia}$ the \emph{Bogovski\u{\i} operator}. 

We show that the integrals in the definition of $\Poinc_{\mia}$ and $\Bogov_{\mia}$ actually converge. 
In order to analyze the properties of the potentials,
we first rewrite the Costabel--McIntosh kernel $\calG_{\mia}$.
Letting $x,y \in \bbR^n$ with $x \neq y$, we find 
\begin{align*}
    \calG_{\mia}(x,y) 
    &= 
    \int_{0}^{\infty} t^{n-\mia} (t+1)^{\mia-1} \vol(\Omega)^{-1} \chi_{\Omega} \left( x+t(x-y) \right)\,dt
    \\
    &= 
    \int_{0}^{\infty} \sum_{\jeza=0}^{\mia-1} \tbinom{\mia-1}{\jeza} t^{n-\mia+\jeza} \vol(\Omega)^{-1} \chi_{\Omega} \left( x+t(x-y) \right)\,dt
    \\
    &= 
    \int_{0}^{\infty} \sum_{\jeza=0}^{\mia-1} \tbinom{\mia-1}{\mia-1-\jeza} t^{n-\mia+\mia-1-\jeza} \vol(\Omega)^{-1} \chi_{\Omega} \left( x+t(x-y) \right)\,dt
    \\
    &= 
    \int_{0}^{\infty} \sum_{\jeza=0}^{\mia-1} \tbinom{\mia-1}{\jeza} t^{n-\jeza-1} \vol(\Omega)^{-1} \chi_{\Omega} \left( x+t(x-y) \right)\,dt
    \\
    &= 
    \sum_{\jeza=0}^{\mia-1} \tbinom{\mia-1}{\jeza} \int_{0}^{\infty} t^{n-\jeza-1} \vol(\Omega)^{-1} \chi_{\Omega}\left( x+t(x-y) \right)\,dt 
    \\
    &= 
    \vol(\Omega)^{-1} \sum_{\jeza=0}^{\mia-1} \tbinom{\mia-1}{\jeza} \, |x-y|^{\jeza-n} \int_{0}^{\infty} r^{n-\jeza-1} \chi_{\Omega}\left(x+r\frac{x-y}{|x-y|}\right)\,dr
    .
\end{align*}
If $x \in \Omega$ and $x \neq y$, 
then we can restrict the inner integrals to the range $0 \leq r \leq \diam(\Omega)$, which gives 
\begin{align*}
    \calG_{\mia}(x,y) 
    &= 
    \vol(\Omega)^{-1} \sum_{\jeza=0}^{\mia-1}\tbinom{\mia-1}{\jeza} |x-y|^{\jeza-n} \int_{0}^{\diam(\Omega)} r^{n-\jeza-1} \,dr 
    \\
    &= 
    \vol(\Omega)^{-1} \sum_{\jeza=0}^{\mia-1}\tbinom{\mia-1}{\jeza} |x-y|^{\jeza-n} \frac{ \diam(\Omega)^{n-\jeza} }{n-\jeza}.
\end{align*}
We are now in a position to show that the potentials are bounded with respect to Lebesgue norms.

\subsection{An operator norm bound with respect to the Lebesgue norm: the Poincar\'e case}

We begin with the Poincar\'e operator. 
Let $B_{\diam(\Omega)}(0)$ be the $n$-dimensional ball centered at the origin.
Suppose that $u \in L^{\infty}\Alt^{\mia}(\Omega)$ with $1 \leq \mia \leq n$.
We estimate $\Poinc_{\mia} u(x)$ pointwise for any $x \in \Omega$ by the result of a convolution of a locally integrable function with $u$:
\begin{align*}
    \left| \Poinc_{\mia} u(x) \right|
    &=
    \left| 
        \int_{\Omega} \calG_{n-\mia+1}(x,y) \,(x-y)\lrcorner u(y)\,dy
    \right| 
\\&\leq 
    \int_{\Omega} \vol(\Omega)^{-1} \sum_{\jeza=0}^{ n-\mia }\tbinom{ n-\mia }{\jeza} \frac{ \diam(\Omega)^{n-\jeza} }{n-\jeza} |x-y|^{\jeza+1-n} \chi_{B_{\diam(\Omega)}(0)}(x-y) |u(y)| \,dy
    .
\end{align*}
We recall the radial integrals 
\begin{align*}
    \int_{B_{\diam(\Omega)}(0)} |z|^{\jeza+1-n} \,dz
    &
    =
    \vol_{n-1}(S_1) \int_{0}^{\diam(\Omega)} r^{\jeza+1-n} r^{n-1} \,dr
    \\&
    =
    \vol_{n-1}(S_1) \int_{0}^{\diam(\Omega)} r^{\jeza} \,dr
    =
    \vol_{n-1}(S_1) \frac{\diam(\Omega)^{\jeza+1}}{\jeza+1}
    ,
\end{align*}
where $S_1 \subseteq \bbR^{n}$ stands for the unit sphere of dimension $n-1$. 
One computes 
\begin{align*}
    &
    \int_{\bbR^{n}} \sum_{\jeza=0}^{n-\mia} \tbinom{n-\mia}{\jeza} \frac{ \diam(\Omega)^{n-\jeza} }{n-\jeza} \chi_{B_{\diam(\Omega)}(0)}(z) |z|^{\jeza+1-n} \,dz
    \\&
    =
    \sum_{\jeza=0}^{n-\mia} \tbinom{n-\mia}{\jeza} \frac{ \diam(\Omega)^{n-\jeza} }{n-\jeza} \int_{B_{\diam(\Omega)}(0)} |z|^{\jeza+1-n} \,dz
    \\&
    =
    \vol_{n-1}(S_1) \sum_{\jeza=0}^{n-\mia} \tbinom{n-\mia}{\jeza} \frac{ \diam(\Omega)^{n-\jeza} }{n-\jeza} \frac{\diam(\Omega)^{\jeza+1}}{\jeza+1}
=
    \vol_{n-1}(S_1) \diam(\Omega)^{n+1} \underbrace{ \sum_{\jeza=0}^{n-\mia} \frac{ \tbinom{n-\mia}{\jeza} }{ (n-\jeza)(\jeza+1) } }_{ =: C_{\Poinc}(n,\mia) \leq 2^{n-\mia}}
    .
\end{align*}
Here we introduce the numerical constant 
\begin{gather}\label{equation:A_Poinc}
    C_{\Poinc}(n,\mia) := \sum_{\jeza=0}^{n-\mia} \frac{ \tbinom{n-\mia}{\jeza} }{ (n-\jeza)(\jeza+1) },
\end{gather}
which depends only on $n$ and $\mia$ and which is bounded by $2^{n-\mia}$. 

In particular, the integral $\Poinc_{\mia} u(x)$ is absolutely convergent for any choice of $x \in \Omega$
and its magnitude is pointwise dominated by the convolution of $|u|$ against an integrable function.
Young's convolution inequality now implies: 
\begin{align*}
    \| \Poinc_{\mia} u \|_{L^{p}(\Omega)}
    &
    \leq 
    \vol_{n-1}(S_1) C_{\Poinc}(n,\mia) \frac{ \diam(\Omega)^{n} }{\vol(\Omega)} 
    \diam(\Omega)
    \| u \|_{L^{p}(\Omega)}
    \\&
    \leq 
    n C_{\Poinc}(n,\mia) \frac{ \vol_{}(B_{\diam(\Omega)}(0)) }{\vol(\Omega)} 
    \diam(\Omega)
    \| u \|_{L^{p}(\Omega)}
    .
\end{align*}
We have assumed so far that $u \in L^{\infty}\Alt^{\mia}(\Omega)$. 
Since that space is dense in the Lebesgue spaces, a density argument establishes the following: 
for any $1 \leq p \leq \infty$ we have a bounded linear operator 
\begin{gather*}
    \Poinc_{\mia} : L^{p}\Alt^{\mia}(\Omega) \rightarrow L^{p}\Alt^{\mia-1}(\bbR^n).
\end{gather*}

\begin{remark}
    As already explained in the original work of Costabel and McIntosh~\cite{costabel2010bogovskiui}, the operators $\Poinc_{\mia}$ preserve polynomial differential forms.
\end{remark}

\subsection{An operator norm bound with respect to the Lebesgue norm: the Bogovski\u{\i} case}

We analyze the Bogovski\u{\i} potential operator by similar means. 
Suppose that $u \in L^{\infty}\Alt^{\mia}(\bbR^{n})$ with $\supp u \subseteq \overline\Omega$ and that $x \in \bbR^{n}$.
First, if $x \notin \Omega$, then the convexity of $\Omega$ implies that $y + t( x - y ) \notin \Omega$ for all $t > 1$. Hence, $\calG_{\mia}(x,y) = 0$ and therefore $\Bogov_{\mia} u(x) = 0$ in that case.
Consider now the case $x \in \overline\Omega$. 
We estimate $\Bogov_{\mia} u(x)$ pointwise by 
\begin{align*}
    \left| \Bogov_{\mia} u(x) \right|
    &=
    \left| 
        \int_{\Omega} \calG_{\mia}(x,y) \,(x-y)\lrcorner u(y)\,dy
    \right| 
    \\&\leq 
    \int_{\Omega} \vol(\Omega)^{-1} \sum_{\jeza=0}^{\mia-1}\tbinom{\mia-1}{\jeza} \frac{ \diam(\Omega)^{n-\jeza} }{n-\jeza} |x-y|^{\jeza+1-n} |u(y)| \,dy
    \\&\leq 
    \int_{\bbR^{n}} \vol(\Omega)^{-1} \sum_{\jeza=0}^{\mia-1}\tbinom{\mia-1}{\jeza} \frac{ \diam(\Omega)^{n-\jeza} }{n-\jeza} \chi_{B_{\diam(\Omega)}(0)}(x-y) |x-y|^{\jeza+1-n} |u(y)| \,dy
    .
\end{align*}
Using once more the radial integrals discussed above, we compute 
\begin{align*}
    &
    \int_{\bbR^{n}} \sum_{\jeza=0}^{\mia-1} \tbinom{\mia-1}{\jeza} \frac{ \diam(\Omega)^{n-\jeza} }{n-\jeza} \chi_{B_{\diam(\Omega)}(0)}(z) |z|^{\jeza+1-n} \,dz
    \\&
    =
    \sum_{\jeza=0}^{\mia-1} \tbinom{\mia-1}{\jeza} \frac{ \diam(\Omega)^{n-\jeza} }{n-\jeza} \int_{B_{\diam(\Omega)}(0)} |z|^{\jeza+1-n} \,dz
    \\&
    =
    \vol_{n-1}(S_1) \sum_{\jeza=0}^{\mia-1} \tbinom{\mia-1}{\jeza} \frac{ \diam(\Omega)^{n-\jeza} }{n-\jeza} \frac{\diam(\Omega)^{\jeza+1}}{\jeza+1}
=
    \vol_{n-1}(S_1) \diam(\Omega)^{n+1} \underbrace{ \sum_{\jeza=0}^{\mia-1} \frac{ \tbinom{\mia-1}{\jeza} }{ (n-\jeza)(\jeza+1) } }_{ =: C_{\Bogov}(n,\mia) \leq 2^{\mia-1}}
    .
\end{align*}
Here we introduce the numerical constant 
\begin{gather}\label{equation:A_Bogov}
    C_{\Bogov}(n,\mia) := \sum_{\jeza=0}^{\mia-1} \frac{ \tbinom{\mia-1}{\jeza} }{ (n-\jeza)(\jeza+1) }, 
\end{gather}
which depends only on $n$ and $\mia$ and which is bounded by $2^{\mia-1}$. 
Similar as above,
the integral $\Bogov_{\mia} u(x)$ is absolutely convergent for any choice of $x \in \bbR^{n}$
and its magnitude is pointwise dominated by the convolution of $|u|$ against an integrable function.
Young's convolution inequality now implies: 
\begin{align*}
    \| \Bogov_{\mia} u \|_{L^{p}(\Omega)}
    &
    \leq 
    \vol_{n-1}(S_1) C_{\Bogov}(n,\mia) \frac{ \diam(\Omega)^{n} }{\vol(\Omega)} 
    \diam(\Omega)
    \| u \|_{L^{p}(\Omega)}
    \\&
    \leq 
    n C_{\Bogov}(n,\mia) \frac{ \vol_{}(B_{\diam(\Omega)}(0)) }{\vol(\Omega)} 
    \diam(\Omega)
    \| u \|_{L^{p}(\Omega)}
    .
\end{align*}
We have assumed so far that $u$ is essentially bounded.
Since that space is dense in the Lebesgue spaces, a density argument yields: 
for any $1 \leq p \leq \infty$ we have a bounded linear operator 
\begin{gather*}
    \Bogov_{\mia} : L^{p}\Alt^{\mia}(\Omega) \rightarrow L^{p}\Alt^{\mia-1}(\bbR^n).
\end{gather*}
Moreover, $\supp \Bogov_{\mia} u \subseteq \overline\Omega$,
that is, the reconstructed potential has support contained within $\overline\Omega$.

\subsection{Rewriting the potentials operators}

More properties of these operators become apparent after a change of variables. 
We write down the full definition of these operators and perform two substitutions.
For the Poincar\'e operator, we substitute $a = x + t(y-x)$, and then we substitute $s = (t-1)/t$,
leading to 
\begin{align*}
    \Poinc_{\mia} u(x) 
    &= 
    \vol(\Omega)^{-1}
    \int_{\Omega} \int_{1}^\infty (t-1)^{\mia-1}t^{n-\mia} 
    \chi_{\Omega}\left(x+t(y-x)\right) 
    (x-y)\lrcorner u(y) \,dt\,dy 
    \\
    &=
    \vol(\Omega)^{-1}
    \int_{\bbR^{n}} \chi_{\Omega}(a) \,(x-a)\lrcorner \int_{0}^{1} t^{\mia-1} u\left(a+t(x-a)\right)\,dt\,da
    .
\end{align*}
For the Bogovski\u{\i} operator, we first substitute $a = y + t(x-y)$, and then we substitute $s = t/(t-1)$,
leading to 
\begin{align*}
    \Bogov_{\mia} u(x) 
    &= 
    \vol(\Omega)^{-1}
    \int_{\Omega} \int_{1}^\infty (t-1)^{n-\mia}t^{\mia-1} 
    \chi_{\Omega}\left(y+t(x-y)\right) 
    (x-y)\lrcorner u(y) \,dt\,dy 
    \\
    &=
    - 
    \vol(\Omega)^{-1}
    \int_{\bbR^{n}} \chi_{\Omega}(a) \,(x-a)\lrcorner \int_{1}^\infty t^{\mia-1} u\left(a+t(x-a)\right)\,dt\,da
    .
\end{align*}
Given $a \in \Omega$, we introduce the potentials 
\begin{align*}
    \Poinc_{\mia,a} u(x) 
    &:= 
    (x-a)\lrcorner \int_{0}^{1} t^{\mia-1} u\left(a+t(x-a)\right)\,dt
    ,
    \\
    \Bogov_{\mia,a} u(x) 
    &:= 
    - (x-a)\lrcorner \int_{1}^\infty t^{\mia-1} u\left(a+t(x-a)\right)\,dt
    .
\end{align*}
By definition,
\begin{align*}
    \Poinc_{\mia} u(x) 
    =
    \vol(\Omega)^{-1}
    \int_{\Omega} \Poinc_{\mia,a} u(x) \,da,
    \quad 
    \Bogov_{\mia} u(x) 
    =
    \vol(\Omega)^{-1}
    \int_{\Omega} \Bogov_{\mia,a} u(x) \,da
    .
\end{align*}

\subsection{Interaction of potentials with the exterior derivative}

We study the interaction of the Poincar\'e and Bogovski\u{\i} operators with the exterior derivative in more detail.
The main arguments are well-known and establish the exactness of several de~Rham complexes. 
We recapitulate these arguments since our variants of the regularized potential operators are not yet included in the published literature. 

We make use of the following notation~\cite{licht2022basis}:
for any mapping $\sigma \in \Sigma(k,n)$, we let $[\sigma] := \{ \sigma(1), \dots, \sigma(k) \}$ be its image. 
When $p \in [\sigma]$, then the member of $\Sigma(k-1,n)$ with image $[\sigma] \setminus \{p\}$ is written $\sigma-p$.

Suppose that $u \in C^{\infty}\Alt^{\mia}(\bbR^n)$.
We rewrite the Poincar\'e potential,
\begin{align*}
    \Poinc_{\mia,a} u(x) 
    &= 
    (x-a)\lrcorner \int_{0}^{1} t^{\mia-1} u\left(a+t(x-a)\right)\,dt 
    \\&
    = 
    (x-a)\lrcorner 
    \sum_{\sigma \in \Sigma(\mia,n)}
    \int_{0}^{1} 
    t^{\mia-1} u_{\sigma}\left(a+t(x-a)\right) \cartanx^{\sigma} \,dt 
    \\&
    = 
    \sum_{\sigma \in \Sigma(\mia,n)} \sum_{i=1}^{\mia}
    \int_{0}^{1} 
    t^{\mia-1} u_{\sigma}\left(a+t(x-a)\right) (-1)^{i-1} (x-a)_{\sigma(i) } \,\cartanx^{\sigma-\sigma(i)} \,dt 
    ,
\end{align*}
and compute its exterior derivative:
\begin{align*}
    \cartan \Poinc_{\mia,a} u(x) 
    &= 
    \sum_{\sigma \in \Sigma(\mia,n)} 
    \int_{0}^{1} 
    \mia t^{\mia-1} u_{\sigma}\left(a+t(x-a)\right) \,\cartanx^{\sigma} \,dt 
    \\&\qquad
    + 
    \sum_{\substack{ \sigma \in \Sigma(\mia,n), \; 1 \leq i \leq \mia \\ 1 \leq j \leq n, \; j \notin [\sigma-\sigma(i)] }}
\int_{0}^{1} 
    t^{\mia} \frac{ \partial u_{\sigma} }{\partial x_{j}}\left(a+t(x-a)\right) (-1)^{i-1} (x-a)_{\sigma(i) } \,\cartanx^{j} \wedge \cartanx^{\sigma-\sigma(i)} \,dt 
    .
\end{align*}
We write the exterior derivative of $u$ as 
\begin{align}\label{math:exteriorderivative_of_u_in_hodgediscussion}
    \cartan u(x)
    =
    \sum_{\sigma \in \Sigma(\mia,n)} \sum_{j=1}^{n}
    \frac{ \partial u }{\partial x_{j}}(x) \,\cartanx^{j} \wedge \cartanx^{\sigma} 
    =
    \sum_{\substack{ \sigma \in \Sigma(\mia,n) \\ 1 \leq j \leq n, \; j \notin [\sigma] }}
    \frac{ \partial u }{\partial x_{j}}(x) \,\cartanx^{j} \wedge \cartanx^{\sigma} 
    ,
\end{align}
and apply the Poincar\'e potential operator to this result, which gives 
\begin{align*}
    \Poinc_{\mia+1,a} \cartan u(x)
    &=
    (x-a)\lrcorner 
    \sum_{\substack{ \sigma \in \Sigma(\mia,n) \\ 1 \leq j \leq n, \; j \notin [\sigma] }}
\int_{0}^{1} t^{\mia} \frac{ \partial u_{\sigma} }{\partial x_{j}}\left(a+t(x-a)\right) \,dt 
    \,\cartanx^{j} \wedge \cartanx^{\sigma}
    \\&
    = 
    \sum_{\substack{ \sigma \in \Sigma(\mia,n) \\ 1 \leq j \leq n, \; j \notin [\sigma] }} 
    \int_{0}^{1} t^{\mia} \frac{ \partial u_{\sigma} }{\partial x_{j}}\left(a+t(x-a)\right) \,dt (x-a)_{j}
    \,\cartanx^{\sigma} 
    \\&\qquad 
    - 
    \sum_{\substack{ \sigma \in \Sigma(\mia,n) \\ 1 \leq j \leq n, \; j \notin [\sigma] \\ 1 \leq i \leq k }}
(-1)^{i-1}
    \int_{0}^{1} t^{\mia} \frac{ \partial u_{\sigma} }{\partial x_{j}}\left(a+t(x-a)\right) \,dt 
    \,(x-a)_{\sigma(i)} 
    \,\cartanx^{j} \wedge \cartanx^{\sigma-\sigma(i)}
    .
\end{align*}
Next, we add the exterior derivative of the potential and the potential of the exterior derivative.
Taking into account cancellations, this gives the identity 
\begingroup\allowdisplaybreaks
\begin{align*}
    &
    \cartan \Poinc_{\mia,a} u(x)
    +
    \Poinc_{\mia+1,a} \cartan u(x)
    \\&
    =
    \sum_{\sigma \in \Sigma(\mia,n)} 
    \int_{0}^{1} 
    \mia t^{\mia-1} u_{\sigma}\left(a+t(x-a)\right) \,\cartanx^{\sigma-\sigma(i)} \,dt 
    \\&\qquad
    + 
    \sum_{\substack{ \sigma \in \Sigma(\mia,n), \; 1 \leq i \leq \mia \\ 1 \leq j \leq n, \; j \notin [\sigma-\sigma(i)] }}
\int_{0}^{1} 
    t^{\mia} \frac{ \partial u_{\sigma} }{\partial x_{j}}\left(a+t(x-a)\right) (-1)^{i-1} (x-a)_{\sigma(i) } \,\cartanx^{j} \wedge \cartanx^{\sigma-\sigma(i)} \,dt 
    \\&\qquad
    +
    \sum_{\substack{ \sigma \in \Sigma(\mia,n) \\ 1 \leq j \leq n, \; j \notin [\sigma] }} 
    \int_{0}^{1} t^{\mia} \frac{ \partial u_{\sigma} }{\partial x_{j}}\left(a+t(x-a)\right) \,dt \,(x-a)_{j} \,\cartanx^{\sigma}
    \\&\qquad
    - 
    \sum_{\substack{ \sigma \in \Sigma(\mia,n) \\ 1 \leq j \leq n, \; j \notin [\sigma] \\ 1 \leq i \leq k }}
(-1)^{i-1}
    \int_{0}^{1} t^{\mia} \frac{ \partial u_{\sigma} }{\partial x_{j}}\left(a+t(x-a)\right) \,dt 
    \,(x-a)_{\sigma(i)} \,\cartanx^{j} \wedge \cartanx^{\sigma-\sigma(i)}
    \\&
    =
    \sum_{\sigma \in \Sigma(\mia,n)} 
    \int_{0}^{1} 
    \mia t^{\mia-1} u_{\sigma}\left(a+t(x-a)\right) \,\cartanx^{\sigma-\sigma(i)} \,dt 
    \\&\qquad
    +
    \sum_{\substack{ \sigma \in \Sigma(\mia,n) \\ 1 \leq j \leq n, \; j \notin [\sigma] }} 
    \int_{0}^{1} t^{\mia} \frac{ \partial u_{\sigma} }{\partial x_{j}}\left(a+t(x-a)\right) \,(x-a)_{j} \,dt \,\cartanx^{\sigma}
    \\&\qquad
    + 
    \sum_{\substack{ \sigma \in \Sigma(\mia,n), \\ 1 \leq i \leq \mia }}
    \int_{0}^{1} 
    t^{\mia} \frac{ \partial u_{\sigma} }{\partial x_{\sigma(i)}}\left(a+t(x-a)\right) (-1)^{i-1} (x-a)_{\sigma(i) } \,\cartanx^{\sigma(i)} \wedge \cartanx^{\sigma-\sigma(i)} \,dt 
    \\&
    =
    \sum_{\sigma \in \Sigma(\mia,n)} 
    \int_{0}^{1} \frac{\partial}{\partial t} \left( t^{\mia} u_{\sigma}\left(a+t(x-a)\right) \right) \,dt \,\cartanx^{\sigma}
    =
    \sum_{\sigma \in \Sigma(\mia,n)} 
    \left( u_{\sigma}(x) - 0^k u_{\sigma}(a) \right) \,\cartanx^{\sigma}
    .
\end{align*}
\endgroup
We conclude that, 
\begin{alignat*}{4}
    u(x) &= \Poinc_{1,a} \cartan u(x) - u(a), \qquad                              && k = 0,
    \\
    u(x) &= \cartan \Poinc_{\mia,a} u(x) + \Poinc_{\mia+1,a} \cartan u(x), \qquad && 1 \leq \mia \leq n.
\end{alignat*}
In summary, after taking the average over $a \in \Omega$:
\begin{alignat}{4}\label{math:hodgedecomposition:poinc}
    u(x) &= \Poinc_{1} \cartan u(x) - \vol(\Omega)^{-1} \int_\Omega u(a) \,da,  \qquad && k = 0,
    \\
    u(x) &= \cartan \Poinc_{\mia} u(x) + \Poinc_{\mia+1} \cartan u(x),          \qquad && 1 \leq \mia \leq n.
\end{alignat}

Even though the discussion for the Bogovski\u{\i} operator is large analogous, some modifications are needed. 
Suppose that $u \in C^{\infty}\Alt^{\mia}(\bbR^n)$ with $\supp u \subseteq \overline\Omega$.
We rewrite the Bogovski\u{\i} potential,
\begin{align*}
    \minus 
    \Bogov_{\mia,a} u(x) 
    &= 
    (x-a)\lrcorner \int_{1}^\infty t^{\mia-1} u\left(a+t(x-a)\right)\,dt 
    \\&
    = 
    (x-a)\lrcorner 
    \sum_{\sigma \in \Sigma(\mia,n)}
    \int_{1}^\infty 
    t^{\mia-1} u_{\sigma}\left(a+t(x-a)\right) \,\cartanx^{\sigma} \,dt 
    \\&
    = 
    \sum_{\sigma \in \Sigma(\mia,n)} \sum_{i=1}^{\mia}
    \int_{1}^\infty 
    t^{\mia-1} u_{\sigma}\left(a+t(x-a)\right) (-1)^{i-1} \,(x-a)_{\sigma(i) } \,\cartanx^{\sigma-\sigma(i)} \,dt 
    .
\end{align*}
We want to take its exterior derivative, but we can generally only do that in the distributional sense. 
Away from the pivot point $a$, the form $\Bogov_{\mia,a} u(x)$ is differentiable in $x$, 
and so we compute its exterior derivative over $\Omega \setminus \{a\}$:
\begin{align*}
    \minus 
    \cartan \Bogov_{\mia,a} u(x) 
    &= 
    \sum_{\sigma \in \Sigma(\mia,n)} 
    \int_{1}^\infty 
    \mia t^{\mia-1} u_{\sigma}\left(a+t(x-a)\right) \,\cartanx^{\sigma} \,dt 
    \\&\qquad
    + 
    \sum_{\substack{ \sigma \in \Sigma(\mia,n), \; 1 \leq i \leq \mia \\ 1 \leq j \leq n, \; j \notin [\sigma-\sigma(i)] }}
\int_{1}^\infty 
    t^{\mia} \frac{ \partial u_{\sigma} }{\partial x_{j}}\left(a+t(x-a)\right) (-1)^{i-1} (x-a)_{\sigma(i) } \,\cartanx^{j} \wedge \cartanx^{\sigma-\sigma(i)} \,dt 
    .
\end{align*}
The derivative of $\Bogov_{\mia,a} u$ over the whole domain, which is what we need, can only be taken in the sense of distributions. 
Let $\phi \in C^{\infty}\Alt^{n-\mia}(\Omega)$ be smooth and compactly supported over $\Omega$. 
We let $\epsilon > 0$ and calculate 
\begin{align*}
    (-1)^{n(k-1)} \int_{ \Omega \setminus B_\epsilon(a) }
    \Bogov_{\mia,a} u(x) \wedge \cartan \phi 
    =
    \int_{ S_\epsilon(a) }
    \trace_{S_\epsilon(a)}{ \Bogov_{\mia,a} u(x) \wedge \phi } 
    -
    \int_{ \Omega \setminus B_\epsilon(a) }
    \cartan \Bogov_{\mia,a} u(x) \wedge \phi 
    ,
\end{align*}
where $\trace_{S_{\epsilon}(a)}$ denotes the trace onto the sphere $S_{\epsilon}(a)$. 
In the limit as $\epsilon$ goes to zero, the two integrals over $\Omega \setminus B_\epsilon(a)$ in the above equation converge to the integrals of the respective integrands over $\Omega \setminus \{a\}$. 
To understand the derivative of $\Bogov_{\mia,a} u$ over the domain $\Omega$, 
we study the remaining surface integral. 
We apply several substitutions:
\begin{align*}
    &\int_{ S_\epsilon(a) }
    \trace_{S_\epsilon(a)}{ \Bogov_{\mia,a} u(x) \wedge \phi(x) } 
    \\&\quad 
    =
\epsilon^{n-1}
    \int_{ S_1(a) }
    \trace_{S_1(a)}{ \Bogov_{\mia,a} u( \epsilon x + a - \epsilon a ) \wedge \phi( \epsilon x + a - \epsilon a ) }
    \\&\quad 
    =
    \epsilon^{n-1}
    \int_{ S_1(a) }
    \int_{1}^\infty 
    \trace_{S_1(a)}{ 
        \epsilon t^{\mia-1} (x-a)\lrcorner u\left( a + \epsilon t( x - a )\right) 
        \wedge \phi( \epsilon (x-a) + a ) 
    }
    \,dt 
    \\&\quad 
    =
\epsilon^{n-1}
    \int_{ S_1(a) }
    \int_{\epsilon}^\infty 
    \trace_{S_1(a)}{ 
        \epsilon^{-\mia+1} s^{\mia-1} (x-a)\lrcorner u\left( a + s( x - a )\right) 
        \wedge \phi( \epsilon (x-a) + a ) 
    }
    \,ds 
    \\&\quad 
    =
\epsilon^{n-\mia}
    \int_{ S_1(a) }
    \int_{\epsilon}^\infty 
    \trace_{S_1(a)}{ 
        s^{\mia-1} (x-a)\lrcorner u\left( a + s( x - a )\right) 
        \wedge \phi( \epsilon (x-a) + a ) 
    }
    \,ds 
    .
\end{align*}
We make a case distinction. When $k<n$, then the double integral itself is bounded uniformly in $\epsilon > 0$ and so the last expression vanishes as $\epsilon$ goes to zero. When instead $k=n$, then the last expression equals 
\begin{align*}
    & 
    \int_{ S_1(a) }
    \int_{\epsilon}^\infty 
    \trace_{S_1(a)}{ 
        s^{n-1} (x-a)\lrcorner u\left( a + s( x - a )\right) 
        \wedge \phi( \epsilon (x-a) + a ) 
    }
    \,ds 
    \\&\quad 
    =
    \int_{\bbR^{n} \setminus B_{\epsilon}(0)}
        u\left( y \right) 
        \wedge \phi\left( \epsilon \frac{y-a}{ \vecnorm{y-a} } + a \right) 
    \,dy 
    .
\end{align*}
The limit of this is $\phi(a) \int_\Omega u(x)$ as $\epsilon$ goes to zero. 
To complete the discussion, we write the exterior derivative of $u$ as in~\eqref{math:exteriorderivative_of_u_in_hodgediscussion}, 
and apply the Bogovski\u{\i} potential operator to this result, which gives 
\begin{align*}
    \minus 
    \Bogov_{\mia+1,a} \cartan u(x)
    &=
    (x-a)\lrcorner 
    \sum_{\sigma \in \Sigma(\mia,n)} \sum_{j=1}^{n}
    \int_{1}^\infty t^{\mia} \frac{ \partial u_{\sigma} }{\partial x_{j}}\left(a+t(x-a)\right) \,dt 
    \,\cartanx^{j} \wedge \cartanx^{\sigma}
    \\&
    = 
    \sum_{\substack{ \sigma \in \Sigma(\mia,n) \\ 1 \leq j \leq n, \; j \notin [\sigma] }} 
    \int_{1}^\infty t^{\mia} \frac{ \partial u_{\sigma} }{\partial x_{j}}\left(a+t(x-a)\right) \,dt \,(x-a)_{j} \,\cartanx^{\sigma}
    \\&\qquad 
    - 
    \sum_{\substack{ \sigma \in \Sigma(\mia,n) \\ 1 \leq j \leq n, \; j \notin [\sigma] \\ 1 \leq i \leq k }}
(-1)^{i-1}
    \int_{1}^\infty t^{\mia} \frac{ \partial u_{\sigma} }{\partial x_{j}}\left(a+t(x-a)\right) \,dt 
    (x-a)_{\sigma(i)} \,\cartanx^{j} \wedge \cartanx^{\sigma-\sigma(i)}
    .
\end{align*}
We add the (distributional) exterior derivative of the potential and the potential of the exterior derivative.
In a manner that fully analogous to the discussion of the averaged Poincar\'e operator save for the modification when $k=n$, 
we come to the conclusion that 
\begin{alignat*}{4}
    u(x) &= \cartan \Bogov_{\mia,a} u(x) + \Bogov_{\mia+1,a} \cartan u(x),          \qquad && 0 \leq \mia \leq n-1,
    \\
    u(x) &= \cartan \Bogov_{n,a} u(x) - \left( \int_{\Omega} u(a) \right) \delta_a, \qquad && k = n.
\end{alignat*}
Here, $\delta_a$ denotes the Dirac delta at $a$. 
In summary, after taking the average over $a \in \Omega$:
\begin{alignat}{4}\label{math:hodgedecomposition:bogov}
    u(x) &= \cartan \Bogov_{\mia} u(x) + \Bogov_{\mia+1} \cartan u(x),                                   \qquad && 0 \leq \mia \leq n-1,
    \\
    u(x) &= \cartan \Bogov_{n} u(x) - \left( \vol(\Omega)^{-1} \int_{\Omega} u(a) \right) \chi_{\Omega}, \qquad && k = n.
\end{alignat}

\subsection{Operator norms as bounds for the Poincar\'e--Friedrichs constants}

We are now ready to state the main results of this section. 
Recall that $\diam(\Omega) > 0$ is the diameter of $\Omega$ and that $B_{\diam(\Omega)}(0)$ is the $n$-dimensional ball centered at the origin.
The following upper bounds for the Poincar\'e--Friedrichs constants are proportional to the domain diameter and are independent of the Lebesgue exponent $1 \leq p \leq \infty$.
However, the space dimension $n$ and the form degree $\mia$ enter the estimates, 
namely through definitions~\eqref{equation:A_Poinc} and~\eqref{equation:A_Bogov} of respectively $C_{\Poinc}(n,\mia)$ and $C_{\Bogov}(n,\mia)$.

\begin{theorem}\label{theorem:bogovpoinc}
    Let $\Omega \subseteq \bbR^{n}$ be a bounded convex open set and let $1 \leq p \leq \infty$. 
    We have bounded operators 
    \begin{align*}
        \Poinc_{\mia} : L^{p}\Alt^{\mia}(\Omega) \rightarrow W^{p}\Alt^{\mia-1}(\Omega)
        ,
        \qquad 
        \Bogov_{\mia} : L^{p}\Alt^{\mia}(\Omega) \rightarrow W^{p}_{0}\Alt^{\mia-1}(\Omega)
        .
    \end{align*}
    They satisfy the operator norm bounds 
    \begin{align*}
        \| \Poinc_{\mia} u \|_{L^{p}(\Omega)}
        &\leq 
C_{{\PF},\Poinc,\Omega,k,p}
        \| u \|_{L^{p}(\Omega)}
        ,
        \\
        \| \Bogov_{\mia} u \|_{L^{p}(\Omega)}
        &\leq 
C_{{\PF},\Bogov,\Omega,k,p}
        \| u \|_{L^{p}(\Omega)}
        ,
    \end{align*}
    where 
    \begin{align}
        C_{{\PF},\Poinc,\Omega,k,p} &:= C_{\Poinc}(n,\mia) \vol_{n-1}(S_{1}(0)) \frac{\diam(\Omega)^{n}}{\vol(\Omega)} \diam(\Omega)
        ,
        \\
        C_{{\PF},\Bogov,\Omega,k,p} &:= C_{\Bogov}(n,\mia) \vol_{n-1}(S_{1}(0)) \frac{\diam(\Omega)^{n}}{\vol(\Omega)} \diam(\Omega)
        .
    \end{align}
    For any $u \in W^{p}\Alt^{\mia}(\Omega)$ it holds for a.e.\ $x \in \Omega$ that 
    \begin{alignat}{4}
        u(x) &= \Poinc_{1} \cartan u(x) - \vol(\Omega)^{-1} \int_\Omega u(a) \;da, && \qquad k = 0,
        \label{math:theorem:bogovpoinc:poincare:zero}
        \\
        u(x) &= \cartan \Poinc_{\mia} u(x) + \Poinc_{\mia+1} \cartan u(x),        && \qquad 1 \leq \mia \leq n.
        \label{math:theorem:bogovpoinc:poincare:k}    
    \end{alignat}
    In particular, if $u \in W^{p}    \Alt^{\mia}(\Omega)$ with $0 \leq \mia \leq n-1$, then $w = \Poinc_{\mia} \cartan u$ satisfies $\cartan w = \cartan u$.
    
    For any $u \in W^{p}_{0}\Alt^{\mia}(\Omega)$ it holds for a.e.\ $x \in \Omega$ that 
    \begin{alignat}{4}
        u(x) &= \cartan \Bogov_{\mia} u(x) + \Bogov_{\mia+1} \cartan u(x),
        \qquad && 0 \leq \mia \leq n-1,
        \label{math:theorem:bogovpoinc:bogov:k}
        \\
        u(x) &= \cartan \Bogov_{n} u(x) - \left( \vol(\Omega)^{-1} \int_{\Omega} u(a) \right) \chi_{\Omega},
        \qquad && k=n.
        \label{math:theorem:bogovpoinc:bogov:n}
    \end{alignat}
    In particular, if $u \in W^{p}_{0}\Alt^{\mia}(\Omega)$ with $0 \leq \mia \leq n-1$, then $w = \Bogov_{\mia} \cartan u$ satisfies $\cartan w = \cartan u$.

\end{theorem}

\begin{proof}
    Consider the case $1 \leq p < \infty$.
Because the domain is bounded, 
    the subspaces $C^{\infty}_{c}\Alt^{\mia}(\Omega)$ are dense in $L^{p}\Alt^{\mia}(\Omega)$,
    and so the stated operator norm bounds follow by an approximation argument. 
Taking the limit on both sides of equation then implies the inequality with $p = \infty$
    because the $L^{\infty}$ norm is the limit of the Lebesgue norms as $p$ goes to infinity.

    Consider now $u \in W^{p}\Alt^{\mia}(\Omega)$ with $1 \leq k \leq n$. 
    We write $\zeta := \Poinc_{\mia  } u$. 
    There exists a sequence $u_{i} \in C^{\infty}\Alt^{\mia}(\overline\Omega)$ that converges to $u$ in $W^{p}\Alt^{\mia}(\Omega)$. 
    For any test form $v \in C^{\infty}_{c}\Alt^{n-\mia-1}(\Omega)$, we verify 
    \begin{align*}
        \int_{\Omega} v \wedge u_{i} 
        &=
        \int_{\Omega} v \wedge \Poinc_{\mia+1} \cartan u_{i}
        +
        \int_{\Omega} v \wedge \cartan \Poinc_{\mia  } u_{i}
        \\&=
        \int_{\Omega} v \wedge \Poinc_{\mia+1} \cartan u_{i}
        +
        (-1)^{\mia(n-\mia)+1}
        \int_{\Omega} \cartan v \wedge \Poinc_{\mia  } u_{i}
        .
    \end{align*}
    By the continuity of bounded linear functionals, we find 
    \begin{align*}
        \int_{\Omega} v \wedge u 
        -
        \int_{\Omega} v \wedge \Poinc_{\mia+1} \cartan u 
        &=
        (-1)^{\mia(n-\mia)+1}
        \int_{\Omega} \cartan v \wedge \Poinc_{\mia  } u 
        .
    \end{align*}
    Hence, by definition, $\zeta \in W^{p}\Alt^{\mia-1}(\Omega)$ with $\cartan \zeta = u - \Poinc_{\mia+1} \cartan u$.
    This shows~\eqref{math:theorem:bogovpoinc:poincare:k}, and~\eqref{math:theorem:bogovpoinc:poincare:zero} follows by an approximation argument.

    Analogously, suppose that $u \in W^{p}_{0}\Alt^{\mia}(\Omega)$ with $0 \leq \mia \leq n-1$. 
    We write $\zeta := \Bogov_{\mia} u$. 
    There exists a sequence $u_{i} \in C^{\infty}_{c}\Alt^{\mia}(\overline\Omega)$ that converges to $u$ in $W^{p}\Alt^{\mia}(\Omega)$. 
    For any test form $v \in C^{\infty}\Alt^{n-\mia-1}(\Omega)$, 
    which is the restriction of some member of $C^{\infty}\Alt^{n-\mia-1}(\bbR^{n})$, 
    it holds that 
    \begin{align*}
        \int_{\Omega} v \wedge u_{i} 
        &=
        \int_{\Omega} v \wedge \Bogov_{\mia+1} \cartan u_{i}
        +
        \int_{\Omega} v \wedge \cartan \Bogov_{\mia  } u_{i}
        \\&=
        \int_{\Omega} v \wedge \Bogov_{\mia+1} \cartan u_{i}
        +
        (-1)^{\mia(n-\mia)+1}
        \int_{\Omega} \cartan v \wedge \Bogov_{\mia  } u_{i}
        .
    \end{align*}
    By the continuity of bounded linear functionals, we find 
    \begin{align*}
        \int_{\Omega} v \wedge u 
        -
        \int_{\Omega} v \wedge \Bogov_{\mia+1} \cartan u
        &=
        (-1)^{\mia(n-\mia)+1}
        \int_{\Omega} \cartan v \wedge \Bogov_{\mia  } u 
        .
    \end{align*}
    By definition, $\zeta \in W^{p}_{0}\Alt^{\mia-1}(\Omega)$ with $\cartan \zeta = u - \Bogov_{\mia+1} \cartan u$.
    This shows~\eqref{math:theorem:bogovpoinc:bogov:k}, and~\eqref{math:theorem:bogovpoinc:bogov:n} follows by an approximation argument.
    
    Everything else is now apparent, and the proof is complete. 
\end{proof}

These constants are generally not optimal. 
For example, when $p=2$ and when only the divergence is considered, 
we have the following improved estimate. 

\begin{lemma}\label{lemma:PFfordivergence}
    Let $\Omega \subseteq \bbR^{n}$ be a bounded open set. 
For each $\bfu \in \bfW^{2}(\divergence,\Omega)$ there exists $\bfw \in \bfW^{2}(\divergence,\Omega)$ with $\divergence \bfw = \divergence \bfu$ and 
    \begin{align*}
        \| \bfw \|_{L^{2}(\Omega)} 
        \leq 
        \diam(\Omega) \| \divergence \bfu \|_{L^{2}(\Omega)}
        .
    \end{align*}
\end{lemma}
\begin{proof}
    This is a reduction to the Friedrichs inequality. 
    The space $W^{1,2}_{0}(\Omega)$ is the closure of the smooth functions with support in $\Omega$ in the Hilbert space $W^{1,2}(\Omega)$. 
    Then $\nabla : W^{1,2}_{0}(\Omega) \subseteq L^{2}(\Omega) \rightarrow \bfL^{2}(\Omega)$ is a closed densely-defined linear operator 
    whose smallest singular value is bounded from below by $\diam(\Omega)^{-1}$, according to the Friedrichs inequality. 
    The adjoint is the closed densely-defined linear operator $-\divergence : \bfW^{2}(\divergence,\Omega) \subseteq \bfL^{2}(\Omega) \rightarrow L^{2}(\Omega)$,
    which has the same smallest singular value. 
\end{proof}

\begin{remark}\label{remark:reg_Poinc_Bog}
    The classical Poincar\'e operator is known for its role in proving the exactness of the smooth de~Rham complex over star-shaped domains~\cite{lee2012smooth}.
    The Bogovski\u{\i}-type operators were first studied for the divergence operator and are a staple in the mathematics of hydrodynamics~\cite{bogovskii1979solution}.
    Costabel and McIntosh~\cite{costabel2010bogovskiui} regularize the potentials by averaging over pivot points within an interior ball using a smooth compactly supported weight function, 
    which is why they can study domains star-shaped with respect to a ball. 
    Their operators are pseudo-differential operators of negative order, 
    because their averaging uses a smooth weight; this proves that the operators are bounded between a variety of function spaces. 
    Explicit bounds for the higher-order seminorms of these pseudo-differential operators 
    have been recently contributed by Guzman and Salgado~\cite{guzman2021estimation}. 
    Instead, we average over the entire domain, which requires a convex geometry, 
    and we are interested in boundedness in the Lebesgue $p$-norms. 
    We establish computable bounds on the Poincar\'e--Friedrichs constants,
    which had been absent from the literature, to the best of our knowledge.
\end{remark}

\section{Shellable triangulations of manifolds}\label{section:advancedtriangulations}

We return to the theory of triangulations, as our main objective requires some further concepts.
We are interested in simplicial complexes that triangulate domains and that are \emph{shellable}.
Such shellable simplicial complexes are constructed by successively adding simplices in a well-structured manner. 
Local patches (stars) within triangulations of dimension two or three are examples of shellable complexes. 
The monographs by Kozlov~\cite{kozlov2008combinatorial} and Ziegler~\cite{ziegler1995lectures} are our main references for this section. 
We also refer to Lee's monograph~\cite{lee2011topological} for any further background on manifolds. 

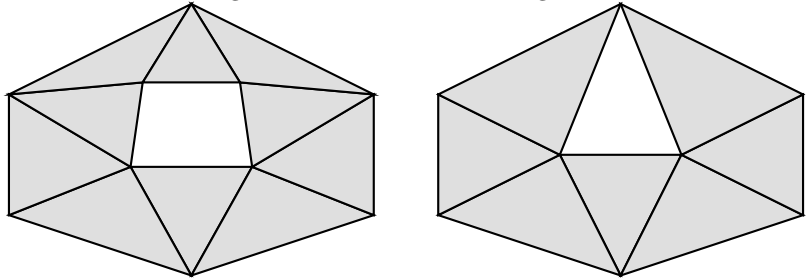
\begin{figure}[h]
\caption{Left: manifold triangulation of an annulus. Right: not a manifold triangulation.}\label{figure:not_manifold_triang}
\centering
\begin{tikzpicture}[scale=0.8]
\coordinate (B0)  at (0,-1); \coordinate (BL)  at (-3,0); \coordinate (BR)  at (3,0); \coordinate (CLL) at (-3,2); \coordinate (CL)  at (-1,0.8); \coordinate (CR)  at ( 1,0.8); \coordinate (CRR) at ( 3,2); \coordinate (Z)   at (0,3.5); \coordinate (ML) at  (-0.8,2.2);\coordinate (MR) at  ( 0.8,2.2);\draw[thick,fill=gray!25] (B0) -- (CL) -- (CR) -- cycle;
    \draw[thick,fill=gray!25] (B0) -- (CL) -- (BL) -- cycle;
    \draw[thick,fill=gray!25] (BL) -- (CL) -- (CLL) -- cycle;
    \draw[thick,fill=gray!25] (B0) -- (CR) -- (BR) -- cycle;
    \draw[thick,fill=gray!25] (BR) -- (CR) -- (CRR) -- cycle;
    \draw[thick,fill=gray!25] (ML) -- (CL) -- (CLL) -- cycle;
    \draw[thick,fill=gray!25] (MR) -- (CR) -- (CRR) -- cycle;
    \draw[thick,fill=gray!25] (Z) -- (ML) -- (CLL) -- cycle;
    \draw[thick,fill=gray!25] (Z) -- (MR) -- (CRR) -- cycle;
    \draw[thick,fill=gray!25] (Z) -- (ML) -- (MR) -- cycle;
\end{tikzpicture}
\qquad
\begin{tikzpicture}[scale=0.8]
\coordinate (B0)  at (0,-1); \coordinate (BL)  at (-3,0); \coordinate (BR)  at (3,0); \coordinate (CLL) at (-3,2); \coordinate (CL)  at (-1,1); \coordinate (CR)  at ( 1,1); \coordinate (CRR) at ( 3,2); \coordinate (Z)   at (0,3.5); \draw[thick,fill=gray!25] (B0) -- (CL) -- (CR) -- cycle;
    \draw[thick,fill=gray!25] (B0) -- (CL) -- (BL) -- cycle;
    \draw[thick,fill=gray!25] (BL) -- (CL) -- (CLL) -- cycle;
    \draw[thick,fill=gray!25] (B0) -- (CR) -- (BR) -- cycle;
    \draw[thick,fill=gray!25] (BR) -- (CR) -- (CRR) -- cycle;
    \draw[thick,fill=gray!25] (Z) -- (CL) -- (CLL) -- cycle;
    \draw[thick,fill=gray!25] (Z) -- (CR) -- (CRR) -- cycle;
\end{tikzpicture}
\end{figure}

\subsection{Triangulations of manifolds}\label{subsection:manifoldtriangulation}

Our discussion requires some notions and results concerning triangulated manifolds.
We define an $n$-dimensional simplicial complex to be a \emph{manifold triangulation} if the underlying set $\underlying{\calT}$ is an $n$-dimensional manifold, possibly with boundary.
We recall\footnote{In the usual terminology of topology and differential geometry, our notion of manifold is a manifold with boundary, even if the boundary is empty.}
that this means that for every $x \in \underlying{\calT}$
there exists an open neighborhood $U(x) \subseteq \underlying{\calT}$ and an embedding $\phi : U(x) \rightarrow \bbR^{n}$
such that $\phi(0) = 0$ and $\phi$ is an isomorphism either onto the open unit ball $\calB = \{ x \in \bbR^{n} \suchthat |x| < 1 \}$
or onto the half-ball $\{ x \in \calB \suchthat x_1 \geq 0 \}$.
In the former case, $x$ is called an \emph{interior point}, and in the latter case $x$ is called a \emph{boundary point}. 
Any simplicial complex that triangulates an $n$-dimensional manifold must be $n$-dimensional. 
Figure~\ref{figure:not_manifold_triang} shows an example of simplicial complex that triangulates a manifold and an example that does not.

The following special cases receive particular interest:
an \emph{$n$-ball triangulation} is any triangulation of a topological (closed) $n$-ball, and an \emph{$n$-sphere triangulation} is any triangulation of a topological $n$-sphere.

We know that any manifold $\calM$ has got a topological boundary $\partial\calM$, possibly empty. 
If $\calM$ is $n$-dimensional, then the $\partial\calM$ is a topological manifold without boundary of dimension $n-1$. 
We gather a few helpful observations on how these notions relate to triangulations.
While the reader might deem them obvious, we nevertheless include proofs. 

\begin{lemma}\label{lemma:boundarysimplices}
    Let $\calT$ be a finite $n$-dimensional simplicial complex whose underlying set is a manifold $\calM$. 
    Then the simplices contained in the boundary constitute a triangulation of the boundary.
    Moreover, if $F \in \calT$ is a face (i.e, $F$ has dimension $n-1$), then 
    \begin{itemize}
        \item $F$ is not contained in the boundary if and only if it is contained in exactly two $n$-simplices.
        \item $F$ is contained in the boundary if and only if it is contained in exactly one $n$-simplex.
\end{itemize}
\end{lemma}
\begin{proof}
    We prove these statements in several steps.
    \begin{enumerate}
    \item 
    Let $\mathring \calM := \calM \setminus \partial \calM$ denote the interior of the manifold. 
    We will use the following fact:\footnote{To see this, one easily constructs a continuous deformation of $\calM$ into itself to move any chosen point on $S$ to any other chosen point on $S$.} if $S \in \calT$ has an inner point that lies on $\partial \calM$, then all inner points of $S$ are on $\partial \calM$. 
    Since the boundary $\partial \calM$ is closed, every $S \in \calT$ is either a subset of the boundary or all its inner points lie in the interior $\mathring \calM$ of the manifold. 
    
    \item 
    We recall an auxiliary result.
    Suppose that $Y$ is a topological space homeomorphic to a sphere of dimension $m$ and that $X \subseteq Y$ is homeomorphic to a sphere of dimension $m-1$, where $m \geq 1$. 
    As a consequence of the Jordan--Brouwer separation theorem~\cite[Corollary IV.5.24]{mayer1989algebraische}, \cite[Corollary VIII.6.4]{massey1981algebraic}, 
    we know that $Y \setminus X$ has got two connected components.

    \item
    Let now $F \in \subsimplex_{n-1}(T)$ be a face and let $z_F \in F$ be its barycenter. 
    Since $\calT$ is finite, we let $\mathring B_F$ be an open neighborhood around $z_F$ 
    homeomorphic to an $n$-dimensional ball 
    so small that its closure $\overline{B_F}$ only intersects those $n$-simplices of $\calT$ that already contain $z_F$ and no faces other than $F$.
    Suppose there are distinct $n$-simplices $T_1, T_2, \dots, T_K$ that contain $z_F$. 
    The intersection of any two of them is $F$, but their interiors are disjoint because otherwise they would coincide. 
    
    If $z_F$ is an interior point of $\calM$, 
    then it follows by our assumptions that $\mathring B_F$ is homeomorphic to an open $n$-ball and $\partial \mathring B_F$ is homeomorphic to a sphere of dimension $n-1$. 
    Consider $X = F \cap \partial \mathring B_F$.
    If $n=1$, then $X$ is empty and $\partial \mathring B_F$ has $K$ distinct connected components. 
    If $n > 1$, then $X$ is homeomorphic to a sphere of dimension $n-2$
    and again $\partial \mathring B_F \setminus X$ has $K$ distinct connected components. 
    But by the auxiliary result above, $K = 2$. 
    We conclude that $F$ is contained in two $n$-simplices of $\calT$.
    
    Consider the case that $z_F$ lies on the boundary of $\calM$ and suppose that $F$ is contained in $K$ distinct $n$-simplices of $\calT$.
    By adding at least one dimension, we can double\footnote{The reader is referred to Lee's textbook~\cite{lee2012smooth} for more background and the technicalities.} the manifold $\calM$ along the boundary and obtain the doubled manifold $\calM'$. 
    Similarly, we can construct a doubling of the triangulation $\calT'$ such that $F$ is contained in exactly $2K$ distinct $n$-simplices of $\calT'$. 
    We know that $\calM'$ is a manifold without boundary, and hence $F$ is an interior simplex of $\calT'$.
    This implies $K=1$.
    So any boundary face can only be contained in one single $n$-simplex. 
    
    \item
    Clearly, the simplices of $\calT$ contained in the boundary constitute a simplicial complex. 
    Every $x \in \calM$ is an inner point of some simplex $S \in \calT$. 
    If $x \in \partial \calM$ is a boundary point, then $S$ must be a boundary simplex, 
    so the boundary simplices triangulate all of $\partial \calM$.
    \end{enumerate}
    All desired results are thus proven. 
\end{proof}

Suppose that $\calT$ is an $n$-dimensional simplicial complex that triangulates a manifold. 
Those simplices of the manifold triangulation that are subsets of the boundary of the underlying manifold are called \emph{boundary simplices}. 
All other simplices of the manifold triangulation are called \emph{inner simplices}. 
We have seen that the boundary simplices of a manifold triangulation constitute a triangulation of the manifold's boundary. 
We call this simplicial complex the \emph{boundary complex}. It has dimension $n-1$.

We continue with a few more observations about the topology of local patches (stars) of manifold triangulations. 
This topic is surprisingly non-trivial. We only gather some results that are hard to find in the literature. 

\begin{lemma}\label{lemma:startopology}
    Let $\calT$ be a finite $n$-dimensional simplicial complex whose underlying space is a manifold $\calM$.
    Suppose that $1 \leq n \leq 3$. Then the following holds:
    \begin{itemize}
        \item
        If $S \in \calT$ is an inner simplex, 
        then $\patch_{\calT}(S)$ is a simplicial $n$-ball with $S$ as an inner simplex
        and $\carapace_{\calT}(S)$ is a simplicial $(n-1)$-sphere. 
        \item
        If $S \in \calT$ is a boundary subsimplex, 
        then $\patch_{\calT}(S)$ is a simplicial $n$-ball with $S$ as a boundary simplex,
        and $\carapace_{\calT}(S)$ is a simplicial $(n-1)$-ball. \end{itemize}
\end{lemma}
\begin{proof}
    The statement is obvious if $n = 1$, so we assume $n \geq 2$ in what follows.
    We prove these statements in several steps. 
    The reader is assumed to have some background in topology. 
    \begin{enumerate}
    \item 
    Let $S$ be any simplex with vertices $v_0, v_1, \dots, v_k$, with barycenter $z_{S}$, and dimension $k$.
    Let $\calS := \patch_{\calT}(S)$ be its star. 
    Each $l$-dimensional simplex $T \in \calS$ that contains $S$ 
    has vertices $v_0$, $v_1$, $\dots$, $v_k$, $v_{k+1}^{S}$, $\dots$, $v_{l}^{S}$. 
    For any such simplex, we introduce a decomposition $T_{0}, \dots, T_{k}$, where each $T_{i}$ has vertices 
    $v_0$, $\dots$, $v_{i-1}$, $z_{S}$, $v_{i+1}$, $\dots$, $v_k$, $v_{k+1}^{S}$, $\dots$, $v_{l}^{S}$.
The collection $\calS^\ast$ of these simplices and their subsimplices constitute a simplicial complex 
    that triangulates the same underlying set as $\calS$.
    Moreover, $\calS^\ast = \patch_{\calS^\ast}(z_{S})$. 
    In particular, $z_{S}$ is a boundary vertex of $\calS^\ast$ if and only if $S$ is a boundary simplex of $\calS$. 
    So it remains to study the topology of vertex stars. 
    
    \item 
    Suppose that $2 \leq n \leq 3$ and that $\calM$ is a manifold without boundary. 
    Under these assumptions, 
    as explained in the proof of Theorem~1 in~\cite{Siebenmann1979},
    the set $\carapace_{\calT}(V)$ is a triangulation of a sphere of dimension $n-1$ for any inner vertex $V$. 
    There exists a homeomorphism from the closed cone of $\underlying{\carapace_{\calT}(V)}$ onto the local star $\underlying{\patch_{\calT}(V)}$.
    But then that closed cone and hence $\underlying{\patch_{\calT}(V)}$ are homeomorphic to an $n$-dimensional ball. 
    
    \item 
    If $2 \leq n \leq 3$ and $\calM$ has a non-empty boundary, then we use an approach as in the proof of Lemma~\ref{lemma:boundarysimplices}: 
    we let $\calM'$ denote the doubling of the manifold and $\calT'$ be the doubling of the triangulation $\calT$. 
    Let $V \in \calT$ be a vertex. 
    If $V$ is an inner vertex of $\calT$, then $\carapace_{\calT}(V) \subseteq \calT \subseteq \calT'$ triangulates a sphere of dimension $n-1$ and $\patch_{\calT}(V) \subseteq \calT \subseteq \calT'$ triangulates a ball of dimension $n$, as discussed above. 
    If $V$ is a boundary vertex of $\calT$, then it is an inner vertex of $\calT'$,
    and so $\carapace_{\calT}(V) \subseteq \calT'$ triangulates a sphere of dimension $n-1$ and $\patch_{\calT}(V) \subseteq \calT'$ triangulates a ball of dimension $n$.
    We also know that $\carapace_{\partial\calT}(V) \subseteq \partial\calT$ triangulates a sphere of dimension $n-2$ and $\patch_{\partial\calT}(V) \subseteq \partial\calT$ triangulates a ball of dimension $n-1$. 
    The embedding of $\carapace_{\partial\calT}(V) \subseteq \partial\calT$ is homeomorphic to the standard embedding of the $(n-2)$-dimensional unit sphere into the $(n-1)$-dimensional unit sphere,
    by the topological Schoenflies theorem~\cite{Moise1977geometric}
    It follows that $\carapace_{\calT}(V)$ triangulates a topological ball of dimension $n-1$.
    Since the closed cone of $\underlying{\carapace_{\calT}(V)}$ is homeomorphic to the star $\underlying{\patch_{\calT}(V)}$, we conclude that $\patch_{\calT}(V)$ triangulates an $n$-dimensional ball.

\end{enumerate}
    All relevant results are proven. 
\end{proof}

\begin{lemma}\label{lemma:connectivity}
    Let $\calT$ be a finite $n$-dimensional simplicial complex whose underlying space is a manifold $\calM$.
    If the underlying space of $\calT$ is connected, then $\calT$ is face-connected.
\end{lemma}
\begin{proof}
    We first show that each vertex star is face-connected via a short induction argument.
    Clearly, any simplicial $1$-ball and simplicial $1$-sphere are face-connected. 
    Now, if $n \geq 1$ and $V \in \calT$,
    then the $n$-simplices in $\patch_{\calT}(V)$ are in correspondence to the $(n-1)$-simplices in $\carapace_{\calT}(V)$.
    Hence, $\patch_{\calT}(V)$, a triangulation of dimension $n$, is face-connected if and only if $\carapace_{\calT}(V)$, a triangulation of dimension $n-1$, is face-connected.
The induction argument implies that each vertex star in $\calT$ is face-connected.

    We assume that the underlying space $\underlying{\calT}$ is connected, and hence path-connected. 
    Given $n$-simplices $S,S' \in \calT$, 
    there exists a path $\gamma : [0,1] \rightarrow \underlying{\calT}$ from the barycenter of $S$ to the barycenter of $S'$.
    We can choose a sequence of $n$-simplices $S = \hat S_0,\hat S_1,\dots,\hat S_m=S' \in \calT$ without repetitions 
    that cover the path and whose intersections with the path are homeomorphic to $[0,1]$.
    For any index $1 \leq i \leq m$, 
    the intersection $\gamma([0,1]) \cap S_{i}$ and $\gamma([0,1]) \cap S_{i+1}$ intersect at one point lying in a common subsimplex of $S_{i}$ and $S_{i+1}$.
    We deduce that each two consecutive simplices in the sequence $\hat S_0,\hat S_1,\dots,\hat S_m$ will have at least one vertex in common. 
    As each vertex star is face-connected, 
    there exists a sequence $S=S_0,S_1,\dots,S_{m}=S' \in \calT$ 
    where $S_{i} \cap S_{i-1}$ is a face of both $S_{i}$ and $S_{i-1}$ for all $1 \leq i \leq m$.
    This just means that $\calT$ is face-connected. 
\end{proof}

\begin{remark}
    Triangulations with the property that all vertex stars are homeomorphic to a ball are also called \emph{combinatorial}~\cite[Section~1]{Bagchi2005}.
    All manifolds of dimension up to three admit smooth structures and smooth manifolds admit combinatorial triangulations. 
    There are triangulations of manifolds in more than three dimensions where not every vertex star is homeomorphic to a ball. 
    
    Not every simplicial complex is the triangulation of some (embedded) topological manifold with or without boundary. 
    When the dimension is at least five, then there are manifolds for which no computer algorithm, given a finite simplicial complex as input, can decide whether the input is the triangulation of that manifold~\cite{chernavsky2006unrecognizability}.
    Going further, it has been shown that deciding whether a simplicial complex triangulates \emph{any} manifold cannot be decided by any computer algorithm~\cite{poonen2014undecidable}.     
We therefore are not in pursuit of any easy combinatorial property that indicates whether a simplicial complex (without any further specific assumptions) triangulates a manifold.

    Conversely, not all topological manifolds, even if compact, can be described as a triangulation. 
    Such manifolds, some even compact and simply-connected, appear in dimension four and higher~\cite{akbulut2014casson}.
\end{remark}

\subsection{Shellable simplicial complexes}\label{subsection:shellability}

The notions of shelling and shellable triangulation have been discussed widely in combinatorial topology and polytopal theory. 
Formally, a triangulation is shellable if its full-dimensional simplices can be enumerated such that each simplex intersects the union of the previously listed simplices in a codimension one triangulation of a manifold. 
This forces the intermediate triangulations to be particularly well-shaped. 
We build upon the notion of shelling as introduced in~\cite[Definition 8.1]{ziegler1995lectures},
where our definition of shelling is equivalent to the notion of the shellings of simplicial complexes, see also~\cite[Remark~8.3]{ziegler1995lectures}.

Suppose that $\calT$ is an $n$-dimensional simplicial complex and that we have an enumeration of the $n$-simplices $T_{0}, T_{1}, T_{2}, \dots \in \subsimplex_{n}(\calT)$.
For any enumeration, we call 
\begin{align*}
    \Gamma_m 
    := 
    \big( 
        T_{0} \cup T_{1} \cup \dots \cup T_{m-1} 
    \big) 
    \cap 
    T_{m}, 
    \qquad 
    1 \leq m,
\end{align*}
the $m$-th \emph{interface set}. 
We call the enumeration a \emph{shelling} if each interface set $\Gamma_m$ is a triangulated manifold of dimension $n-1$. 

\begin{remark}
    In our setting, this is equivalent to the interface set $\Gamma_{m}$ being a non-empty collection of faces of the simplex $T_{m}$.
    In particular, for $n=2$, this interface $\Gamma_{m}$ is a collection of edges and cannot include single vertices of $T_{m}$,
    and for $n=3$, the interface $\Gamma_{m}$ is a collection of faces that cannot include any single vertices or single edges of $T_{m}$.
\end{remark}
\begin{remark}
    We interpret a shelling as the construction of a triangulation 
    by successively attaching simplices such that the intermediate triangulations are well-behaved. 
    Conversely, the reverse enumeration describes a successive decomposition of the triangulation, hence the name ``shelling''.
\end{remark}

We now study some major consequences of the definition of shellable simplicial complexes.
Shellable simplicial complexes are constructed via successive adhesion of simplices.
The resulting succession of intermediate simplicial complexes consists of simplicial balls or spheres.
In particular, this applies to the full triangulation.
Consequently, a domain with a shellable triangulation must be contractible.

\begin{lemma}\label{lemma:shell_triang_man}
    Let $\calT$ be an $n$-dimensional simplicial complex
    with a shelling $T_{0}, T_{1}, T_{2}, \dots, T_{M}$,
    such that each simplex of dimension $n-1$ is contained in at most two simplices.
Then
    $$
        X_{m} := T_{0} \cup T_{1} \cup \dots \cup T_{m}
    $$
    is a triangulated manifold, possibly with boundary, for all $0 \leq m \leq M$.
In particular, $X_{m}$ is a topological $n$-ball when $m < M$,
    and
    $X_{M}$ is either a topological $n$-ball or topological $n$-sphere.
\end{lemma}
\begin{proof}
    We prove this claim by induction.
    Certainly, if $\calT$ contains only one single $n$-simplex, then it is a shellable triangulation of a topological $n$-ball.
    Next, let $1 \leq m \leq M$ and suppose that
    \begin{align*}
        X_{m-1} := T_{0} \cup T_{1} \cup \dots \cup T_{m-1}
    \end{align*}
    is already known to be a topological $n$-ball. Let $T_{m}$ be the next $n$-simplex in the shelling.
    By definition, $\Gamma_{m} := X_{m-1} \cap T_{m}$ is a submanifold of $\partial T_{m}$,
    triangulated by some faces of $T_{m}$ and their subsimplices.

    Let $F$ be such a face.
    By assumption, $F$ must be contained in exactly one $n$-simplex of the partial sequence $T_{0}, T_{1}, \dots, T_{m-1}$,
    and $F$ is in the boundary of $X_{m-1}$.
    We conclude that $\Gamma_{m}$ triangulates a submanifold of the boundary of $X_{m-1}$.

    On the one hand, if $\Gamma_{m}$ is the entire boundary of $T_{m}$, then it is a topological sphere of dimension $n-1$.
    Since $\Gamma_{m}$ is a submanifold of the boundary of $X_{m-1}$, which by induction assumption is also a topological sphere of dimension $n-1$,
    we conclude that $\Gamma_{m}$ is the whole boundary of $X_{m}$.
    By basic geometric topology, $X_{m}$ is a topological $n$-sphere, and $T_{m}$ can only be the last simplex in the shelling, $m=M$.
On the other hand,
    if $\Gamma_{m}$ is a proper subset of the entire boundary of $T_{m}$,
    then $X_{m-1} \cup T_{m}$ is still a topological $n$-ball.
\end{proof}

Each time an $n$-simplex is added, a local star in the full triangulation around some lower-dimensional simplex is completed.

\begin{lemma}\label{lemma:existenceofstar}
    Suppose that an $n$-dimensional manifold triangulation $\calT$ has a shelling $T_{0}, T_{1}, T_{2}, \dots, T_{M}$.
    For $0 \leq m \leq M$, write
    $$
        X_{m} := T_{0} \cup T_{1} \cup \dots \cup T_{m}.
    $$
    For $1 \leq m \leq M$, write
    $$
        \Gamma_{m} := X_{m-1} \cap T_{m}.
    $$
    Then $\Gamma_{m}$ is a union of $\ell$ different faces of $T_{m}$, $1 \leq \ell \leq n+1$.
    If $m < M$, then the intersection of those faces is an interior simplex $S_{m} \in \calT$ of dimension $n-\ell$ that satisfies
    $$
        \patch_{X_{m}}(S_{m}) = \patch_{\calT}(S_{m}).
    $$
\end{lemma}
\begin{proof}
    By definition, $\Gamma_{m}$ is a triangulated submanifold of the boundary of $T_{m}$,
    and so it must be a collection of $\ell$ faces of $T$, $1 \leq \ell \leq n+1$.
    Note that $\ell = n + 1$ can only happen for the last enumerated simplex, $m = M$, if $\calT$ triangulates an $n$-sphere.
    $\Gamma_{m}$ also constitutes a local patch (star) of $(n-1)$-dimensional simplices around some simplex $S_{m}$ of dimension $n-\ell$ in $\Gamma_{m}$.
    By definition, $S_{m}$ is a boundary simplex of $X_{m}$,
    and it is an interior simplex of $X_{m+1}$.
    But then $S_{m}$ cannot be a subsimplex of any of the simplices $T_{m+1}, \dots, T_{M}$,
    which means that $\patch_{X_{m+1}}(S_{m}) = \patch_{\calT}(S_{m})$.
\end{proof}

We collect important examples and counterexamples of shellable triangulations.
Essentially, in two space dimensions, interesting triangulations are shellable, but non-shellable triangulations occur in dimension three and higher.
In practice, triangulations of domains in 3D are shellable except for special cases that we consider as rare.
Our main interest are local patches (stars) within triangulations: these are shellable up to three space dimensions.

\begin{example}\label{example:shell_simpl}
    Any simplex $T$ (trivially) has a shelling, consisting only of $T$ itself. 
    The boundary complex $\partial \calT(T)$ has a shelling:
    any enumeration of the boundary faces of $T$ constitutes a shelling; see Example 8.2.(iii) in~\cite{ziegler1995lectures}.
\end{example}
\begin{example}
    The standard triangulation of the $3$-dimensional cube by six tetrahedra, the Kuhn triangulation~\cite{kuhn1960some}, is shellable, as are its higher-dimensional generalizations.\footnote{We remark that Kuhn attributes this triangulation to Lefschetz~\cite{lefschetz2015introduction}.}
\end{example}
\begin{example}
    There exists a non-shellable triangulation of a tetrahedron and of a cube in $n=3$, see~\cite[Example 8.9]{ziegler1995lectures}.
\end{example}

\begin{lemma}\label{lemma:shell_2D}
    Any simplicial $2$-\disk\ is shellable.
    Any simplicial $2$-sphere is shellable. 
\end{lemma}
\begin{proof}
    First, 
    let $\calS$ be any triangulation of a $2$-sphere. 
    By removing any triangle $S \in \calS$, we obtain a triangulation $\calT$ of a $2$-\disk.
    Any shelling of that triangulation $\calT$ can be extended to a shelling of $\calS$ by re-inserting the first triangle $S$.
    So it remains to show that any triangulation $\calT$ of a two-dimensional \disk\ is shellable. 
    We will construct the shelling in reverse. 
    
    Write $M = \underlying{\calT}$. 
    There is nothing to show if $\calT$ contains only one triangle. 
    We call a triangle $T \in \calT$ \emph{good in $\calT$} if it intersects the boundary $\partial M$ in a non-empty union of edges. 
    Hence, a triangle is good in $\calT$ if its intersection with $\partial M$ is either one, two, or three edges,
    and a triangle is not good in $\calT$ if that intersection is either empty, only some of its vertices, or a vertex and the opposite edge.
We show by an induction argument over the number of triangles that every triangulation of a $2$-\disk\ that contains at least two triangles also contains at least two good triangles. 

    Clearly, this is the case if the triangulation of the $2$-\disk\ contains two triangles. 
    Now suppose the induction claim is true when the triangulation includes at most $N$ triangles,
    and assume that $\calT$ includes $N+1$ triangles. 
    As we travel along the boundary, we traverse along edges of at least two simplices, 
    and therefore there are at least two triangles with an edge on the boundary. 
    Suppose that $\calT$ does not have at least two triangles that are good in $\calT$.
    Then there exists a triangle $T'$ that intersects $\partial M$ 
    in one edge and its opposite vertex.
    Removing $T'$ splits the manifold into two face-connected components, each of which is a topological $2$-\disk. 
    By the induction assumption, each of those components contains at least two triangles 
    that are good in the respective component. 
    So each component has at least one triangle that is also good in $\calT$. 
    Hence, $\calT$ contains two good triangles, which completes the induction step. 
    
    We conclude that whenever $\calT$ triangulates a $2$-\disk,
    it contains a good triangle $T$. 
    If $T$ has three edges in the boundary, then $T = M$ and we are trivially done. 
    If $T$ intersects with the boundary in exactly one or two edges, 
    then $\overline{M \setminus T}$ is still a topological $2$-\disk.
    The triangulation $\calT'$ that is obtained by removing $T$
    is a triangulation of some $2$-\disk\ that intersects $T$ only at either two or one edges.
    Any shelling of $\calT'$ can in this way be extended to a shelling of $\calT$, and the proof is complete. 
\end{proof}

\begin{lemma}
    Let $\calT$ be a $3$-dimensional manifold triangulation and $S \in \calT$.
    Then $\patch_{\calT}(S)$ is shellable. 
\end{lemma}
\begin{proof}
    The statement is trivially true if $S$ is a tetrahedron. 
    The statement is clear if $S$ is an inner or boundary face of $\calT$,
    where we only need to enumerate either one or two tetrahedra. 
    The statement is still easily verified if $S$ is an inner or boundary edge of $\calT$:
    one chooses a starting tetrahedron (with a boundary face if $S$ is a boundary edge) and rotates around the edge in a fixed direction to create a suitable enumeration. 
    When $S$ is an inner vertex, then the faces (triangles) of $\patch_{\calT}(S)$ that do not contain $V$ constitute a simplicial $2$-sphere.
    Similarly, when $S$ is a boundary vertex, then the faces (triangles) of $\patch_{\calT}(S)$ that do not contain $V$ constitute a simplicial $2$-\disk.
    Both these $2$-dimensional complexes are shellable by Lemma~\ref{lemma:shell_2D}, and any such shelling immediately yields a shelling of $\patch_{\calT}(S)$ since there is a one-to-one correspondence between the tetrahedra in $\calT$ and the triangles.\footnote{For $S$ an inner vertex,~\cite[Lemma~B.1]{ern2020stable} also yields the claim.}
\end{proof}

\begin{lemma}\label{lemma:stars_are_shellable}
    Let $\calT$ be an $n$-dimensional shellable triangulation and $V \in \Vertices(\calT)$ be a vertex.
    Then $\patch_{\calT}(V)$ is shellable. 
\end{lemma}
\begin{proof}
    This is Lemma~8.7 in~\cite{ziegler1995lectures}.
\end{proof}

\begin{lemma}
    The barycentric refinement of any triangulation of a three-dimensional convex domain is shellable.
    The second barycentric refinement of any triangulation of a convex domain is shellable.
\end{lemma}
\begin{proof}
    See \cite[Theorem A]{Adi_Ben_subd_shellable_17};
    see also \cite[Proposition 1]{Burg_Mani_71}
\end{proof}

\begin{lemma}
    Let $\calM$ be a $3$-dimensional manifold homeomorphic to a $3$-simplex.
    Then any triangulation of $\calM$ has a shellable refinement.
\end{lemma}
\begin{proof}
    Let $\calT$ be a triangulation of $\calM$ and let $\calT_{0}$ be a triangulation of the $3$-simplex.
According to the three-dimensional ``Hauptvermutung''~\cite[Theorem 4]{moise1952affine},
    the triangulations $\calT$ and $\calS$ admit respective refinements $\calT'$ and $\calS'$ such that $\calT$ is isomorphic to $\calS$.
The first barycentric refinement of $\calS'$ is shellable~\cite[Theorem A]{Adi_Ben_subd_shellable_17}.
    We conclude that $\calT$ has a refinement that is shellable.
\end{proof}

\begin{remark}
    Not all triangulable sets admit a triangulation that is also shellable.
    Even if a set admits some shellable triangulation, some of its other triangulations might not be shellable.
    We refer to~\cite[Example 8.9]{ziegler1995lectures} for examples of non-shellable triangulations of cubes and tetrahedra in three dimensions. 
    Moreover, if we extend the non-shellable triangulation of the tetrahedron from~\cite[Example 8.9]{ziegler1995lectures} to a triangulation of a hypertetrahedron by suspending it from a new point $v^\star$, then the resulting new triangulation is non-shellable and coincides with the patch around $v^\star$.
    This demonstrates that patches around boundary simplices are not necessarily shellable when the space dimension $n$ is larger than three.
\end{remark}

\begin{remark}
    Whether a simplicial complex is shellable can be checked, in principle, simply by trying out all the possible enumerations.
    That we cannot do much better than this is captured in the fact that testing general simplicial complexes (with some fixed dimension) for shellability is NP-complete~\cite{goaoc2019shellability}.
    This complexity class even holds
    if we only test two-dimensional simplicial complexes for shellability,
    or if we only test shellability of contractible three-dimensional simplicial complexes.
\end{remark}

\section{Reflections and deformations on shellable stars}\label{section:extension}

This section is devoted to geometric operations that are crucial for our main result in Section~\ref{section:poincarefriedrichs} below. 
Consider the situation where we have an $n$-dimensional local patch (star) around some simplex $S$ and some $n$-dimensional simplex $T$ within that local star. 
We construct a homeomorphism going from the simplex $T$ onto its complement within the local star around $S$.
This homeomorphism, which we interpret as a nonlinear reflection, preserves the interface. 
We ensure that the homeomorphism is bi-Lipschitz, and we are particularly interested in the norms of its Jacobian.
This nonlinear reflection, which is similar to the two- and three-color maps in~\cite[Sections~5.3 and~6.3]{ern2020stable} and the symmetrization maps in~\cite[Section~7.6]{Chaum_Voh_p_rob_3D_H_curl_24}, will be used subsequently in generalizing the discussion in Section~\ref{section:gradient} to the setting of differential forms.

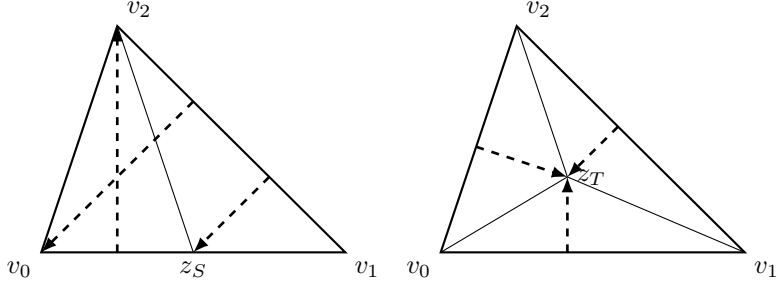
\begin{figure}[t]
    \caption{Illustration of Lemma~\ref{lemma:stardivision}. Left:
    the triangle $T = [ v_0, v_1, v_2 ]$ is bisected at the edge $S = [ v_0, v_1 ]$, leading to two new triangles.
    The height vector to $v_2$ in all three triangles remains the same. The height vector to $z_S$ in the new triangle $[z_S, v_1, v_2]$ is one half of the height vector to $v_0$ in the original triangle.
    Right:
    the triangle $T$ is trisected, leading to three new triangles.
    The height vector to $z_T$ from the opposite edge in any triangle is one third of the original height vector of that edge.}
    \begin{center}
    \begin{tikzpicture}
\coordinate (A) at (0,0); \coordinate (B) at (4,0); \coordinate (C) at (1,3); \coordinate (M) at ($(A)!0.5!(B)$); \coordinate (H_A) at ($(C)!(A)!(B)$); \coordinate (H_B) at ($(C)!(B)!(A)$); \coordinate (H_C) at ($(A)!(C)!(B)$); \coordinate (H_M) at ($(B)!(M)!(C)$); \draw[thick] (A) -- (B) -- (C) -- cycle;
\draw[] (C) -- (M);
        \draw[dashed,->,line width=1.0pt,>=latex] (H_A) -- (A); \draw[dashed,->,line width=1.0pt,>=latex] (H_M) -- (M); \draw[dashed,->,line width=1.0pt,>=latex] (H_C) -- (C); \node[above right] at (C) {$v_2$};
        \node[below left] at (A) {$v_0$};
        \node[below right] at (B) {$v_1$};
        \node[below] at (M) {$z_S$};
    \end{tikzpicture}
    \begin{tikzpicture}

\coordinate (A) at (0,0); \coordinate (B) at (4,0); \coordinate (C) at (1,3); \coordinate (G) at (1.666666666,1); \draw[thick] (A) -- (B) -- (C) -- cycle;
\draw[] (A) -- (G);
        \draw[] (B) -- (G);
        \draw[] (C) -- (G);
\coordinate (H_GA) at ($(B)!(G)!(C)$); \coordinate (H_GB) at ($(A)!(G)!(C)$); \coordinate (H_GC) at ($(A)!(G)!(B)$); 

\draw[dashed,->,line width=1.0pt,>=latex] (H_GA) -- (G); \draw[dashed,->,line width=1.0pt,>=latex] (H_GB) -- (G); \draw[dashed,->,line width=1.0pt,>=latex] (H_GC) -- (G); 

\node[above right] at (C) {$v_2$};
        \node[below left] at (A) {$v_0$};
        \node[below right] at (B) {$v_1$};
        \node[right] at (G) {$z_T$};
    \end{tikzpicture}
    \end{center}
\end{figure}

The following observation, which we state without proof,
controls the volume and some heights when a simplex is partitioned via barycentric subdivision.

\begin{lemma}\label{lemma:stardivision}
    Let $T$ be an $n$-dimensional simplex with an $\ell$-dimensional subsimplex $S$ ($\ell \geq 0$) and let $z_S$ be the barycenter of $S$.
    Let $T'$ be one of the $n$-dimensional simplices obtained by splitting $T$ in accordance with the barycentric subdivision of $S$ at $z_{S}$.
    \begin{itemize}
        \item $\vol(T') = \vol(T) / (\ell+1)$. 
        \item The height vector of $v \in \Vertices(T) \setminus \Vertices(S)$ in $T'$ is the height vector of $v$ in $T$.
        \item The height vector of $z_S$ in $T'$ is the height vector of the single vertex $v \in \Vertices(T) \setminus \Vertices(T')$ in $T$,
scaled by $(\ell+1)^{-1}$.
    \hfill\qed
    \end{itemize}
\end{lemma}

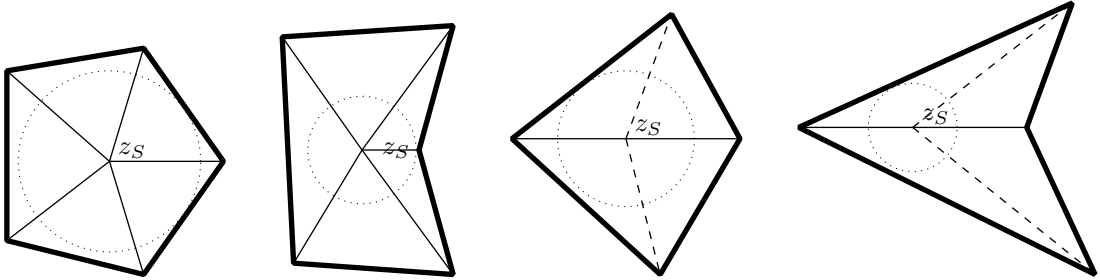
\begin{figure}[t]
    \caption{Illustration of Proposition~\ref{proposition:stars_are_starshaped} and its proof:
    for a convex vertex star, for a non-convex vertex star, both cases where the barycentric refinement of $S$ has no effect,
    for a convex edge star, and for a non-convex edge star.
    }\label{figure:stars_are_starshaped}
    \centering
    \begin{tabular}{c}
        \begin{tikzpicture}
[line join=bevel,x={( 1.5cm, 0mm)},y={( 0mm, 1.5cm)},z={( 1.5*3.85mm, -1.5*3.85mm)}]
            \coordinate (O) at (  0.0,  0.0, 0.0);
            \coordinate (A) at (  1.0,  0.0, 0.0);
            \coordinate (B) at (  0.3,  1.0, 0.0);
            \coordinate (C) at ( -0.9,  0.8, 0.0);
            \coordinate (D) at ( -0.9, -0.7, 0.0);
            \coordinate (E) at (  0.3, -1.0, 0.0);
            \draw[dotted] (0,0) circle [radius=0.8];
            \node at (0.2, 0.1) {$z_S$};
            \draw[line width=0.05em] (O) -- (A);
            \draw[line width=0.05em] (O) -- (B);
            \draw[line width=0.05em] (O) -- (C);
            \draw[line width=0.05em] (O) -- (D);
            \draw[line width=0.05em] (O) -- (E);
            \draw[line width=0.2em] (A) -- (B) -- (C) -- (D) -- (E) -- (A) -- cycle;
        \end{tikzpicture}\qquad
        \begin{tikzpicture}
[line join=bevel,x={( 1.5cm, 0mm)},y={( 0mm, 1.5cm)},z={( 1.5*3.85mm, -1.5*3.85mm)}]
            \coordinate (O) at (  0.0,  0.0, 0.0);
            \coordinate (A) at ( -0.7,  1.0, 0.0);
            \coordinate (B) at (  0.8,  1.1, 0.0);
            \coordinate (C) at (  0.5,  0.0, 0.0);
            \coordinate (D) at (  0.8, -1.1, 0.0);
            \coordinate (E) at ( -0.6, -1.0, 0.0);
            \draw[dotted] (0,0) circle [radius=0.475];
            \node at (0.3, 0.0) {$z_S$};
            \draw[line width=0.05em] (O) -- (A);
            \draw[line width=0.05em] (O) -- (B);
            \draw[line width=0.05em] (O) -- (C);
            \draw[line width=0.05em] (O) -- (D);
            \draw[line width=0.05em] (O) -- (E);
            \draw[line width=0.2em] (A) -- (B) -- (C) -- (D) -- (E) -- (A) -- cycle;
        \end{tikzpicture}\qquad
        \begin{tikzpicture}
[line join=bevel,x={( 1.5cm, 0mm)},y={( 0mm, 1.5cm)},z={( 1.5*3.85mm, -1.5*3.85mm)}]
            \coordinate (O) at (  0.0,  0.0, 0.0);
            \coordinate (A) at (  1.0,  0.0, 0.0);
            \coordinate (B) at ( -1.0,  0.0, 0.0);
            \coordinate (C) at (  0.4,  1.1, 0.0);
            \coordinate (D) at (  0.3, -1.2, 0.0);
            \node at (0.2, 0.1) {$z_S$};
            \draw[dotted] (0,0) circle [radius=0.60];
            \draw[line width=0.05em] (O) -- (A);
            \draw[line width=0.05em] (O) -- (B);
            \draw[dashed,line width=0.05em] (O) -- (C);
            \draw[dashed,line width=0.05em] (O) -- (D);
            \draw[line width=0.2em] (A) -- (C) -- (B) -- (D) -- (A) -- cycle;
        \end{tikzpicture}\qquad
        \begin{tikzpicture}
[line join=bevel,x={( 1.5cm, 0mm)},y={( 0mm, 1.5cm)},z={( 1.5*3.85mm, -1.5*3.85mm)}]
            \coordinate (O) at (  0.0,  0.0, 0.0);
            \coordinate (A) at (  1.0,  0.0, 0.0);
            \coordinate (B) at ( -1.0,  0.0, 0.0);
            \coordinate (C) at (  1.4,  1.1, 0.0);
            \coordinate (D) at (  1.6, -1.3, 0.0);
            \node at (0.2, 0.1) {$z_S$};
            \draw[dotted] (0,0) circle [radius=0.39];
            \draw[line width=0.05em] (O) -- (A);
            \draw[line width=0.05em] (O) -- (B);
            \draw[dashed,line width=0.05em] (O) -- (C);
            \draw[dashed,line width=0.05em] (O) -- (D);
            \draw[line width=0.2em] (A) -- (C) -- (B) -- (D) -- (A) -- cycle;
        \end{tikzpicture}\end{tabular}
\end{figure}

The following auxiliary result on the geometry of local stars is of independent interest.

\begin{proposition}\label{proposition:stars_are_starshaped}
    Let $\calT$ be a triangulation of an $n$-dimensional domain.
    Let $S \in \calT$ be an inner simplex of dimension $\ell < n$.
    Then $\underlying{\patch_{\calT}(S)}$ is star-shaped with respect to the ball $\Ball_{h/(\ell+1)}(z_S)$,
    where $z_S$ denotes the barycenter of $S$,
    and where $h > 0$ denotes the maximum height of any vertex of $S$ within any $n$-simplex of $\patch_{\calT}(S)$.
\end{proposition}
\begin{proof}
    Let $\calR$ denote the simplicial complex obtained from $\patch_{\calT}(S)$ by barycentric refinement of $S$.
    Then $\calR$ is its own vertex star around $z_{S}$.
    Taking into account Lemma~\ref{lemma:stardivision},
    the desired result follows for $\ell > 0$ if we show it for $\ell = 0$.

    So we consider the case $\ell = 0$.
    Any codimension one boundary face $F \in \patch_{\calT}(S)$ corresponds to some $n$-simplex $T \in \patch_{\calT}(S)$,
    where $F$ is the face of $T$ opposite to its vertex $z_{S}$.
    For any such codimension one boundary face $F \in \patch_{\calT}(S)$,
    we let $H_{F}$ denote its affine hull, which is a half-plane.

    Assume now that $\varrho > 0$ is at most the minimum height of $z_{S}$ in any $n$-simplex $T \in \patch_{\calT}(S)$.
    If $x \in \Ball_{\varrho}(z_S)$ and $F \in \patch_{\calT}(S)$ is a boundary face,
    then our choice of $\varrho$ implies that $x$ and $z_S$ must lie on the same side of $H_{F}$.
    In particular, $\Ball_{\varrho}(z_S)$ cannot intersect any boundary face of $\patch_{\calT}(S)$,
    which shows $\Ball_{\varrho}(z_S) \subseteq \underlying{\patch_{\calT}(S)}$.

    To show that $\underlying{\patch_{\calT}(S)}$ is star-shaped with respect to $\Ball_{\varrho}(z_S)$,
    it suffices to show that for every two-dimensional hyperplane $P$ through $z_{S}$
    the intersection $\underlying{\patch_{\calT}(S)} \cap P$ is star-shaped with respect to $\Ball_{\varrho}(z_S) \cap P$.
    The latter is deduced observing that $\underlying{\patch_{\calT}(S)} \cap P$ is a two-dimensional polygon within the plane $P$
    and using a well-known result of planar geometry~\cite{lee1979optimal}.
\end{proof}

We now provide the desired bi-Lipschitz transformation, which maps a local star onto itself.
We interpret it as a nonlinear reflection across the interface between some fixed simplex of the local star and the remainder of that star.
We give detailed estimates for the singular values of the Jacobian.

\begin{figure}[t]
    \caption{Sketch of the geometric situation in the proof of Proposition~\ref{proposition:starreflection}.
    The simplex $T$ completes the star around an interior vertex, and $\UPatch$ is the complement of the simplex $T$ within that star. 
    Here, $\Gamma_1$ denotes the interface between $T$ and $\UPatch$.
    The reflection $\Xieins$ maps $T$ into $\UPatch$ while being the identity the interface $\Gamma_1$.
    The mapping $\Xieins$ is the identity outside of $\calK$ and $\calK^{c}$, only affecting the latter two sets.
    }\label{figure:geometricconstruction}
    \centering
    \begin{tabular}{cc}
        \begin{tikzpicture}
[line join=bevel,x={( 1.5cm, 0mm)},y={( 0mm, 1.5cm)},z={( 1.5*3.85mm, -1.5*3.85mm)}]

            \coordinate (O) at (  0.0,  0.0, 0.0);
            \coordinate (A) at (  1.0,  0.0, 0.0);
            \coordinate (B) at (  0.3,  1.0, 0.0);
            \coordinate (C) at ( -0.9,  0.8, 0.0);
            \coordinate (D) at ( -0.9, -0.7, 0.0);
            \coordinate (E) at (  0.3, -1.0, 0.0);
            \coordinate (F) at (  0.3, -1.0, 0.0);

            \fill[gray!30] (A) -- (B) -- (O) -- cycle;

            \node at (-0.15, -0.5) {$\UPatch$};
            \node at (0.4, 0.35) {$T$};
            \node at (0.45, -0.2) {$\Gamma_1$};
            \node at (-1.10, 0) {$\Gamma_2$};

            \draw[dashed,line width=0.2em] (O) -- (A); \draw[dashed,line width=0.2em] (O) -- (B); \draw[] (A) -- (B); \draw[line width=0.2em] (B) -- (C) -- (D) -- (E) -- (A);
        \end{tikzpicture}\qquad
        \begin{tikzpicture}
[line join=bevel,x={( 1.5cm, 0mm)},y={( 0mm, 1.5cm)},z={( 1.5*3.85mm, -1.5*3.85mm)}]

            \coordinate (O) at (  0.0,  0.0, 0.0);
            \coordinate (A) at (  1.0,  0.0, 0.0);
            \coordinate (B) at (  0.3,  1.0, 0.0);
            \coordinate (C) at ( -0.9,  0.8, 0.0);
            \coordinate (D) at ( -0.9, -0.7, 0.0);
            \coordinate (E) at (  0.3, -1.0, 0.0);
            \coordinate (F) at (  0.3, -1.0, 0.0);
            \coordinate (AB) at ( 0.65,  0.5, 0.0);
            \coordinate (nAB) at (-0.65, -0.5, 0.0);

            \fill[gray!30] (A) -- (B) -- (O) -- cycle;
            \fill[gray!10] (A) -- (O) -- (nAB) -- cycle;
            \fill[gray!10] (B) -- (O) -- (nAB) -- cycle;

            \node at ( 0.28,  0.5) {$\calK$};
            \node at (-0.15, -0.5) {$\calK^{c}$};

            \draw[dotted] (O) -- (A); \draw[dotted] (O) -- (B); \draw[dotted] (O) -- (AB); \draw[dotted] (O) -- (nAB); \draw[dotted] (nAB) -- (A); \draw[dotted] (nAB) -- (B); \draw[line width=0.2em] (B) -- (C) -- (D) -- (E) -- (A) -- cycle;
        \end{tikzpicture}\end{tabular} 
\end{figure}
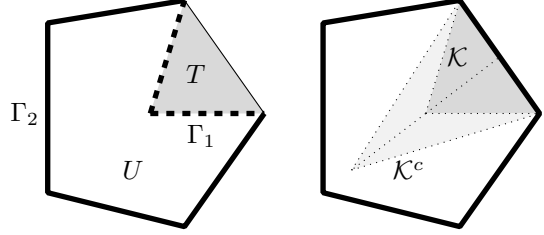

\begin{proposition}\label{proposition:starreflection}
    Let $\calT$ be a triangulation of an $n$-dimensional domain. 
    Let $S \in \calT$ be an inner simplex of dimension $\ell < n$,
    let $T \in \patch_{\calT}(S)$ be of dimension $n$,
    and let 
    \begin{align*}
        \UPatch := \overline{ \underlying{\patch_{\calT}(S)} \setminus T },
        \qquad 
        {\Gamma_1} := \UPatch \cap T.
    \end{align*}
    The following holds, where the constants on the right-hand sides are as stated in the proof. 

    There exists a bi-Lipschitz piecewise affine mapping
    \begin{align*}
        \Xieins : T \rightarrow \Xieins(T) \subseteq \UPatch 
    \end{align*}
    which is the identity along ${\Gamma_1} = T \cap \UPatch$. 
    At any point, the singular values $\sigma_1 \geq \dots \geq \sigma_n$ of its Jacobian satisfy 
    \begin{gather*}
        \sigma_1 \leq \Cfive{n}{\ell}(\calT),
        \quad 
        \sigma_2,\dots,\sigma_n \leq \Cfiveprime{n}{\ell}(\calT),
        \quad 
        | \det\Jacobian\Xieins |      \leq \Cdetfive{n}{\ell}(\calT),
        \\
        \sigma_n^{-1} \leq \Csechs{n}{\ell}(\calT),
        \quad 
        \sigma_1^{-1},\dots,\sigma_{n-1}^{-1} \leq \Csechsprime{n}{\ell}(\calT),
        \quad 
        | \det\Jacobian\Xieins^{-1} | \leq \Cdetsechs{n}{\ell}(\calT).
    \end{gather*}
    Moreover, for any $0 \leq k \leq n$ and $p \in [1,\infty]$,
    \begin{align*}
        \sigma_{1}\cdots\sigma_{k} | \det\Jacobian\Xieins |^{-\frac 1 p}
        \leq 
        \Cfivetrafo{n}{k}{\ell}{p}(\calT),
        \qquad 
        \sigma_{n}^{-1}\cdots\sigma_{n-k+1}^{-1} | \det\Jacobian\Xieins |^{\frac 1 p}
        \leq 
        \Csechstrafo{n}{k}{\ell}{p}(\calT).
    \end{align*}
\end{proposition}
\begin{proof}
    We derive the estimate in several steps. 
    In what follows, we use the notation $\hat z$ for the normalization of any vector $z \in \bbR^{n}$.
    \begin{itemize}
        \item 
        Without loss of generality, after shifting the coordinate system, the barycenter $z_{S}$ of $S$ is the origin, $z_{S} := 0$. 
        
        We fix the subsimplex $S' \subseteq T$ that is complementary to $S$ within the simplex $T$.
        Note that 
        ${\Gamma_1} = \UPatch \cap T$ is the union of exactly those faces of $T$ that contain $S$, 
        We let $z_{S'}$ be the midpoint of $S'$.
        
        We apply barycentric refinement to the simplex $S$ and then the complementary simplex $S'$.
        This produces a new triangulation $\calK$ of $T$ whose $n$-dimensional simplices contain the vertices $z_{S}$ and $z_{S'}$.
        
        \item 
        We use Proposition~\ref{proposition:stars_are_starshaped}:
        $\patch_{\calT}(S)$ is star-shaped with respect to the closed ball $\overline{\Ball_{\varrho}(z_S)}$,
        where $\varrho = h / (\ell+1)$ and $h$ is the maximum height of any vertex of $S$ within any $n$-simplex $T \in \patch_{\calT}(S)$.

        \item 
        Let $\rho \in (0,1]$, yet to be determined. 
        We define $y := - \rho z_{S'}$ as vector in the opposite direction of $z_{S'}$ and with length $\rho \vecnorm{ z_{S'} }$.
        Henceforth, we assume $\rho \leq \varrho / \vecnorm{ z_{S'} }$ so that $y \in \overline{\Ball_{\varrho}(z_S)}$.  
        By construction, the closed line segment from $z_{S}$ to $y$ is contained in $\overline{\Ball_{\varrho}(z_S)}$,
        and we conclude that the convex closure of $T$ and $y$ must lie in $\patch_{\calT}(S)$.
        
        We define another triangulation $\calK^{c}$ as follows:
        given any $n$-simplex $K \in \calK$,
        we replace its vertex $z_{S}'$ by the vertex $y$ at the opposite position,
        thus obtaining a new $n$-simplex $K^{c}$.
        Indeed, that $\calK^c$ is a simplicial complex follows easily from $\calK$ being a simplicial complex.
        
        We let the simplicial complex $\calK^\ast$ be the union of $\calK$ and $\calK^c$.
        By construction, all its $n$-simplices contain $z_{S}$ as a subsimplex, which is an inner vertex of that triangulation. 
        In particular, $\calK^\ast$ is its own star around $z_{S}$.

        \item 
        We introduce a new mapping $\Theta : \underlying{\calK} \rightarrow \underlying{\calK^c}$ between the underlying sets as follows:
        if $K \in \calK$ be an $n$-simplex and $K^{c} \in \calK^{c}$ is constructed from $K$,
        then $\Theta\restriction_{K}$ is the unique affine bijection that maps the vertex $z_{S'}$ to the vertex $y$ and leaves all other vertices the same.
It follows that 
        \begin{align*}
            |\det( \Jacobian      \Theta\restriction_{K} )| = \frac{\vol(K^{c})}{\vol(K)}
            .
        \end{align*}
        We want to further characterize the singular values of this transform's Jacobian.
        Let $h_{z} \in \bbR^{n}$ be the height vector of $z_{S'}$ inside the simplex $K$,
        that is, the vector pointing to $z_{S'}$ and standing orthogonally on the affine hull of the face $K$ that opposes $z_{S'}$.
        Let us verify that
        \begin{align*}
            \Theta\restriction_{K}(x)
            = 
            x
            - 
            (1+\rho) \frac{\langle \hat h_{z}, x \rangle}{\langle \hat h_{z}, \hat z_{S'} \rangle}
            \hat z_{S'}
            .
        \end{align*}
        Indeed, the right-hand side equals $x$ whenever $x$ lies in the plane orthogonal to $h_{z}$,
        and when $x = z_{S'}$, then 
        \begin{align*}
            \Theta\restriction_{K}\left( z_{S'} \right)
            &
            =
            z_{S'}
            - 
            (1+\rho) \hat z_{S'} \frac{\langle \hat h_{z}, z_{S'} \rangle}{\langle \hat h_{z}, \hat z_{S'} \rangle}
            \\&
            =
            z_{S'}
            - 
            (1+\rho) \hat z_{S'} \vecnorm{z_{S'}}
            =
            z_{S'}
            - 
            (1+\rho) z_{S'} 
            =
            - \rho z_{S'} = y
            .
        \end{align*}
        If we orthogonally decompose $z_{S'} = h_{z} + b_{z}$, where $b_{z} \in \bbR^{n}$, then
        \begin{align*}
            \Theta\restriction_{K}( h_{z} )
            &= 
            h_{z}
            - 
            (1+\rho) \frac{\langle h_{z}, h_{z} \rangle}{\langle h_{z}, z_{S'} \rangle} z_{S'}
            \\&
            = 
            h_{z}
            - 
            (1+\rho) \frac{\langle h_{z}, z_{S'} \rangle}{\langle h_{z}, z_{S'} \rangle} z_{S'}
= 
            h_{z}
            - 
            (1+\rho) z_{S'}
= 
            - \rho h_{z}
            - 
            (1+\rho) b_{z}
            .
        \end{align*}
        Evidently, the transformation $\Theta\restriction_{K}$ equals the identity on the orthogonal complement of the span of $h_{z}$ and $b_{z}$.
        Let $\beta$ be the angle between $z_{S'}$ and $h_{z}$. 
        Then $\vecnorm{ h_{z} } = \cos(\beta) \vecnorm{ z_{S'} }$ and $\vecnorm{ b_{z} } = \sin(\beta) \vecnorm{ z_{S'} }$. 
        It remains to study the singular values of the matrix
        \begin{align*}
            M_{\Theta,K}
            := 
            \begin{pmatrix}
            -\rho                                         & 0
            \\ 
            -(1+\rho) \vecnorm{ b_{z} }/\vecnorm{ h_{z} } & 1
            \end{pmatrix}
            =
            \begin{pmatrix}
            -\rho                 & 0
            \\ 
            -(1+\rho) \tan(\beta) & 1
            \end{pmatrix}
            .
        \end{align*}
        Its two singular values are:
        \begin{align}
            \sigma_{\max}(\Theta,K)
            := 
            \frac 1 2 \sqrt{ \left( 1 + \rho \right)^{2} + (1+\rho)^{2} \tan(\beta)^{2} } + \frac 1 2 \sqrt{ \left( 1 - \rho \right)^{2} + (1+\rho)^{2} \tan(\beta)^{2} }
            \label{math:theta_max}
            ,
            \\
            \sigma_{\min}(\Theta,K)
            := 
            \frac 1 2 \sqrt{ \left( 1 + \rho \right)^{2} + (1+\rho)^{2} \tan(\beta)^{2} } - \frac 1 2 \sqrt{ \left( 1 - \rho \right)^{2} + (1+\rho)^{2} \tan(\beta)^{2} }
            \label{math:theta_min}
            .
        \end{align}
        The singular values $\sigma_{\max}(\Theta,K) \geq 1$ and $\sigma_{\min}(\Theta,K) \leq 1$ are monotonically increasing and decreasing, respectively, in $\tan(\beta)$.
        All other singular values of $\Jacobian\Theta\restriction_{K}$ equal $1$, hence \eqref{math:theta_max} and \eqref{math:theta_min} must the extremal maximal singular values of $\Jacobian\Theta\restriction_{K}$.
        It is evident that $\rho = \vol(K^{c}) / \vol(K)$.
        
        We develop explicit bounds for these singular values. 
        The definition of the tangent and the decomposition $z_{S'} = h_{z} + b_{z}$ 
        imply that $\tan(\beta) = \vecnorm{ b_{z} } / \vecnorm{ h_{z} }$.
        
        Recall that $K \in \calK$ is obtained from $T$ via barycentric subdivisions:
        first at $z_{S}$, leading to some intermediate $n$-simplex $T'$, and then at $z_{S'}$, leading to $K \subseteq T'$.
        Let $F_{z} \subseteq K$ be the face opposite to the vertex $z_{S'}$, which is some face of $T'$.
        Now, the height of $F_{z}$ in $K$ is $(n-\ell)$-th of the height of $F_{z}$ in $T'$
        since $S'$ has dimension $n-\ell-1$. 
        The height of the face $F_{z}$ in $T'$ is just the height of its opposing vertex, which lies in $S'$,
        and which equals the height of that vertex in $T$.
Thus, 
        \begin{align*}
            \frac{ \vecnorm{ b_{z} } }{ \vecnorm{ h_{z} } }
            = 
            \frac{ \vecnorm{ b_{z} } }{ (n-\ell)^{-1} \vecnorm{ (n-\ell) h_{z} } }
            \leq
            \frac{ \vecnorm{ z_{S'} } }{ (n-\ell)^{-1} \vecnorm{ (n-\ell) h_{z} } }
            \leq 
            (n-\ell) \aspectratio(T)
.
        \end{align*}
        We abbreviate 
        \begin{align}\label{math:upper_bound_for_theta}
            \mu_{T,\ell} 
            := 
            \frac 1 2 \sqrt{ \left( 1 + \rho \right)^{2} + (1+\rho)^{2} (n-\ell)^{2} \aspectratio(T)^{2} } 
            + 
            \frac 1 2 \sqrt{ \left( 1 - \rho \right)^{2} + (1+\rho)^{2} (n-\ell)^{2} \aspectratio(T)^{2} }
            .
        \end{align}
        This establishes bounds on the singular values of the transformation. 
        In summary, the singular values of the Jacobian of $\Theta$ at almost every $x$ satisfy 
        \begin{align}
            \sigma_{1}(\Theta,x) \leq \mu_{T,\ell},
            \quad 
            \sigma_{2}(\Theta,x) = \dots = \sigma_{n-1}(\Theta) = 1,
            \quad 
            \sigma_{n}(\Theta,x)^{-1} = \frac{ \sigma_{1}(\Theta,x) }{ \left| \det \Jacobian\Theta\restriction_{K} \right| } \leq \mu_{T,\ell} / \rho.
        \end{align}
        Finally, we obtain $\Xieins$ by extending $\Theta$ to all of the patch as the identity.
    \end{itemize}
    Having shown all the desired estimates, and the proof is complete. 
\end{proof}

\begin{remark}
    We notice that $\Xieins$ above is not only bi-Lipschitz, but in fact also piecewise affine with respect to an essentially non-overlapping simplicial decomposition of its domain.
    Whilst the reflection mapping in Proposition~\ref{proposition:starreflection} serves our purpose,
    it might be possible to improve the analysis or construction and lower the Jacobian estimates. 
\end{remark}

\begin{remark}\label{remark:improved_starreflection}
    The estimates in Proposition~\ref{proposition:starreflection} are by no means optimal.
    Some improvements are immediately possible if $\ell=n-1$, corresponding to patches around faces.
    Then the reflection $\Xi$ in Lemma~\ref{lemma:volumecomparison} satisfies the same properties as the mapping $\Xieins$,
    mapping the simplex $T$ bijectively onto the $\UPatch$ and being the identity along the common face $S$.
    However, at any point, the singular values $\sigma_1 \geq \dots \geq \sigma_n$ of the Jacobian of $\Xi$ satisfy 
\begin{gather*} 
        | \det\Jacobian\Xi^{  } |      \leq \volumeratio(\calT),
        \\
        \sigma_{1}, \sigma_{n}^{-1}
        \leq 
        \frac{1}{2} 
        \sqrt{ \left( \diameterratio(\calT) \aspectratio(\calT) + 1 \right)^{2} + \aspectratio(\calT)^{2} }
        +
        \frac{1}{2} 
        \sqrt{ \left( \diameterratio(\calT) \aspectratio(\calT) - 1 \right)^{2} + \aspectratio(\calT)^{2} }
        ,
        \\
        \sigma_2,\dots,\sigma_{n-1},\sigma_{n} \leq 1.
    \end{gather*}
    In particular,
    \begin{align*}
        \sigma_{1}\cdots\sigma_{k} | \det\Jacobian\Xizwei |^{-\frac 1 p}
        \leq {} &
        \left( 
            \frac{1}{2} 
            \sqrt{ \left( \diameterratio(\calT) \aspectratio(\calT) + 1 \right)^{2} + \aspectratio(\calT)^{2} }\right.\\
            {} & +
            \left. \frac{1}{2} 
            \sqrt{ \left( \diameterratio(\calT) \aspectratio(\calT) - 1 \right)^{2} + \aspectratio(\calT)^{2} }
        \right)
        \volumeratio(\calT)^{\frac 1 p}
        .
    \end{align*}
    These estimates improve over the ones in Proposition~\ref{proposition:starreflection} when the reflection is over a single face.
    We believe that better estimates can be computed from the geometric data also in the other cases.
\end{remark}

\section{Constructive estimates of Poincar\'e--Friedrichs constants}\label{section:poincarefriedrichs}

We are now in the position to develop Poincar\'e--Friedrichs constants for the exterior derivative over domains with shellable triangulations; this includes the curl and divergence operators in three dimensions. 

Our estimates use local Poincar\'e--Friedrichs constants over simplices with various boundary conditions.
The analysis of the Poincar\'e potential operators in Section~\ref{section:potentialoperator} has produced Poincar\'e--Friedrichs constants for bounded convex domains and their analogs for the case of full boundary conditions. 
Now, we also prove a Poincar\'e--Friedrichs inequality for differential forms over a simplex but subject to homogeneous boundary conditions along a collection of faces.

\begin{lemma}\label{lemma:mixedbconsimplex:exteriorderivative}
    Let $T$ be an $n$-simplex 
    and let $\Gamma = F_{0} \cup \dots \cup F_{\ell}$ be the union of $\ell+1$ different faces ($\ell < n$) of $T$.
    Suppose that $u \in W^{p}\Lambda^{k}(T)$ such that 
    $\trace_{T,\Gamma} u = 0$.
    Then there exists $w \in W^{p}\Lambda^{k}(T)$ such that 
    $\trace_{T,\Gamma} w = 0$
    and  
    \begin{align*}
        \cartan w = \cartan u,
        \quad 
        \| w \|_{L^{p}(\Omega)} 
        \leq 
        C_{\PF,T,\Gamma,\ell,k,p}
        \| \cartan u \|_{L^{p}(\Omega)}
        .
    \end{align*}
    Here, $C_{\PF,T,\Gamma,\ell,k,p} > 0$ is a constant such that 
    \begin{align*}
        C_{\PF,T,\Gamma,\ell,k,p}
        \leq 
        n! \cdot 2^{\ell+1} 
        C_{\Bogov}(n,k+1) 
        \vol_{n-1}(S_1(0))
        \cdot 
        \algebraicshapemeasure(T)^{k-1}
        \Ceins{n}(T) 
        \diam(T)
        .
    \end{align*}
\end{lemma}
\begin{proof}
    There exists an affine bijection $\varphi : \Delta^n \rightarrow T$ 
    that maps the convex closure of the $n$ unit vectors onto the face $F_0$.
    We can also assume that the face of $\Delta^n$ orthogonal to the $i$-th coordinate axis, $1 \leq i \leq n$
    is mapped onto the face $F_i$. 
    In what follows, we let $\widetilde \UPatch$ be the convex set obtained from reflecting $\Delta^n$ along the coordinate axes $\ell+1$ through $n$. 
    We see that $\vol(\widetilde \UPatch) = 2^{n-\ell} \vol(\Delta^n) = 2^{n-\ell}/(n!)$
    
    We let $\hat u := \varphi^{\ast} u \in W^{p}\Lambda^{k}(\Delta^n)$ and define $\hat g \in L^{p}\Lambda^{k+1}(\Delta^n)$ via $\hat g := \cartan \varphi^{\ast} u$. 
    Then $\cartan \hat u = \hat g$. 
    We let $\tilde u \in W^{p}\Lambda^{k}(\widetilde \UPatch)$ be the extension of $\hat u$ onto $\widetilde \UPatch$ by reflection across the coordinate axes,
    and let $\tilde g \in L^{p}\Lambda^{k+1}(\widetilde \UPatch)$ be the extension of $\hat g$ onto $\widetilde \UPatch$ by reflection across the coordinate axes. 
    By construction, $\tilde u \in W^{p}_{0}\Lambda^{k}(\widetilde \UPatch)$ with $\tilde g = \cartan \tilde u$.
    We observe 
    \begin{align*}
        \| \tilde g \|_{L^{p}(\widetilde \UPatch)}
        \leq 
        2^{\frac {n-\ell} p}
        \| \hat g \|_{L^{p}(\Delta^n)}
        .
    \end{align*}
    By the analysis for the Bogovski\u{\i} potential operators, Theorem~\ref{theorem:bogovpoinc}, 
    there exists $\mathring w \in W^{p}_{0}\Lambda^{k}(\widetilde \UPatch)$
    with $\cartan \mathring w = \cartan \tilde u$ satisfying the bounds 
    \begin{align}\label{math:mixedbconsimplex:exteriorderivative:internalconstant}
        \| \mathring w \|_{L^{p}(\widetilde \UPatch)}
        &\leq 
        C_{{\PF},\Bogov,\widetilde \UPatch,k,p}
        \| \tilde g \|_{L^{p}(\widetilde \UPatch)}
.
    \end{align}
    Here,
    \begin{align*}
        C_{{\PF},\Bogov,\widetilde \UPatch,k,p} 
        &
        = 
        C_{\Bogov}(n,\mia+1) \vol_{n-1}(S_{1}(0)) \frac{\diam(\widetilde \UPatch)^{n}}{\vol(\widetilde \UPatch)} \diam(\widetilde \UPatch)
        \\&
        = 
        C_{\Bogov}(n,\mia+1) \vol_{n-1}(S_{1}(0)) \frac{2^{n}}{2^{n-\ell}/n!} 2
        = 
        C_{\Bogov}(n,\mia+1) \vol_{n-1}(S_{1}(0)) \cdot n! 2^{\ell+1} 
        .
    \end{align*}
    We let $\tilde w \in W^{p}\Lambda^{k}(\widetilde \UPatch)$ be defined by averaging the $2^{n - \ell}$ reflections of $\mathring w$ along the coordinate axes $\ell+1$ through $n$. Thus,
    \begin{align*}
        \| \tilde w \|_{L^{p}(\Delta^{n})}
        =
        2^{-\frac{n-\ell}{p}}
        \| \tilde w \|_{L^{p}(\widetilde \UPatch)}
        ,
        \qquad 
        \| \tilde w \|_{L^{p}(\widetilde \UPatch)}
        \leq 
        \| \mathring w \|_{L^{p}(\widetilde \UPatch)}
        .
    \end{align*}
    Since $\tilde g$ has got the same symmetries, $\cartan \tilde w = \tilde g$.
    
    We let $w \in W^{p}\Lambda^{k}(T)$ be defined by $w := \varphi^{-\ast} ( \tilde w\restriction_{\Delta^n} )$.
    By construction, $\cartan w = \cartan u$,
    and $w$ has vanishing trace along $F_{0} \cup \dots \cup F_{\ell}$.
    
    To complete the discussion, we recall by Proposition~\ref{proposition:pullbackestimate} that 
    \begin{align*}
        \| w \|_{L^{p}(T)}
        &\leq 
        |\det(\Jacobian \varphi)|^{\frac 1 p} 
        \matnorm{ \Jacobian \varphi^{-1} }^{k}
        \| \tilde w \|_{L^{p}(\Delta^n)}
        ,
        \\
        \| \hat g \|_{L^{p}(\Delta^n)}
        &\leq 
        |\det( \Jacobian \varphi^{-1} )|^{\frac 1 p} 
        \matnorm{\Jacobian \varphi}^{k+1}
        \| \cartan u \|_{L^{p}(T)}
        .
    \end{align*}
    It remains to use Lemma~\ref{lemma:measurerelationships}, and the desired inequality is shown. 
\end{proof}

\begin{remark}\label{remark:mixedbconsimplex:exteriorderivative:hilbert}
    Under the assumptions of Lemma~\ref{lemma:mixedbconsimplex:exteriorderivative}, 
    the special case $p=2$ may allow for a better estimate for the auxiliary constant in~\eqref{math:mixedbconsimplex:exteriorderivative:internalconstant}.
    We conjecture the Poincar\'e--Friedrichs constants over bounded convex domains are non-increasing in the form degree $k$, as is already known for smoothly bounded convex domains~\cite{guerini2004eigenvalue}.
The complete inequality then reads: 
    there exists $w \in W^{2}\Lambda^{k}(T)$ such that 
    $\trace_{T,\Gamma} w = 0$
    and  
    \begin{align*}
        \cartan w = \cartan u,
        \quad 
        \| w \|_{L^{2}(\Omega)} 
        \leq 
        C_{\PF,T,\Gamma,\ell,k,2}
        \| \cartan u \|_{L^{2}(\Omega)}
        ,
    \end{align*}
    where 
    \begin{align*}
        C_{\PF,T,\Gamma,\ell,k,2}
        \leq 
        \frac{2}{\pi}
        \algebraicshapemeasure(T)^{k-1}
        \Ceins{n}(T) 
        \diam(T)
        .
    \end{align*}
    We conjecture that the constant can be improved further, to be independent of the tetrahedron's eccentricity,
    and that the result can be extended to Lebesgue exponents $1 \leq p \leq \infty$.
\end{remark}

We prepare some further notation for our two main results. 
Whenever $\calT$ is an $n$-dimensional triangulation, we use the abbreviation
\begin{align*}
    C_{\PF,\Gamma,k,p}(\calT) 
    :=
    \max\limits_{ \substack{ T \in \calT \\ \dim(T)=n \\ 0 \leq l \leq n } }
    C_{\PF,T,\Gamma,l,k,p}
    .
\end{align*}
Suppose that $\calT$ is an $n$-dimensional triangulation and has a shelling $T_{0}, T_{1}, \dots, T_{M}$. 
For each $0 \leq m \leq M$, we write 
\begin{align*}
    X_{m} := T_0 \cup T_1 \cup \dots \cup T_{m}.
\end{align*}
By Lemma~\ref{lemma:shell_triang_man}, each $X_{m}$ is a triangulated $n$-dimensional submanifold with boundary. 
For each $1 \leq m \leq M$, the interface of the new simplex to the previous intermediate triangulation is 
\begin{align*}
    \Gamma_{m} := T_{m} \cap X_{m-1} .
\end{align*}
According to Lemma~\ref{lemma:existenceofstar}, 
for each $1 \leq m \leq M$, 
there exists an interior simplex $S_{m} \in \calT$ 
such that $T_{m}$ is the last $n$-simplex in the shelling that belongs to $\patch_{\calT}(S_{m})$.
In particular, already $\patch_{\calT}(S_{m}) \subseteq X_{m}$. 

In addition to that, we introduce 
\begin{align*}
    \UPatch_{m-1} := \patch_{\calT}(S_{m}) \setminus T_{m}
\end{align*}
for the complement $\UPatch_{m-1} \subseteq X_{m-1}$ of the simplex $T_{m}$ in the star $\patch_{\calT}(S_{m})$. 

This manuscript's two main results for shellable triangulations of domains, not necessarily convex or even star-shaped, are the following theorems.
The first one (Theorem~\ref{theorem:poincarefriedrichsestimate:exterior}) is inspired by the recursive estimate in Theorem~\ref{theorem:poincarefriedrichsestimate:grad}: we always construct a potential over the new simplex using the trace value already known along the connecting faces, thereby constructing a potential over larger and larger subdomains and ensuring the correct continuity conditions.
The second one (Theorem~\ref{theorem:fullrecursivesum:exterior}) then unfolds the recursive estimates into the final Poincar\'e--Friedrichs inequality.

\begin{theorem}\label{theorem:poincarefriedrichsestimate:exterior}
Let $\calT$ be a shellable $n$-dimensional triangulation, and let the domain $\Omega \subseteq \bbR^{n}$ be the interior of the underlying set of $\calT$.
    Let $T_0, T_1, \dots, T_M$ be a shelling of $\calT$.
    Then for any $u \in W^{p}\Alt^{k}(\Omega)$, where $1 \leq p \leq \infty$, 
    there exists $w \in W^{p}\Alt^{k}(\Omega)$ with $\cartan w = \cartan u$
    and such that the following estimates hold:
    \begin{gather*}
        \| w \|_{L^{p}(T_{0})} \leq C_{{\PF},\Poinc,T_{0},k,p} \| \cartan u \|_{L^{p}(T_{0})}
        ,
    \end{gather*}
    and for each $T_{m} \in \calT$, $1 \leq m \leq M$, we have the recursive estimate 
    \begin{align*}
        \| w \|_{L^{p}(T_{m})}
        &
        \leq  
        C_{\PF,T_{m},\Gamma_{m},k,p}(\calT) 
        \left( 
            \| \cartan u      \|_{L^{p}(T_{m})} 
            +
            \Cfivetrafo{n}{k+1}{\ell}{p}(\calT)
\| \cartan u \|_{L^{p}(\UPatch_{m-1})}
        \right)
+ 
        \Cfivetrafo{n}{k}{\ell}{p}(\calT)
\| w \|_{L^{p}(\UPatch_{m-1})}
        ,
    \end{align*}
    where $0 \leq \ell < n$ is such that $T_{m}$ has $n - \ell$ faces in common with the previous simplices. 
\end{theorem}
\begin{proof}
    Let $u \in W^{p}\Alt^{k}(\Omega)$. 
    First, from Theorem~\ref{theorem:bogovpoinc}, there exists $w_0 \in W^{p}\Alt^{k}(T_{0})$ satisfying $\cartan w_0 = \cartan u$ over $T_{0}$ together with
    \begin{gather*}
        \| w_0 \|_{L^{p}(T_{0})} \leq C_{{\PF},\Poinc,T_{0},k,p} \| \cartan u \|_{L^{p}(T_{0})}.
    \end{gather*}
    Suppose that $0 < m \leq M$ such that there exists $w_{m-1} \in W^{p}\Alt^{k}(X_{m-1})$ with $\cartan w_{m-1} = \cartan u$ over $X_{m-1}$. 
    This is already true for $m = 1$.
    By assumption, $T_{m}$ and $X_{m-1}$ share the interface $\Gamma_{m}$, which is a collection of faces of $T_{m}$. 
    In accordance to Lemma~\ref{lemma:existenceofstar}, adding $T_{m}$ completes a star in $\calT$ around some simplex $S_{m}$, 
    and we let $\UPatch_{m-1} \subseteq X_{m-1}$ be the complement of $T_{m}$ in that newly completed star. 
    Write $\ell$ for the dimension of $S_{m}$.
    
    Similar as in~\cite[Equations~{\char`\(}5.12{\char`\)} and~{\char`\(}5.14{\char`\)}]{ern2020stable} or~\cite[Equations~{\char`\(}6.7{\char`\)} and~{\char`\(}6.9{\char`\)}]{Chaum_Voh_p_rob_3D_H_curl_24}, 
    we now define the field $\widetilde w_{m} \in W^{p}\Alt^{k}(T_{m})$ via 
    \begin{align}\label{equation:w''_ext_der}
        {\widetilde w}_{m} := u\restriction_{T_m} + \Xieins^{\ast}\left( (w_{m-1} - u)\restriction_{\UPatch_{m-1} } \right).
    \end{align} With $u \in W^{p}\Lambda^{k}(\Omega)$, Equation~\eqref{math:pullbackcommutes}, 
    and the mapping $\Xieins$ of Proposition~\ref{proposition:starreflection} being the identity along $\Gamma_m$,
    one verifies that 
    \begin{align*}
        \trace_{{\Gamma_m}} {\widetilde w}_{m}
        &
        = 
        \trace_{{\Gamma_m}} u\restriction_{T_m} + \trace_{{\Gamma_m}} \Xieins^{\ast} w_{m-1|\UPatch_{m-1} } - \trace_{{\Gamma_m}} \Xieins^{\ast} u\restriction_{\UPatch_{m-1} }
\\&
        = 
        \trace_{{\Gamma_m}} \Xieins^{\ast} w_{m-1|\UPatch_{m-1} }
        \\&
        = 
        \trace_{{\Gamma_m}} w_{m-1|\UPatch_{m-1}}
        .
    \end{align*}
    Moreover, again by \eqref{math:pullbackcommutes} 
    \begin{align*}
        \cartan {\widetilde w}_{m} 
        &
        = 
        \cartan u\restriction_{T_m} + \cartan \Xieins^{\ast} w_{m-1|\UPatch_{m-1} } - \cartan \Xieins^{\ast} u\restriction_{\UPatch_{m-1} } 
        \\&
        = 
        \cartan u\restriction_{T_m} + \Xieins^{\ast} \cartan w_{m-1|\UPatch_{m-1} } - \Xieins^{\ast} \cartan u\restriction_{\UPatch_{m-1} } 
        = 
        \cartan u\restriction_{T_m}
        ,
    \end{align*}
    since $\cartan w_{m-1} = \cartan u$ over $X_{m-1}$ and thus in particular over $\UPatch_{m-1}$.
    Setting $w_{m} := w_{m-1}$ over $X_{m-1}$ and $w_{m} := {\widetilde w}_{m} + w''_{m}$ over $T_{m}$,
    we thus verify that $w_{m} \in W^{p}\Alt^{k}(X_{m})$ with $\cartan w_{m} = \cartan u$ over $X_{m}$. 
    
    We now estimate norms.
    By construction,
    \begin{align*}
        \| w_{m} \|_{L^{p}(T_{m})}
        \leq  
        \| \Xieins^{\ast} w_{m}\restriction_{\UPatch_{m-1}} \|_{L^{p}(T_{m})}
        + 
        \| u - \Xieins^{\ast} u\restriction_{\UPatch_{m-1}} \|_{L^{p}(T_{m})}
        .
    \end{align*}
    Due to Lemma~\ref{lemma:mixedbconsimplex:exteriorderivative}, 
    which applies since $\trace_{\Gamma_{m}} \left( u\restriction_{T_m} - \Xieins^{\ast} u\restriction_{\UPatch_{m-1}} \right) = 0$, 
    and again using \eqref{math:pullbackcommutes} we have
    \begin{align*}
        \| u - \Xieins^{\ast} u\restriction_{\UPatch_{m-1}} \|_{L^{p}(T_{m})} 
        &
        \leq 
        C_{\PF,T_{m},\Gamma_{m},\ell,k,p} 
        \left( 
            \| \cartan u \|_{L^{p}(T_{m})} + \| \cartan \Xieins^{\ast} u\restriction_{\UPatch_{m-1}} \|_{L^{p}(T_{m})} 
        \right)
        \\&
        = 
        C_{\PF,T_{m},\Gamma_{m},\ell,k,p} 
        \left( 
            \| \cartan u \|_{L^{p}(T_{m})} + \| \Xieins^{\ast} \cartan u\restriction_{\UPatch_{m-1}} \|_{L^{p}(T_{m})} 
        \right)
        .
    \end{align*}
    Proposition~\ref{proposition:pullbackestimate} and Proposition~\ref{proposition:starreflection} now show 
    \begin{align*}
        \| \Xieins^{\ast} w_{m} \restriction_{\UPatch_{m-1}} \|_{L^{p}(T_{m})}
        &\leq 
        \Cfivetrafo{n}{k}{\ell}{p}(\calT)
\| w_{m} \|_{L^{p}(\UPatch_{m-1})}
        ,
        \\
        \| \Xieins^{\ast} \cartan u\restriction_{\UPatch_{m-1}} \|_{L^{p}(T_{m})}
        &\leq 
        \Cfivetrafo{n}{k+1}{\ell}{p}(\calT)
\| \cartan w_{m} \|_{L^{p}(\UPatch_{m-1})}
        .
    \end{align*}
    We have assumed that $\cartan w_{m-1} = \cartan u\restriction_{X_{m-1}}$. 
    The existence of $w \in W^{p}\Alt^{k}(\Omega)$ satisfying the recursive estimate follows 
    by setting $w\restriction_{T_m} := w_{m}$ for all $0 \leq m \leq M$.
\end{proof}

We finally derive an estimate for the Poincar\'e--Friedrichs constant, in analogy to Theorem~\ref{theorem:fullrecursivesum:grad}.
For ease of discussion, we choose to present the result for a general class of recursive estimates.
Of course, implementations should make use of the simplifications warranted by the specific recursive structure,
such as the recursion presented in Theorem~\ref{theorem:poincarefriedrichsestimate:exterior}. 

\begin{theorem}\label{theorem:fullrecursivesum:exterior}
Let $\calT$ be a shellable $n$-dimensional triangulation, and let the domain $\Omega \subseteq \bbR^{n}$ be the interior of the underlying set of $\calT$.
    Let $T_0, T_1, \dots, T_M$ be a shelling of $\calT$.
    Suppose that $u, w \in W^{p}\Alt^{k}(\Omega)$, where $1 \leq p \leq \infty$, 
    such that $\cartan w = \cartan u$ and the following recursive estimate holds:
    \begin{gather*}
        \| w \|_{L^{p}(T_{0})} 
        \leq 
        A_{0,0}
        \| \cartan u \|_{L^{p}(T_{0})}
        ,
    \end{gather*}
    and for each $T_{m} \in \calT$, $1 \leq m \leq M$, we have the recursive estimate 
    \begin{align*}
        \| w \|_{L^{p}(T_{m})}
        &
        \leq  
        \sum_{\ell=0}^{m} A_{\ell} \| \cartan u \|_{L^{p}(T_{\ell})} 
        +
        \sum_{\ell=0}^{m-1} B_{m,\ell} \| w \|_{L^{p}(T_{\ell})} 
        .
    \end{align*}
    Then we have an inequality of the form 
    \begin{align*}
        \| w \|_{L^{p}(\Omega)}
        &
        \leq 
        \left(
            \sum_{m=0}^{M}
            \left( \sum_{\ell=0}^{M} C_{m,\ell}^{q} \right)^{\frac p q}
        \right)^{\frac 1 p}
        \| \cartan u \|_{L^{p}(\Omega)}
        ,
    \end{align*}
    where $q \in [1,\infty]$ satisfies $1 = 1/p + 1/q$ and with obvious modifications if $p=1$ or $p=\infty$. 
    Here, 
    \begin{align*}
        C_{m,\ell} = \sum_{ m = i_L > \dots > i_1 \geq \ell } B_{i_{L},i_{L-1}} \cdots B_{i_{2},i_{1}} A_{i_{1},\ell}.
    \end{align*}
\end{theorem}
\begin{proof}    
Unwrapping the recursion, we obtain an estimate of the form 
    \begin{align*}
        \| w \|_{L^{p}(T_{m})} 
        \leq 
        \sum_{\ell = 0}^{m}
        \underbrace{ 
            \left( 
                \sum_{ m = i_L > \dots > i_1 \geq \ell }
                B_{i_{L},i_{L-1}} \cdots B_{i_{2},i_{1}} A_{i_{1},\ell}
            \right)
        }_{ =: C_{m,\ell} }
        \| \cartan u \|_{L^{p}(T_{\ell})} 
    \end{align*}
    Here, $C_{m,\ell}$ denotes the coefficient of $\| \cartan u \|_{L^{p}(T_{\ell})}$, possibly zero, 
    as it appears in the unwrapped recursive estimate of $\| w \|_{L^{p}(T_{m})}$. 
    Once again, The global Poincar\'e--Friedrichs inequality follows via H\"older's inequality:
    \begin{align*}
        \| w \|_{L^{p}(\Omega)}^{p}
        &
        = 
        \sum_{m=0}^{M}
        \| w \|_{L^{p}(T_{m})}^{p}
\leq 
        \sum_{m=0}^{M}
        \left( \sum_{\ell=0}^{M} C_{m,\ell}^{q} \right)^{\frac p q}
        \sum_{\ell'=0}^{M} \| \cartan u \|_{L^{p}(T_{\ell'})}^{p} 
        \leq 
        \left(
            \sum_{m=0}^{M}
            \left( \sum_{\ell=0}^{M} C_{m,\ell}^{q} \right)^{\frac p q}
        \right)
        \| \cartan u \|_{L^{p}(\Omega)}^{p} 
        ,
    \end{align*}
    where $q \in [1,\infty]$ is as described above. The proof is complete. 
\end{proof}

\section{Numerical examples}\label{section:numericalexamples}

We wish to assess the practical quality of our upper bounds, with exclusive focus on the Hilbert space case $p=2$.
To that end, we compare our estimates for the Poincar\'e--Friedrichs constants with the exact constants of a few examples domains in dimension 2 and 3.
In lieu of the exact values, finite element eigenvalue computations on a refined mesh provide reference values as a proxy for the exact value.

Theorem~\ref{theorem:fullrecursivesum:exterior} estimates the gradient, curl and divergence Poincar\'e--Friedrichs constants.
In addition, Theorem~\ref{theorem:fullrecursivesum:grad} provides another upper bound for the gradient Poincar\'e--Friedrichs constant.
We remark that the Poincar\'e--Friedrichs constant for the divergence can always be estimated by less than the diameter of the domain (see Lemma~\ref{lemma:mixedbconsimplex}). 

\subsection{Software and algorithms}

All computations were implemented with an in-house C++ code implementing finite element exterior calculus.
The reference Poincar\'e constants are the inverse square roots of reference eigenvalues, which are computed as solutions of finite element eigenvalue problems in mixed formulation over sufficiently refined meshes.

To compute the upper bound for the Poincar\'e constant of Theorem~\ref{theorem:fullrecursivesum:grad}, we need to find paths between simplices.
One quickly sees that a spanning tree of the face-connection graph provides paths that minimize the estimate.
We choose a depth-first search that attempts to minimize the resulting constant.
While this is feasible for small two-dimensional examples, the computation takes considerably longer for our three-dimensional examples.

We need a shelling to apply Theorem~\ref{theorem:fullrecursivesum:grad} and to estimate the Poincar\'e--Friedrichs constant.
Finding a shelling optimizing a geometric target quantity is a challenging problem in computational geometry and theoretical computer science.
Brute-force enumeration of shellings is still feasible for small two-dimensional triangulations but becomes practically infeasible in dimension three because the number of shellings sees a combinatorial explosion.
Instead, we employ a backtracking search guided by a greedy strategy:
Starting with a single simplex, we successively add simplices to the partial shellings until a complete shelling is found.
Among the candidates for the next simplex, our practical heuristic prefers those that introduce the smallest possible geometric constants. 
We always prioritize the completion of face stars because then the induced constants are easier to control via Lemma~\ref{lemma:volumecomparison}.
Repeating this backtracking search for all possible initial simplices and generating up to $10$ shellings each time, we pick the optimal shelling among those found.
From our computations, we have observed that the Poincar\'e--Friedrichs estimates from different shellings may differ by orders of magnitudes;
that said, the method may not necessarily produce the global optimum.

In order to ensure tight estimates, we compute the relevant geometric parameters, such as heights, individually for each simplex and use individualized estimates for each local star. These estimates are numerically tighter than the ones stated in Theorems~\ref{theorem:fullrecursivesum:grad} and~\ref{theorem:fullrecursivesum:exterior}, as we directly implement the calculations used in the proofs.

\subsection{Estimates for partial boundary conditions}

The Poincar\'e--Friedrichs constants over simplices subject to boundary conditions along a few faces frequently enter our estimates.
We therefore prefer to have tight upper estimates for them.

Whenever $T$ is a simplex, the Poincar\'e--Friedrichs constants have upper bounds of the form $C \diam(T)$ for some numerical constant $C > 0$.
Lemma~\ref{lemma:mixedbconsimplex:exteriorderivative} and the subsequent remark (using $p=2$) have established that $C = 2/\pi \approx 0.45015815807$. In the special case of the gradient potential, where the Poincar\'e constant can be characterized variationally, the Poincar\'e constant without boundary conditions is already an upper bound for the Poincar\'e constant with boundary conditions, and one can choose $C = 1/\pi \approx 0.31830988618$.
This conforms to the well-known fact that the eigenvalues of the scalar Laplacian with mixed boundary conditions are between the Neumann and Dirichlet eigenvalues.
Finally, the factor $1/\pi$ in the previous two estimates has a known improvement when the simplex is a triangle:
then we can replace the factor $1/\pi$ by $1/j_{1,1} \approx 0.2609803592$, where $j_{1,1}$ is the zero of a Bessel function; see~\cite{laugesen2010minimizing}.

For the purpose of comparison, we compute reference Poincar\'e--Friedrichs constants for $p=2$ over the reference tetrahedron via a higher-order finite element method over a mesh after three iterations of uniform refinement.
These data are tabulated in Table~\ref{table:referencepoincarefriedrichs}. 
The upper bound provided by Lemma~\ref{lemma:mixedbconsimplex:exteriorderivative} and Remark~\ref{remark:mixedbconsimplex:exteriorderivative:hilbert} has an overestimate by a factor of at most six.
The estimate seems within practical range, at least on simplices with good shape regularity. 
We expect the overestimate to worsen as the eccentricity of the tetrahedron increases.
The results indicate room for improvement and future research should tighten the gap towards the practically observed values, including for ill-shaped simplices.

\begin{table}[t]
    \caption{
    Reference $L^2$ Poincar\'e--Friedrichs constants for the gradient, curl, and divergence operators over the reference tetrahedron,
    where boundary conditions are imposed along the boundary part $\Gamma$, 
    computed via lowest-order finite element methods on a refined mesh. 
    Notably, the top-left should and the bottom-right should asymptotically coincide, as should the bottom-left and the top-right.
    By comparison, Lemma~\ref{lemma:mixedbconsimplex:exteriorderivative} and Remark~\ref{remark:mixedbconsimplex:exteriorderivative:hilbert} 
    give the upper bound $2/\pi \cdot \diam(\Delta^3) \approx 0.90031631615$.
    }\label{table:referencepoincarefriedrichs}
    \centering
    \begin{tabular}{|l|c|c|c|}
        \hline
                                                  & $\grad$           & $\curl$           & $\divergence$   \\ \hline
        $\Gamma = F_0 \cup F_1 \cup F_2 \cup F_3$ & 0.0862501765      & 0.1453729386      & 0.2601720480    \\ \hline
        $\Gamma =          F_1 \cup F_2 \cup F_3$ & 0.1093817645      & 0.1829680131      & 0.3493931507    \\ \hline
        $\Gamma =                   F_2 \cup F_3$ & 0.1450219664      & 0.2428248005      & 0.1845090722    \\ \hline
        $\Gamma =                            F_3$ & 0.2057601732      & 0.1527746636      & 0.1223402289    \\ \hline
        $\Gamma =                      \emptyset$ & 0.2631059409      & 0.1458215887      & 0.0874066554    \\ \hline
    \end{tabular}
\end{table}

\subsection{Two-dimensional examples}

We consider the following example domains in two dimensions:
the unit square $\Omega_Q$, the L-shaped domain $\Omega_L$, and the slit domain $\Omega_{S}$:
\begin{gather*}
    \Omega_{Q} = ( 0,1)^2,
    \qquad 
    \Omega_{L} = (-1,1)^2 \setminus [0,1)^2,
    \\
    \Omega_{S} = (-1,1)^2 \setminus ( [0,1) \times \{0\} ).
\end{gather*}
We consider the following triangulations; see also Figure~\ref{figure:triangulations2D}:
\begin{itemize}
    \item $\calT_{Q}$:  square triangulation with two triangles
    \item $\calT_{L}$:  L-shaped domain triangulation with four triangles
    \item $\calT_{S}$:  Slit domain triangulation, five triangles 
    \item $\calT_{S'}$: Slit domain triangulation, 8 triangles, all of which touch the origin.
    \item $\calT_{S''}$: Slit domain triangulation, 8 triangles, four of which touch the origin.
\end{itemize}

The de~Rham complex of a planar domain has only two differential operators,
and their Poincar\'e--Friedrichs constants are the inverse square roots of the smallest non-zero Dirichlet and Neumann Laplacians. 
Standard finite element eigenvalue computations on any sufficiently refined mesh provide acceptable reference values for these;
we choose four steps of uniform refinement.

We compare the reference Poincar\'e constants with the gradient potential estimate in Theorem~\ref{theorem:fullrecursivesum:grad}.
Moreover, we compare the reference Poincar\'e--Friedrichs constants with the bounds computed in Theorem~\ref{theorem:fullrecursivesum:exterior}. Here, we take into account Lemma~\ref{lemma:volumecomparison} and Proposition~\ref{proposition:starreflection}.

The shelling-based estimate for the Poincar\'e constant of the gradient potential outperforms the estimate in Section~\ref{section:gradient} in our examples where the triangles are not all congruent to each other ($\calT_{L}$ and $\calT_{S,5}$).
That is a result of our implementation, where the face-based reflections are not necessarily onto, which allows for some improved estimates.
The results are summarized in Table~\ref{table:estimates2D}. The maximal overestimation factor is below 16, which we deem reasonable.

\begin{figure}
    \caption{Triangulations of 2D domains used in our numerical experiments. Dashed lines represent interior triangle edges.}\label{figure:triangulations2D}
    \centering
    \begin{tikzpicture}
\coordinate (A) at (0,0);
      \coordinate (B) at (2,0);
      \coordinate (C) at (2,2);
      \coordinate (D) at (0,2);

\draw[line width=0.4mm] (A) -- (B) -- (C) -- (D) -- cycle;
      \draw[dashed] (A) -- (C);
    \end{tikzpicture}
    \qquad
    \begin{tikzpicture}
\coordinate (A) at (0,0);
      \coordinate (B) at (2,0);
      \coordinate (C) at (2,1);
      \coordinate (D) at (1,1);
      \coordinate (E) at (1,2);
      \coordinate (F) at (0,2);

\draw[line width=0.4mm] (A) -- (B) -- (C) -- (D) -- (E) -- (F) -- cycle;

\draw[dashed] (A) -- (D);
      \draw[dashed] (F) -- (D);
      \draw[dashed] (B) -- (D);
    \end{tikzpicture}
    \qquad
    \begin{tikzpicture}
\coordinate (A) at (0,0);
      \coordinate (B) at (2,0);
      \coordinate (C) at (2,1);
      \coordinate (D) at (1,1); \coordinate (E) at (2,2);
      \coordinate (F) at (0,2);

\draw[line width=0.4mm] (C) -- (D);
      \draw[line width=0.4mm] (A) -- (B) -- (C);
      \draw[line width=0.4mm] (C) -- (E) -- (F) -- (A);

\draw[dashed] (A) -- (D); \draw[dashed] (B) -- (D); \draw[dashed] (E) -- (D); \draw[dashed] (F) -- (D); \end{tikzpicture}
    \qquad
    \begin{tikzpicture}
\coordinate (A) at (0,0);
      \coordinate (B) at (2,0);
      \coordinate (C) at (2,1);
      \coordinate (D) at (1,1); \coordinate (E) at (2,2);
      \coordinate (F) at (0,2);

      \coordinate (U) at (1,0);
      \coordinate (L) at (0,1);
      \coordinate (O) at (1,2);

\draw[line width=0.4mm] (C) -- (D);
      \draw[line width=0.4mm] (A) -- (B) -- (C);
      \draw[line width=0.4mm] (C) -- (E) -- (F) -- (A);

\draw[dashed] (A) -- (D); \draw[dashed] (B) -- (D); \draw[dashed] (E) -- (D); \draw[dashed] (F) -- (D); \draw[dashed] (U) -- (D); \draw[dashed] (L) -- (D); \draw[dashed] (O) -- (D); \end{tikzpicture}
    \qquad
    \begin{tikzpicture}
\coordinate (A) at (0,0);
      \coordinate (B) at (2,0);
      \coordinate (C) at (2,1);
      \coordinate (D) at (1,1); \coordinate (E) at (2,2);
      \coordinate (F) at (0,2);

      \coordinate (U) at (1,0);
      \coordinate (L) at (0,1);
      \coordinate (O) at (1,2);

\draw[line width=0.4mm] (C) -- (D);
      \draw[line width=0.4mm] (A) -- (B) -- (C);
      \draw[line width=0.4mm] (C) -- (E) -- (F) -- (A);

\draw[dashed] (U) -- (L); \draw[dashed] (L) -- (O); \draw[dashed] (U) -- (C); \draw[dashed] (O) -- (C); \draw[dashed] (U) -- (D); \draw[dashed] (L) -- (D); \draw[dashed] (O) -- (D); \end{tikzpicture}
\end{figure}
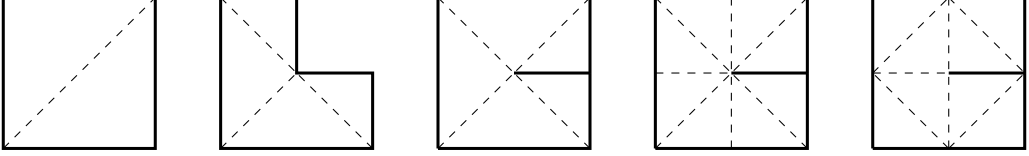

\begin{table}[t]
    \caption{
    Estimates for $L^2$ Poincar\'e--Friedrichs constants over various triangulated 2D domains.
    Reference values for the gradient and divergence (2nd and 7th column) computed with finite element methods together with estimates and ratios: 
    Theorem~\ref{theorem:fullrecursivesum:grad} (3rd and 4th column), using Theorem~\ref{theorem:fullrecursivesum:exterior} with $k=0$ (5th and 6th column), 
    and using Theorem~\ref{theorem:fullrecursivesum:exterior} with $k=1$ (8th and 9th column).
    Note that the divergence constant can always be estimated using Lemma~\ref{lemma:mixedbconsimplex}.
    }\label{table:estimates2D}
    \centering
    \begin{tabular}{|r||r|r|r|r|r||r|r|r|}
        \hline
                        & $\grad$ ref  & $\grad$ est & $\grad$ ratio       & $\grad$ est   & $\grad$ ratio      & $\divergence$ ref  & $\divergence$ est  & $\divergence$ ratio \\ 
        \hline
        $\calT_{Q}$     & 0.318        & 0.904       & 2.842               & 0.904         & 2.842              & 0.225              & 1.339              &  5.953         \\
        \hline
        $\calT_{L}$     & 0.822        & 2.381       & 2.896               & 1.421         & 1.729              & 0.322              & 2.505              &  7.781         \\
        \hline
        $\calT_{L'}$    & 0.822        & 2.391       & 2.909               & 2.391         & 2.909              & 0.322              & 4.053              & 12.587         \\
        \hline
        $\calT_{S}$     & 0.978        & 2.752       & 2.814               & 1.779         & 1.819              & 0.346              & 3.717              & 10.744         \\ 
        \hline $\calT_{S'}$    & 0.978        & 3.131       & 3.202               & 3.131         & 3.202              & 0.346              & 5.453              & 15.762         \\
        \hline $\calT_{S''}$   & 0.978        & 2.761       & 2.824               & 2.761         & 2.824              & 0.346              & 4.662              & 13.476         \\
        \hline
    \end{tabular}
\end{table}

\subsection{Three-dimensional examples}

We consider the following example domains in three dimensions:
the unit cube $\Omega_C$, the Fichera corner domain $\Omega_F$, and the crossed bricks domain $\Omega_{B}$:
\begin{gather*}
    \Omega_{C}  = ( 0,1)^3,
    \qquad 
    \Omega_{F}  = (-1,1)^3 \setminus [0,1)^3,
    \\
    \begin{aligned}
    \Omega_{B} &= 
    \left( (-1,0) \times (-1,0) \times (-1,1) \right)
    \\&\qquad
    \cup 
    \left( (-1,0) \times (-1,1) \times (-1,0) \right)
    \\&\qquad
    \cup 
    \left( (-1,1) \times (-1,0) \times (-1,0) \right)
    .
    \end{aligned}
\end{gather*}
The following triangulations are used in computations: 
\begin{itemize}
    \item $\calT_{C,5}$: cube triangulation with five tetrahedra
    \item $\calT_{C,K}$: Kuhn triangulation of the cube, consisting of six tetrahedra
    \item $\calT_{B,5}$: crossed bricks, four copies of $\calT_{C,5}$
    \item $\calT_{B,K}$: crossed bricks, four copies of $\calT_{C,K}$
    \item $\calT_{F}$: Fichera corner, $24$ simplices 
\end{itemize}

The reference Poincar\'e--Friedrichs constants for the gradient, curl, and divergence operators are found via standard finite element eigenvalue computations, using the lowest-order finite element de~Rham complex over a sufficiently refined mesh (four steps of uniform refinement).
Again, we compare these reference values with estimates obtained Theorem~\ref{theorem:fullrecursivesum:exterior}, and in the special case of the gradient, with Theorem~\ref{theorem:fullrecursivesum:grad}.
The results, summarized in Tables~\ref{table:estimates3D:grad} and~\ref{table:estimates3D:curldiv}, are less satisfactory than in the two-dimensional case; we deem the overestimation factors as still reasonable for patches with few tetrahedra, but they are very high for the Fichera corner with 24 simplices.

\begin{table}[t]
    \caption{
    Estimates for $L^2$ Poincar\'e--Friedrichs constants of the gradient over various triangulated 3D domains.
    Reference values for the gradient (2nd column) computed with finite element methods together with estimates and ratios: 
    Theorem~\ref{theorem:fullrecursivesum:grad} (3rd and 4th column), and using Theorem~\ref{theorem:fullrecursivesum:exterior} with $k=0$ (5th and 6th column).
    }\label{table:estimates3D:grad}
    \centering
    \begin{tabular}{|r||r|r|r|r|r||r|r|r|}
        \hline
                        & $\grad$ ref  & $\grad$ est  & $\grad$ ratio       & $\grad$ est  & $\grad$ ratio  \\
        \hline
        $\calT_{C,5}$   & 0.317        & 4.317        & 13.581              & 3.246        & 10.214         \\
        \hline
        $\calT_{C,K}$   & 0.317        & 3.571        &  11.23              & 11.034       & 34.722         \\
        \hline
        $\calT_{B,5}$   & 1.022        & 7.698        & 7.5285              & 25.972       & 25.397         \\
        \hline
        $\calT_{B,K}$   & 1.022        & 10.106       & 9.8824              & 53.763       & 52.573         \\
        \hline
        $\calT_{F}$     & 0.711        & 10.622       & 14.927              &  268.071     & 377.033        \\ 
        \hline
    \end{tabular} 
\end{table}
  
\begin{table}[t]  
    \caption{
    Estimates for $L^2$ Poincar\'e--Friedrichs constants over various triangulated 3D domains.
    Reference values for the curl and divergence (2nd and 5th column) computed with finite element methods together with estimates and ratios
    that rely on Theorem~\ref{theorem:fullrecursivesum:exterior} with $k=1$ (3rd and 4th column) and with $k=2$ (5th and 6th column).
    Note that the divergence constant can always be estimated using Lemma~\ref{lemma:mixedbconsimplex}.
    }\label{table:estimates3D:curldiv}
    \centering
    \begin{tabular}{|r||r|r|r||r|r|r|}
        \hline
                        & $\curl$ ref  & $\curl$ est & $\curl$ ratio          & $\divergence$ ref & $\divergence$ est & $\divergence$ ratio \\ 
        \hline
        $\calT_{C,5}$   & 0.225        & 141.148     & 627.310                & 0.183             & 3.391             & 18.453          \\
        \hline
        $\calT_{C,K}$   & 0.225        & 12.157      & 54.030                 & 0.183             & 25.899            & 140.920         \\
        \hline
        $\calT_{B,5}$   & 0.331        & 113.084     & 341.056                & 0.233             & 152.886           & 655.993         \\
        \hline
        $\calT_{B,K}$   & 0.331        & 162.687     & 490.655                & 0.233             & 512.273           & 2198.026        \\
        \hline
        $\calT_{F}$     & 0.554        & 25752.342   & 46403.302              & 0.310             & 8958.467          & 28835.266       \\ 
        \hline
    \end{tabular}
\end{table}

\section{Outlook}\label{section:outlook}

This manuscript contributes upper bounds for Poincar\'e--Friedrichs constants. The main application is the curl operator over local patches in low spatial dimensions (i.e., $n=2$ or $n=3$).
However, we believe there are ample opportunities to refine our estimates of Poincar\'e--Friedrichs constants, both conceptually and algorithmically, and extend the technique to further applications.

We use local Poincar\'e--Friedrichs inequalities over single simplices, subject to boundary conditions along some faces.
We believe that these inequalities can be tightened.
In particular, we conjecture that the improvements for triangles~\cite{laugesen2010minimizing} have analogs for tetrahedra.

The gradient Poincar\'e--Friedrichs constant already dominates the Poincar\'e--Friedrichs constants of the $\Lebesgue^2$ de~Rham complex without boundary conditions (or with full boundary conditions) 
over bounded smooth convex domains~\cite{guerini2004eigenvalue}.
We believe that this relationship extends to the entire range of Lebesgue exponents $1 \leq p \leq \infty$ over bounded convex domains;
this would improve on the Poincar\'e--Friedrichs constants derived via the regularized Poincar\'e and Bogovski\u{\i} operators.

Our method hinges on extending differential forms onto a simplex from the complement of that simplex within a local star.
The present manuscript realizes this extension via pullback along a bi-Lipschitz mapping.
It seems worthwhile to explore other techniques, e.g., carefully analyzing trace and extension theorems.

On the algorithmic part, merely checking whether a general triangulation has a shelling is known to be an NP-complete problem and hence computationally infeasible in practice~\cite{goaoc2019shellability}.
Nevertheless, there might be more efficient algorithms for constructing shellings for practically relevant triangulations while optimizing geometric target heuristics.

We anticipate the present results to generalize to shellable polytopal complexes.
Apart from the intrinsic interest in polytopal complexes, such extensions are already relevant for simplicial triangulations:
As a rule of thumb, the estimates deteriorate as the number of simplices increases.
Lumping simplices into a few polytopal subdomains for which the Poincar\'e--Friedrichs constants are controllable could improve the estimates.
For example, the crossed bricks domain easily partitions into four quadrilaterals.

Lastly, estimating Poincar\'e--Friedrichs constants over domains and manifolds with shellable triangulations is fundamentally restricted to topological balls and spheres.
However, we expect our estimates over local patches to serve as subcomponents in estimating Poincar\'e--Friedrichs constants of general $n$-dimensional triangulated manifolds.

\section*{Acknowledgments}
MWL appreciates helpful remarks by Abdellatif Aitelhad, Martin Costabel, John M.\ Lee, and G\"unther M.\ Ziegler.

\appendix

\section{Outline of an alternative estimate}

We explore a family of alternative estimates, based on a variation of the results of Section~\ref{section:extension} and Section~\ref{section:poincarefriedrichs}.
Though we consider these ideas to be of interest, the resulting estimates for Poincar\'e--Friedrichs constants are not competitive in practice.

We begin with an analogue of Proposition~\ref{proposition:starreflection}.
Domains with shellable triangulations are homeomorphic to a ball and hence contractible.
Inspired by that observation, we show that shellable triangulations can be continuously retracted into a single simplex via a sequence of bi-Lipschitz transformations.
At each step, this bi-Lipschitz transformation is defined as contracting a simplex $T$ into a local star around one of the subsimplices of $T$.

\begin{proposition}\label{proposition:deformation}
    Let $\calT$ be a triangulation of an $n$-dimensional domain.
    Let $S \in \calT$ be an inner simplex of dimension $\ell < n$,
    let $T \in \patch_{\calT}(S)$ be of dimension $n$,
    and let
    \begin{align*}
        \UPatch := \overline{ \underlying{\patch_{\calT}(S)} \setminus T },
        \qquad
        {\Gamma_2} := \overline{ \partial \UPatch \setminus \partial T }.
    \end{align*}
    The following holds, where the constants on the right-hand sides are as stated in the proof.
    
    There exists a bi-Lipschitz piecewise affine mapping
    \begin{align*}
        \Xizwei : \underlying{\patch_{\calT}(S)} \rightarrow \Xizwei( \underlying{\patch_{\calT}(S)} ) \subseteq \UPatch
    \end{align*}
    which is the identity along $\partial \UPatch \setminus \partial T$.
    At any point, the singular values $\sigma_1 \geq \dots \geq \sigma_n$ of its Jacobian satisfy
    \begin{gather*}
        \sigma_1 \leq \Csieb{n}{\ell}(\calT),
        \quad
        \sigma_2,\dots,\sigma_n \leq \Csiebprime{n}{\ell}(\calT),
        \quad
        | \det\Jacobian\Xizwei |      \leq \Cdetsieb{n}{\ell}(\calT),
        \\
        \sigma_n^{-1} \leq \Cacht{n}{\ell}(\calT),
        \quad
        \sigma_1^{-1},\dots,\sigma_{n-1}^{-1} \leq \Cachtprime{n}{\ell}(\calT),
        \quad
        | \det\Jacobian\Xizwei^{-1} | \leq \Cdetacht{n}{\ell}(\calT).
    \end{gather*}
    Moreover, for any $0 \leq k \leq n$ and $p \in [1,\infty]$,
    \begin{align*}
        \sigma_{1}\cdots\sigma_{k} | \det\Jacobian\Xizwei |^{-\frac 1 p}
        \leq
        \Csiebtrafo{n}{k}{\ell}{p}(\calT),
        \qquad
        \sigma_{n}^{-1}\cdots\sigma_{n-k+1}^{-1} | \det\Jacobian\Xizwei |^{\frac 1 p}
        \leq
        \Cachttrafo{n}{k}{\ell}{p}(\calT).
    \end{align*}
\end{proposition}

\begin{proof}
    This proof is to be read as a continutation of the proof of Proposition~\ref{proposition:starreflection},
    with the same notation and definitions.

    Notice that $\overline{\partial\underlying{\patch_{\calT}(S)} \setminus {\Gamma_2}}$ is the union of exactly those faces of $T$ that contain $S'$.

    We introduce another mapping $\Phi : \underlying{\calK^\ast} \rightarrow \underlying{\calK^{c}}$ as follows.
    Consider any $n$-simplex $K \in \calK$ and let $K^c \in \calK^{c}$ be its counterpart as defined earlier in the proof. We construct a bi-Lipschitz mapping
    \begin{align*}
        \Phi_{K} : K \cup K^{c} \rightarrow K^{c}.
    \end{align*}
    The construction will be such that the union of $\Phi_{K}$ for all $n$-simplices $K \in \calK$
    will define the desired bi-Lipschitz mapping $\Phi : \underlying{\calK^\ast} \rightarrow \underlying{\calK^{c}}$,
    and $\Phi$ will be the identity along $\partial\underlying{\calK^\ast} \setminus \partial\underlying{\calK}$.

    Once again, $h_{z}$ denotes the height of $z_{S'}$ within $K$.
    Here, we let $Q \subseteq K \cap K^{c}$ be the subsimplex
    that is complementary to the line segment from the origin to $z_{S'}$ in $K$.
    Equivalently, $Q$ is complementary to the line segment from the origin to $y$ in $K^{c}$,
    and $Q$ is complementary to the origin within the face $K \cap K^{c}$.
    From the definition of simplices, we now conclude that any $x \in K \cup K^{c}$
    has a unique representation
    \begin{align*}
        x = \lambda z_{S'} + \mu x_{Q}, \quad \lambda \in [-\rho,1], \quad \mu \in [0,1], \quad |\lambda| + \mu \leq 1, \quad x_{Q} \in Q.
    \end{align*}
    Since $\mu x_{Q}$ lies in the hyperplane spanned by the origin and $Q$,
    on which $h_{z}$ stands orthogonally, we have
    \begin{align*}
        \lambda = \lambda(x) := \frac{\langle h_{z},x\rangle}{\langle h_{z},z_{S'}\rangle}
        .
    \end{align*}
    Based on that observation, we define
    \begin{align*}
        \Phi_{K}(x)
        &
        :=
        \mu x_{Q} + \frac{\rho}{\rho+1} \left( \lambda(x) - 1 \right) z_{S'}
        \\&
        =
        \mu x_{Q} + \frac{\rho}{\rho+1} \lambda(x) z_{S'} - \frac{\rho}{\rho+1} z_{S'}
        \\&
        =
        x - \frac{\rho+1}{\rho+1} \lambda(x) z_{S'} + \frac{\rho}{\rho+1} \lambda(x) z_{S'} - \frac{\rho}{\rho+1} z_{S'}
        \\&
        =
        x - \frac{1}{\rho+1} \lambda(x) z_{S'} - \frac{\rho}{\rho+1} z_{S'}
        \\&
        =
        x - \frac{1}{\rho+1} \frac{\langle h_{z},x\rangle}{\langle h_{z},z_{S'}\rangle} z_{S'} - \frac{\rho}{\rho+1} z_{S'}
        .
    \end{align*}
    We readily verify that this transformation is a bi-Lipschitz mapping from $K \cup K^{c}$ onto $K^{c}$
    that satisfies the desired mapping properties.
    It remains to analyze its Jacobian and get explicit estimates for its singular values.

    We once more introduce an orthogonal decomposition $h_{z} + b_{z} = z_{S'}$, where $b_{z} \in \bbR^{n}$.
    With that,
    \begin{align*}
        \Jacobian \Phi_{K}(x)
        &=
        \Id
        -
        \frac{ 1 }{ 1 + \rho }
        \frac{ z_{S'} \otimes h_{z}^{t} }{ \langle h_{z},z_{S'}\rangle }
        \\
        &=
        \Id
        -
        \frac{ 1 }{ 1 + \rho }
        \frac{ z_{S'} \otimes h_{z}^{t} }{ \langle h_{z},h_{z}\rangle }
        \\
        &=
        \Id
        -
        \frac{ 1 }{ 1 + \rho }
        \frac{ h_{z} \otimes h_{z}^{t} }{ \langle h_{z},h_{z}\rangle }
        -
        \frac{ 1 }{ 1 + \rho }
        \frac{ b_{z} \otimes h_{z}^{t} }{ \langle h_{z},h_{z}\rangle }
        \\
        &=
        \Id
        -
        \frac{ 1 }{ 1 + \rho }
        \hat h_{z} \otimes \hat h_{z}^{t}
        -
        \frac{ \vecnorm{ b_{z} } / \vecnorm{ h_{z} } }{ 1 + \rho }
        \hat b_{z} \otimes \hat h_{z}^{t}
        .
    \end{align*}
    The Jacobian acts as the identity over the orthogonal complement of the span of $h_{z}$ and $z_{S'}$.
    We write $\beta$ for the angle between $h_{z}$ and $z_{S'}$.
    Hence, $\tan(\beta) = \vecnorm{ b_{z} } / \vecnorm{ h_{z} }$.
    Direct computation shows
    \begin{align}
        \Jacobian \Phi_{K}(x) \cdot \hat b_{z} = \hat b_{z}
        ,
        \qquad
        \Jacobian \Phi_{K}(x) \cdot \hat h_{z} = \frac{ \rho }{ 1 + \rho } \hat h_{z} - \frac{ \vecnorm{ b_{z} } / \vecnorm{ h_{z} } }{ 1 + \rho } \hat b_{z}
        .
    \end{align}
    It remains to study the singular values of the matrix
    \begin{align*}
        M_{\Phi,K}
        =
        \begin{pmatrix}
                \frac{ \rho }{ 1 + \rho }    & 0
            \\
            -\frac{ \tan(\beta) }{1+\rho} & 1
        \end{pmatrix}
        .
    \end{align*}
    Computing the eigenvalues of the symmetric matrix $M_{\Phi,K}^{t} M_{\Phi,K}$,
    we obtain the singular values
    \begin{align*}
        \sigma_{\max}(\Phi,K)
        &
        =
        \frac{1}{2(1+\rho)} \sqrt{ (2\rho + 1)^{2} + \tan(\beta)^2 } + \frac{1}{2(1+\rho)} \sqrt{ 1 + \tan(\beta)^2 }
        ,
        \\
        \sigma_{\min}(\Phi,K)
        &
        =
        \frac{1}{2(1+\rho)} \sqrt{ (2\rho + 1)^{2} + \tan(\beta)^2 } - \frac{1}{2(1+\rho)} \sqrt{ 1 + \tan(\beta)^2 }
        .
    \end{align*}
    These are monotonically increasing from $1$ and decreasing from $0.5$, respectively, in $\tan(\beta)$.
    Hence, these are also the maximal and minimal singular values of the Jacobian $\Jacobian \Phi_{K}$, the remaining singular values being equal to $1$.
    Notice that
    \begin{align*}
        \det \Jacobian\Phi = \sigma_{\max}(\Phi,K) \sigma_{\min}(\Phi,K) = \frac{\rho(1+\rho)}{(1+\rho)^2} = \frac{\rho}{1+\rho}.
    \end{align*}
    We now recall that the height of $h_{z}$ in $K \in \calK$ equals $(\ell+1)^{-1}$ multiplied with the height of some vertex of $S$ within $T$.
    Similar as above, we use the upper bound
    \begin{align*}
        \tan(\beta) = \frac{\vecnorm{ b_{z} }}{\vecnorm{ h_{z} }} \leq (\ell+1) \aspectratio(T).
    \end{align*}
    We conclude that the singular values of the Jacobian of $\Phi$ at almost every $x$ satisfy
    \begin{gather}
        \sigma_{1}(\Phi,x) \leq
        \frac{1}{2(1+\rho)} \sqrt{ (2\rho + 1)^{2} + (\ell+1)^2 \aspectratio(T)^2 }
        +
        \frac{1}{2(1+\rho)} \sqrt{               1 + (\ell+1)^2 \aspectratio(T)^2 }
        \label{math:phi_max}
        ,
        \\
        \sigma_{2}(\Phi,x) = \dots = \sigma_{n-1}(\Phi,x) = 1,
        \\
        \sigma_{n}(\Phi,x)^{-1} =
        \frac{ \sigma_{1}(\Phi,x) }{ \left| \det \Jacobian\Phi \right| }
        \leq
\frac{1}{2\rho}
        \sqrt{ (2\rho + 1)^{2} + (\ell+1)^2 \aspectratio(T)^2 }
        +
\frac{1}{2\rho}
        \sqrt{               1 + (\ell+1)^2 \aspectratio(T)^2 }
        \label{math:phi_min}
        .
    \end{gather}
    This finishes the discussion of the transformation $\Phi$.
    Finally, we obtain $\Xizwei$ by extending $\Phi$ to the entire patch as the identity.
\end{proof}

\begin{remark}
    We notice that, similar to the mapping $\Xieins$, the mapping $\Xizwei$ is not only bi-Lipschitz onto its image but also piecewise affine with respect to some essentially non-overlapping decomposition of its respective domains into simplices.
\end{remark}

We are now ready to describe alternative estimates for potential operators that rely on Proposition~\ref{proposition:deformation}.
Let $\calT$ be a shellable $n$-dimensional triangulation, and let the domain $\Omega \subseteq \bbR^{n}$ be the interior of the underlying set of $\calT$.

Let $T_0, T_1, \dots, T_M$ be a shelling of $\calT$.
For each simplex $T_1, \dots, T_M$ there exists a subsimplex $S_1 \subseteq T_1, \dots, S_M \subseteq T_M$ such that adding $T_m$ to the intermediate triangulation ($1 \leq m \leq M$) completes the local star $\patch_{\calT}(S_{m})$ around $S_m$.
Proposition~\ref{proposition:deformation} shows that for each index, there exists a bi-Lipschitz contraction which maps $\patch_{\calT}(S)$ onto $\overline{ \patch_{\calT}(S_{m}) \setminus S_m }$ and is the identity on all other simplices joined so far.

Evidently, the composition of these bi-Lipschitz contractions contracts $\Omega$ onto the interior of $T_{0}$.
Every shellable triangulation can be transformed into a single simplex along a sequence of local bi-Lipschitz (in fact, piecewise linear) deformations, and their bi-Lipschitz constants are controlled by the shape regularity of the triangulation.

Suppose now we want to find a potential for some differential form of degree $k+1$ over the triangulated domain $\Omega$.
Using successive pullbacks along the local contractions,
the original potential problem over the domain is reduced to a potential problem for some right-hand side over the first simplex of the shelling.
The potential over that first simplex, a $k$-form, is then extended along successive reverse pullbacks to the whole domain.
Since pullbacks commute with the exterior derivative, we thus obtain a solution to the original potential problem.

Potential estimates over the first single simplex, which is convex, are well-established.
Together with norm estimates for the pullback of forms (Proposition~\ref{proposition:pullbackestimate}),
we derive estimates for the solution.
We summarize this in the following theorem.

\begin{theorem}
    Let $\calT$ be a shellable $n$-dimensional triangulation, and let the domain $\Omega \subseteq \bbR^{n}$ be the interior of the underlying set of $\calT$.
    Let $T_0, T_1, \dots, T_M$ be a shelling of $\calT$
    and let $0 \leq \ell_{m} < n$, for $1 \leq m \leq M$, be such that $T_{m}$ has $n - \ell_{m}$ faces in common with the previous simplices.
    Then for any $u \in W^{p}\Alt^{k}(\Omega)$, where $1 \leq p \leq \infty$,
    there exists $w \in W^{p}\Alt^{k}(\Omega)$ with $\cartan w = \cartan u$
    such that
    \begin{align*}
        \| w \|_{L^{p}(\Omega)}
        \leq
\left(
        \prod_{m=1}^{M}
            \Cachttrafo{n}{k+1}{\ell_{m}}{p}(\calT)
            \Csiebtrafo{n}{k  }{\ell_{m}}{p}(\calT)
        \right)
        C_{{\PF},T_{0},k,p}
        \| \cartan u \|_{L^{p}(\Omega)}
        .
    \end{align*}
\end{theorem}

However, in our numerical examples (see Section~\ref{section:numericalexamples}),
the estimate performs significantly worse than the estimate derived in Theorem~\ref{theorem:poincarefriedrichsestimate:exterior},
which is why we do not recommend this method for practical computations.

\end{document}